\documentclass[11pt,reqno]{amsart}

\usepackage[T1]{fontenc}
\usepackage[utf8]{inputenc}
\usepackage{lmodern}
\usepackage{amsmath,amssymb,amsfonts,amsthm,mathtools}
\usepackage{enumitem,array}
\usepackage{microtype}
\usepackage{xcolor}
\usepackage[bookmarks=true,colorlinks=true,linkcolor=blue!55!black,citecolor=blue!55!black,urlcolor=blue!55!black]{hyperref}
\hypersetup{
  pdftitle={From rough local mass to logarithmic spectral ladders for damped waves},
  pdfauthor={Henry Shin},
  pdfsubject={Exact local-mass inverse theorems and logarithmic spectral realization for damped waves},
  pdfkeywords={damped waves, rough damping, local mass, inverse resolvent estimates, non-selfadjoint spectrum, logarithmic oscillator}
}
\allowdisplaybreaks
\newtheorem{thmx}{Theorem}

\newtheorem{theorem}{Theorem}[section]
\newtheorem{proposition}[theorem]{Proposition}
\newtheorem{lemma}[theorem]{Lemma}
\newtheorem{corollary}[theorem]{Corollary}
\newtheorem{definition}[theorem]{Definition}
\theoremstyle{remark}
\newtheorem{remark}[theorem]{Remark}

\newcommand{\T}{\mathbb T}
\newcommand{\R}{\mathbb R}
\newcommand{\N}{\mathbb N}
\newcommand{\Nzero}{\mathbb N_0}
\newcommand{\Z}{\mathbb Z}
\newcommand{\C}{\mathbb C}
\newcommand{\one}{\mathbf 1}
\newcommand{\eps}{\varepsilon}
\newcommand{\cA}{\mathcal A}
\newcommand{\cH}{\mathcal H}
\newcommand{\cL}{\mathcal L}
\newcommand{\cT}{\mathcal T}
\newcommand{\cS}{\mathcal S}
\newcommand{\cU}{\mathcal U}

\newcommand{\vertiii}[1]{{\left\vert\kern-0.25ex\left\vert\kern-0.25ex\left\vert #1 \right\vert\kern-0.25ex\right\vert\kern-0.25ex\right\vert}}

\newcommand{\doilink}[1]{\href{https://doi.org/#1}{\nolinkurl{doi:#1}}}
\DeclareMathOperator{\spec}{spec}
\DeclareMathOperator{\dist}{dist}
\DeclareMathOperator{\Ran}{Ran}
\DeclareMathOperator{\Log}{Log}

\numberwithin{equation}{section}

\title[Rough local mass and logarithmic spectral ladders]{From rough local mass to logarithmic spectral ladders for damped waves}
\author{Henry Shin}
\address{San Diego, CA, USA}
\email{hkshin@gmail.com}
\date{July 2026}


\begin{document}

\begin{abstract}
We prove an exact inverse theorem for damped waves on $X=\T\times Y$
with arbitrary bounded measurable transverse damping $a=a(x)\ge0$.
Let $\Theta_a(r)$ be the least average of $a$ on radius-$r$ intervals.
Then $\liminf_{r\downarrow0}\Theta_a(r)>0$ characterizes exponential
stability, while for every $L(S)=\ell(\log(eS))$ with
$\ell\colon[1,\infty)\to[1,\infty)$ nondecreasing, unbounded, and eventually
doubling,
$\Theta_a(r)\gtrsim L(1/r)^{-1}$ is equivalent to stationary- and
generator-resolvent bounds $O(|s|^{-1}L(|s|))$ and $O(L(|s|))$.
One resonant scalar block per dyadic octave already recovers this
multiscale mass bound.  Every such critical profile occurs for indicator
damping on an open dense set of arbitrarily small measure, whereas the
stationary resolvent-to-mass implication fails at every sublinear power gauge.

For the cusp $a(x)=(\log(e/|x|))^{-A}$, $A>0$, near its isolated damping
well, we prove the sharp obstruction is genuine spectrum.  Blow-up yields
$-\partial_y^2+i\log|y|$, whose spectrum is a simple interlaced vertical
ladder.  Every fixed finite ladder portion transfers to damped-wave eigenvalues
with complete blockwise algebraic count and two-term asymptotics.
Combined with the resolvent theorem, these eigenvalues give the
sharp regularized decay scale $\exp[-t^{1/(A+1)}]$.
\end{abstract}

\maketitle

\section{Introduction}\label{sec:intro}

Decay of damped waves is usually organized by two pieces of geometric
information: which rays meet the damping and how rapidly the coefficient
vanishes near those that do not.  For continuous or regularly structured
coefficients, a vanishing order and a microlocal normal form often supply
the second datum.  Neither object need exist for a general bounded measurable
damping.  There may be no regular boundary of the damped region and no
pointwise profile from which a decay rate can be read.

For general bounded measurable damping, the problem is therefore twofold:
identify the measurable feature of the coefficient encoded by
high-frequency dynamics, and determine whether the resulting sharp
nonselfadjoint resolvent scale is carried by eigenvalues or only by
pseudospectrum.  The latter distinction is essential, since a large
resolvent need not place any spectrum near the imaginary axis.

In the separated setting studied here, two complementary results resolve
these problems.  On
$X=\T_x\times Y$, with arbitrary $0\le a=a(x)\in L^\infty(\T)$, define
the lower interval mass
\begin{equation}\label{eq:Theta}
    \Theta_a(r)=\inf_{x\in\T}\frac1{2r}
       \int_{x-r}^{x+r}a(y)\,dy,
    \qquad 0<r\le\pi.
\end{equation}
Write $\cA$ for the associated damped-wave generator on the projected
energy space and $P_X(is)=-\Delta_X-s^2+isa(x)$ for the stationary operator;
their precise domains and zero-mode conventions are given below.
Theorem~\ref{thm:B} proves that the critical lower bounds for this single
coefficient-level quantity are equivalent to the corresponding stationary-
and generator-resolvent bounds.  Its inverse direction is strikingly
sparse: one resonant scalar block at one transverse frequency in each
dyadic octave recovers the mass bound at every sufficiently small scale.
Theorem~\ref{thm:E} then shows, for the canonical logarithmic cusp, that the
scale selected by this rough theorem is attained by genuine damped-wave
eigenvalues, with two-term locations and the complete algebraic count in
every fixed block window.

The distinction from existing results is both the inverse direction and
the spectral realization.  At the exponential endpoint, geometric control
and its Gaussian-beam converse give the classical picture for regular
damping \cite{Ralston1969,RT,BLR,Lebeau1996}.  Beyond that endpoint, sharp
rates have generally been deduced from a prescribed vanishing law,
derivative bounds, or controlled geometry of the damped and undamped sets
\cite{ALN,Kleinhenz2019,DK2020,Sun2023,Kleinhenz2025,DKP,AK,
KleinhenzWang2026}.  For rough coefficients, Burq--Moyano prove a universal
logarithmic regularized-decay estimate for every nontrivial bounded
measurable damping on a compact manifold \cite{MoyanoBurq2023}, while
L\'eautaud shows that H\"older damping failing geometric control forces at
least logarithmic resolvent divergence along a high-frequency sequence
\cite{Leautaud2025}.  None of these results gives a coefficient-level
quantity that is both sufficient for and recoverable from the dynamics,
scale by scale, without continuity or a pointwise profile.

The closest direct comparison for the cusp is
\cite[Theorem~2.2\textup{(3)}]{Kleinhenz2025}.  For the same profile
\[
    a(x)\asymp(\log(1/|x|))^{-A},
\]
that work obtains the generator-resolvent upper bound
$\|(is-\cA)^{-1}\|\lesssim(\log|s|)^A$ and the upper regularized-decay
envelope $\exp(-ct^{1/(A+1)})$, but only the weaker lower obstruction with
exponent $1/A$.  Theorem~\ref{thm:E} closes this gap for the exact cusp:
it constructs true eigenvalues at the predicted logarithmic depth and
thereby gives the matching exponent $1/(A+1)$.  Theorem~\ref{thm:B} is
complementary: it removes derivative and pointwise-profile assumptions and
gives a direct-and-inverse classification for arbitrary bounded measurable
transverse damping.

The unity of the two results is quantitative.  At the logarithmic gauge,
Theorem~\ref{thm:B} gives the global rough dictionary; for the cusp,
upper-mass saturation and Theorem~\ref{thm:E} add sharp sampled growth and
genuine spectrum.  In the sampled statements below, $s_j\to\infty$ ranges
over the distinct positive transverse frequencies; in the final line,
$z_j$ denotes the corresponding principal-rung spectral point constructed
in Theorem~\ref{thm:E}:
\begin{equation}\label{eq:spine}
\begin{aligned}
 \Theta_a(r)\gtrsim(\log(1/r))^{-A}\quad(r\downarrow0)
 &\ \Longleftrightarrow\
 \|P_X(is)^{-1}\|\lesssim |s|^{-1}(\log|s|)^A\\
 &\ \Longleftrightarrow\
 \|(is-\cA)^{-1}\|\lesssim(\log|s|)^A,
 \qquad |s|\gg1,\\
 &\ \Longrightarrow\
 \|e^{t\cA}\cA^{-1}\|\lesssim e^{-ct^{1/(A+1)}},\\
 \Theta_a(r)\asymp(\log(1/r))^{-A}\quad(r\downarrow0)
 &\ \Longrightarrow\
 \|(is_j-\cA)^{-1}\|\asymp(\log s_j)^A,\\
 \begin{gathered}
  z_j\in\spec(\cA),\qquad \Im z_j\sim s_j,\\[-2pt]
  -\Re z_j\asymp(\log s_j)^{-A}
 \end{gathered}
 &\ \Longrightarrow\
 \|e^{t\cA}\cA^{-1}\|\gtrsim e^{-Ct^{1/(A+1)}}.
\end{aligned}
\end{equation}
The first two lines are the global direct-and-inverse stationary/generator
dictionary.  Their upper resolvent bound gives the third line specifically by
the Batty--Duyckaerts $M_{\log}$ theorem and its Hilbert-space formulation
\cite{BD,RST}.  This implication belongs to the broader resolvent--semigroup
framework whose exponential endpoint is the Gearhart--Pr\"uss--Huang theorem
and whose polynomial Hilbert-space theory includes Borichev--Tomilov
\cite{Gearhart,Pruss,Huang,BT}.  The fourth line is the forward
sharpness statement supplied by upper-mass saturation; it does not assert
that a lower resolvent estimate recovers upper-mass saturation.  Resolvent
growth alone cannot furnish the final lower decay bound, because it may be
purely pseudospectral.  The genuine spectral points in
Theorem~\ref{thm:E} supply exactly the missing lower obstruction, completing
the sharp decay law.
Theorem~\ref{thm:A} supplies the exponential endpoint;
Theorem~\ref{thm:C} proves abundance of critical profiles and failure of
the inverse implication at power scales; and Theorem~\ref{thm:D} provides
the exact model identity behind the spectral transfer.

\subsection{Setting and the lower-mass invariant}\label{subsec:intro-setup}

Let $Y$ be a compact connected Riemannian manifold without boundary,
$\dim Y\ge1$, and put
\[
    X=\T_x\times Y,\qquad \T=\R/2\pi\Z.
\]
The product carries the metric $dx^2+g_Y$, and
$\Delta_X=\partial_x^2+\Delta_Y$ is nonpositive.  For
$0\le a=a(x)\in L^\infty(\T)$, consider
\begin{equation}\label{eq:dwe}
    \partial_t^2u-\Delta_Xu+a(x)\partial_tu=0.
\end{equation}
On the projected energy space
$\cH=\dot H^1(X)\times L^2(X)$, where
$\dot H^k(X)=H^k(X)/\C$, we use the homogeneous quotient Hilbert
structure
\[
 \|[u]\|_{\dot H^1(X)}:=\|\nabla u\|_{L^2(X)},\qquad
 \langle [u],[w]\rangle_{\dot H^1(X)}
   :=\langle\nabla u,\nabla w\rangle_{L^2(X)},
\]
and hence
\[
 \langle([u],v),([w],z)\rangle_{\cH}
 :=\langle\nabla u,\nabla w\rangle_{L^2(X)}
   +\langle v,z\rangle_{L^2(X)}.
\]
These expressions are representative-independent and define Hilbert
norms because $X$ is compact and connected.  Define, for every real
$d\in L^\infty(\T)$,
\[
    \cA_d([u],v)=([v],\Delta_Xu-dv),\qquad
    \mathcal D(\cA_d)=\dot H^2(X)\times H^1(X).
\]
For the given damping we abbreviate $\cA=\cA_a$; this operator generates a
contraction semigroup.  For a real coefficient $d\in L^\infty(\T)$, put
\begin{equation}\label{eq:P-def}
    P_{X,d}(is)=-\Delta_X-s^2+isd(x),\qquad
    \mathcal D(P_{X,d}(is))=H^2(X),
\end{equation}
and abbreviate $P_X(is):=P_{X,a}(is)$.
If $0=\lambda_0<\lambda_1\le\lambda_2\le\cdots$ are the eigenvalues of
$-\Delta_Y$ and $s_j=\sqrt{\lambda_j}$, then, specializing to $d=a$,
separation of variables reduces
\eqref{eq:P-def} to
\begin{equation}\label{eq:Q-family}
    Q_{s,\mu}=-\partial_x^2-\mu+isa(x),\qquad \mu\le s^2.
\end{equation}
Write
\[
    \Sigma_Y
      :=\{\sqrt\lambda:\lambda\in\spec(-\Delta_Y),\ \lambda>0\}
\]
for the set of distinct positive transverse frequencies.
Write $\cA_{d,j}$ for the restriction of $\cA_d$ to the reducing subspace
generated by a fixed transverse eigenfunction with eigenvalue $\lambda_j$,
and put $\cA_j=\cA_{a,j}$.  At a nonzero spectral parameter this scalar
block is governed by
\[
    P_{d,j}(z)=-\partial_x^2+\lambda_j+zd+z^2
    \quad\text{on }H^2(\T),
\]
with $d=a$ unless a perturbation is explicitly indicated.  The precise
block and zero-mode conventions are recorded in Lemma~\ref{lem:blocks}.

The geometry behind this reduction is structural.  The leaves
$\{x\}\times Y$ are geodesically invariant, the damping is constant along
them, and the Laplacian is the tensor sum
\[
    -\Delta_X=-\partial_x^2\otimes I+I\otimes(-\Delta_Y).
\]
Exact invariance along the leaves therefore replaces regularity of the
coefficient, while the one-dimensional rough direction leaves only a
uniformly bounded propagating Fourier cluster.  In particular, taking
$Y=\T^{d-1}$ shows that the results include arbitrary bounded measurable
unidirectional damping $a=a(x_1)$ on every flat torus $\T^d$, $d\ge2$.  After the
exact block decomposition, the only quantitative spectral information
used from $Y$ is the dyadic nonlacunarity of its distinct positive
frequencies.

The invariant $\Theta_a$ was defined in \eqref{eq:Theta}, where the
integral is taken over the corresponding torus arc.  For
$a=\one_\Omega$ we write $\vartheta_\Omega=\Theta_{\one_\Omega}$.
The relevant slowly varying scales are the following.

\begin{definition}[Critical gauges]\label{def:critical}
A gauge $L:[S_0,\infty)\to[1,\infty)$, $S_0\ge1$, is \emph{critical} if
\[
    L(S)=\ell(\log(eS)),
\]
where $\ell:[1,\infty)\to[1,\infty)$ is nondecreasing and unbounded, and
there are specified constants $u_0\ge1$ and $C_\ell\ge1$ such that
\[
    \ell(2u)\le C_\ell\ell(u),\qquad u\ge u_0.
\]
The gauge data are the quadruple
$(S_0,u_0,C_\ell,\ell(u_0))$; enlarging $u_0$ changes only these data.
Thus all references below to dependence on the full gauge data include both
the eventual-doubling threshold and the normalization at that threshold.
Typical examples are
\[
    L(S)=(\log(eS))^A,\qquad
    L(S)=(\log\log(e^eS))^p,
\]
and finite products of positive powers of iterated logarithms.  We extend
$L$ constantly below $S_0$.
\end{definition}
A doubling $\ell$ has polynomial upper growth, so
$L(S)=o(S^\varepsilon)$ for every $\varepsilon>0$.  More importantly,
evaluation is stable under fixed power-type changes of scale:
$L(CS^\alpha L(S)^\gamma)\asymp L(S)$; see
Lemma~\ref{lem:scale-stability}.  That stability is exactly what closes the
inverse argument.

\subsection{The rough inverse theorem}

Theorems~\ref{thm:A} and~\ref{thm:B} identify $\Theta_a$ as the
coefficient-level invariant at the endpoint and throughout the critical
class.  Theorem~\ref{thm:B} is the central inverse result in the latter
regime, while Theorem~\ref{thm:C} establishes abundance of critical profiles
and the failure of the inverse implication at power scales.

\begin{thmx}[Endpoint classification]\label{thm:A}
Let $0\le a\in L^\infty(\T)$.  The following are equivalent:
\begin{enumerate}[label=\textup{(\roman*)}]
\item $\displaystyle \liminf_{r\downarrow0}\Theta_a(r)>0$;
\item $\|P_X(is)^{-1}\|_{L^2\to L^2}\lesssim |s|^{-1}$ for $|s|\ge1$;
\item $i\R\subset\rho(\cA)$ and
$\sup_{s\in\R}\|(is-\cA)^{-1}\|_{\cH\to\cH}<\infty$;
\item $e^{t\cA}$ is exponentially stable.
\end{enumerate}
If $a=\one_\Omega$ with $\Omega\subset\T$ measurable, these conditions
hold if and only if $|\T\setminus\Omega|=0$.
\end{thmx}

The near-exponential regime is substantially richer.  A set
$\cS\subset(0,\infty)$ is \emph{dyadically syndetic} if it meets every
sufficiently high block $[2^k,2^{k+1})$; throughout the paper such sets are
subsets of the distinct positive transverse eigenfrequencies.

\begin{thmx}[Critical local-mass dictionary]\label{thm:B}
Let $L$ be critical and let $0\le a\in L^\infty(\T)$ be nontrivial.  The
following are equivalent:
\begin{enumerate}[label=\textup{(\roman*)}]
\item there exist $c>0$ and $r_{\mathrm{mass}}\in(0,\pi]$ such that
\begin{equation}\label{eq:crit-mass}
    \Theta_a(r)\ge cL(1/r)^{-1},\qquad 0<r\le r_{\mathrm{mass}};
\end{equation}
\item there exist $C,s_0>0$ such that
\begin{equation}\label{eq:crit-stat}
    \|P_X(is)^{-1}\|_{L^2\to L^2}
       \le C|s|^{-1}L(|s|),\qquad |s|\ge s_0;
\end{equation}
\item there exist $C,s_0>0$ and a dyadically syndetic set
$\cS\subset\Sigma_Y\cap[s_0,\infty)$ such that, for every $s\in\cS$, the
single resonant block $Q_{s,0}=-\partial_x^2+isa$ is invertible and
\begin{equation}\label{eq:crit-sparse-block}
    \|Q_{s,0}^{-1}\|_{L^2(\T)\to L^2(\T)}
       \le Cs^{-1}L(s);
\end{equation}
\item there exist $C,s_0>0$ such that
\begin{equation}\label{eq:crit-gen}
    \|(is-\cA)^{-1}\|_{\cH\to\cH}
       \le CL(|s|),\qquad |s|\ge s_0.
\end{equation}
\end{enumerate}
If, in addition, there exist $C_+>0$ and $r_1>0$ such that
\begin{equation}\label{eq:crit-upper-saturation}
    \Theta_a(r)\le C_+L(1/r)^{-1},
    \qquad 0<r\le\min(r_1,\pi),
\end{equation}
then, along every sufficiently large transverse eigenfrequency,
\begin{equation}\label{eq:sharp-transverse}
    \|P_X(is_j)^{-1}\|\asymp s_j^{-1}L(s_j),\qquad
    \|(is_j-\cA)^{-1}\|\asymp L(s_j).
\end{equation}
\end{thmx}

The decisive implication is \textup{(iii)}$\Rightarrow$\textup{(i)}: the
resolvent of one scalar resonant channel at one transverse frequency per
octave recovers the full multiscale lower-mass bound.  No estimate for any
nonresonant block is assumed.  Thus $\Theta_a$ is not merely a sufficient
thickness condition; it is a sparsely observable dynamical invariant.  No
continuity, monotonicity, vanishing order, or derivative control is used.

\begin{corollary}[Logarithmic dictionary]\label{cor:log}
Let $A>0$ and let $0\le a\in L^\infty(\T)$ be nontrivial.  The condition
\[
    \Theta_a(r)\gtrsim(\log(e/r))^{-A}\qquad(r\downarrow0)
\]
is equivalent to
\[
    \|P_X(is)^{-1}\|\lesssim |s|^{-1}(\log(e|s|))^A
\]
and to
\[
    \|(is-\cA)^{-1}\|\lesssim(\log(e|s|))^A
\]
at high frequency.  The inverse implication remains valid if there are
constants $C,s_0>0$ and a dyadically syndetic set
$\cS\subset\Sigma_Y\cap[s_0,\infty)$ such that $Q_{s,0}$ is invertible
for every $s\in\cS$ and
\[
    \|Q_{s,0}^{-1}\|\le
       Cs^{-1}(\log(es))^A,
    \qquad s\in\cS.
\]
The same conclusion follows if the full stationary estimate holds on such
a set with $P_X(is)$ invertible at every sampled frequency, or if the full
generator estimate holds there with $is\in\rho(\cA)$.
\end{corollary}

The next theorem shows both abundance and criticality of the dictionary.

\begin{definition}[Dyadic regularity]\label{def:dyadic-regular}
A function $\theta:(0,r_0]\to(0,1]$ is \emph{dyadically regular} if it is
nondecreasing, tends to zero at the origin, and there exist
$c_\theta\in(0,1]$ and $n_0\in\N$ such that, whenever $n\ge n_0$ and
$2^{-n}\le r_0$,
\[
    c_\theta\,\theta(2^{-n})
       \le \theta(2^{-(n+1)})\le \theta(2^{-n}).
\]
The constants $(c_\theta,n_0)$ are part of the dyadic-regularity data.
\end{definition}
For every critical $L$, the profile
$\theta_L:(0,1]\to(0,1]$, $\theta_L(r)=L(1/r)^{-1}$, is
dyadically regular: the arguments of $\ell$ at adjacent dyadic radii have
bounded ratio, so eventual doubling supplies the displayed lower ratio.

\begin{thmx}[Realization and breakdown]\label{thm:C}
\leavevmode
\begin{enumerate}[label=\textup{(\alph*)}]
\item \textup{(Exact rough profiles.)}  Let $L$ be critical.  For every
$\eta>0$ there is an open dense set $\Omega\subset\T$ with
\[
    |\Omega|<\eta,\qquad
    \T\setminus\Omega\ \text{closed, nowhere dense, and of positive measure},
\]
such that
\begin{equation}\label{eq:exact-profile-theorem}
    \vartheta_\Omega(r)\asymp L(1/r)^{-1},\qquad r\downarrow0.
\end{equation}
For $a=\one_\Omega$, the critical bounds in Theorem~\ref{thm:B} hold and
are attained along all sufficiently large transverse eigenfrequencies,
while exponential stabilization fails.
\item \textup{(Sublinear power gauges fail.)}  For every $0<\beta<1$ there
exist a nontrivial continuous damping $0\le a\in C(\T)$ and constants
$C,s_0>0$ such that $P_X(is)$ is invertible for every $|s|\ge s_0$ and
\[
 \|P_X(is)^{-1}\|_{L^2\to L^2}\le C|s|^{-1+\beta},
 \qquad |s|\ge s_0,
\]
but
\[
 \lim_{r\downarrow0}r^{-\beta}\Theta_a(r)=0.
\]
Thus this stationary bound does not imply
$\Theta_a(r)\gtrsim r^\beta$, even among continuous dampings.
More generally, if $0\le b\in L^\infty(\T)$ and
$\Theta_b(r)\asymp r^\gamma$ as $r\downarrow0$ for some $0<\gamma<2$,
then the corresponding generator $\cA_b$ satisfies, along all sufficiently
large transverse eigenfrequencies
\begin{equation}\label{eq:power-window}
    s_j^{\gamma/(2+\gamma)}
      \ \lesssim\ \|(is_j-\cA_b)^{-1}\|
      \ \lesssim\ s_j^{2\gamma/(2+\gamma)}.
\end{equation}
\end{enumerate}
\end{thmx}

The point of part \textup{(b)} is structural.  The inverse principle is
exact throughout the critical logarithmic/doubling class, while the
reverse resolvent-to-mass implication already fails for every sublinear
power gauge.

\subsection{Spectral realization at the logarithmic cusp}

Theorem~\ref{thm:E} is the main spectral result; its model input is the
exact spectral identity of Theorem~\ref{thm:D}.
Resolvent growth alone does not distinguish eigenvalues from
pseudospectrum \cite{TZ,Davies,DSZ}.  To decide whether the critical scale
is carried by true spectrum, write a block eigenvalue as $z=is+w$ and set
\begin{equation}\label{eq:Ts}
    T_s=-\partial_x^2+isa(x),\qquad \mathcal D(T_s)=H^2(\T).
\end{equation}
The exact block equation becomes
\begin{equation}\label{eq:pencil}
    -v''+(s_j^2+za+z^2)v=0,
\end{equation}
and, with $\zeta=-2isw$,
\[
    \left(T_s+\frac{i\zeta}{2s}a-\frac{\zeta^2}{4s^2}\right)v=\zeta v.
\]
Thus a low eigenvalue of $T_s$ must first be found and then transferred
through a genuinely nonnormal analytic pencil.  For the logarithmic cusp,
the scaling limit admits an exact spectral identity.
Proposition~\ref{prop:spectral-calibration}
makes the link with the rough inverse theorem quantitative: a critical
lower-mass bound
forces every eigenvalue with $\Re\lambda\le s$ above the wall
$\Im\lambda\gtrsim s/L(s)$.  Every nonreal block eigenvalue satisfies the
exact depth identity
\[
 -\Re z=\frac12\frac{\int_\T a|v|^2}{\|v\|_2^2}.
\]
In particular this identity holds for every eigenvalue in a transverse block
with $s_j>\|a\|_\infty/2$, which is the high-frequency regime used below.
Thus local mass fixes both the admissible spectral scale and the
structural factor $1/2$ in the wave depth.  The logarithmic ladder below
reaches that forced scale.

For a real-valued function $f$, we use the convention
$f_+:=\max(f,0)$ and $f_-:=\max(-f,0)$.  For $\gamma>0$ and
$\sharp\in\{N,D\}$, let $H_\gamma^\sharp$ be the
self-adjoint operator associated with the form
\[
 h_\gamma[u]=\int_0^\infty |u'|^2
       +\gamma\int_0^\infty(\log y)_+|u|^2
       -\gamma\int_0^\infty(\log y)_-|u|^2
\]
on $H^1(\R_+)\cap L^2((\log y)_+dy)$, with the additional trace condition
$u(0)=0$ when $\sharp=D$.  Let $\cL_\gamma^\sharp$ be the maximal
separated realization of $-\partial_y^2+i\gamma\log y$ with domain
\[
\begin{aligned}
 \mathcal D(\cL_\gamma^\sharp)=\{u\in L^2(\R_+):{}&
 u,u'\in AC_{\rm loc}(0,\infty)\ \text{and have finite traces at }0,\\
 &
 -u''+i\gamma\log y\,u\in L^2(\R_+),\\
 &u'(0)=0\ \text{if }\sharp=N,\quad
 u(0)=0\ \text{if }\sharp=D\}.
\end{aligned}
\]
and let $\cL_\gamma$ be the maximal distributional realization of
$-\partial_y^2+i\gamma\log|y|$ on $L^2(\R)$, namely
\[
 \mathcal D(\cL_\gamma)=\{u\in L^2(\R):
 -u''+i\gamma\log|y|\,u\in L^2(\R)
 \text{ in }\mathcal D'(\R)\}.
\]
The existence of the endpoint
traces and the equivalent form and maximal-domain descriptions are proved
in Section~\ref{sec:oscillator}.

\begin{thmx}[The imaginary logarithmic oscillator]\label{thm:D}
The operators $\cL_\gamma^\sharp$ and $\cL_\gamma$ are closed with compact
resolvent.  Let
$E^\sharp_{\gamma,0}<E^\sharp_{\gamma,1}<\cdots$ be the eigenvalues of
$H_\gamma^\sharp$, and choose a real $L^2(\R_+)$-normalized eigenfunction
$\phi^\sharp_{\gamma,n}$ for each $E^\sharp_{\gamma,n}$.  Its continuation
in the principal-logarithm sector is constructed in
Section~\ref{sec:oscillator}.  Then
\begin{equation}\label{eq:log-osc-spec-gamma}
    \spec(\cL_\gamma^\sharp)
      =\left\{\frac{\gamma\pi}{4}+iE^\sharp_{\gamma,n}:n\ge0\right\},
\end{equation}
every eigenvalue is algebraically simple, and its eigenfunction is the
boundary-rotated Sturm--Liouville mode
$u^\sharp_{\gamma,n}(y)=\phi^\sharp_{\gamma,n}(e^{i\pi/4}y)$.  Moreover
\begin{equation}\label{eq:log-osc-spec-scaled}
    E^\sharp_{\gamma,n}
       =\gamma\left(E^\sharp_{1,n}-\frac12\log\gamma\right).
\end{equation}
Under the even/odd decomposition, the full-line operator is unitarily
equivalent to $\cL_\gamma^N\oplus\cL_\gamma^D$.  The two half-line spectra strictly
interlace,
\[
    E^N_{\gamma,0}<E^D_{\gamma,0}<E^N_{\gamma,1}<E^D_{\gamma,1}<\cdots.
\]
For $m\in\Nzero$, set
\[
 \bar E_{\gamma,2m}:=E^N_{\gamma,m},\qquad
 \bar E_{\gamma,2m+1}:=E^D_{\gamma,m}.
\]
Then
\[
    \spec(\cL_\gamma)
       =\left\{\frac{\gamma\pi}{4}+i\bar E_{\gamma,n}:n\ge0\right\}
\]
is a simple parity-alternating vertical ladder.  For the same
$m\in\Nzero$, put
$\phi_{\gamma,2m}=\phi^N_{\gamma,m}$ and
$\phi_{\gamma,2m+1}=\phi^D_{\gamma,m}$, and let $u_{\gamma,2m}$ be the even
extension of $y\mapsto\phi_{\gamma,2m}(e^{i\pi/4}y)$ and
$u_{\gamma,2m+1}$ the odd extension of
$y\mapsto\phi_{\gamma,2m+1}(e^{i\pi/4}y)$.  With the normalized half-line
modes fixed above, every rung has nonzero bilinear
self-overlap:
\[
    \int_\R u_{\gamma,n}(y)^2\,dy
       =2e^{-i\pi/4}\int_0^\infty\phi_{\gamma,n}(r)^2\,dr
       =2e^{-i\pi/4}\ne0.
\]
For $\gamma=1$, set
$\bar E_n:=\bar E_{1,n}$, $\phi_n:=\phi_{1,n}$,
$u_n:=u_{1,n}$, $E_0^N:=E^N_{1,0}$,
$\nu_n=\pi/4+i\bar E_n$, and
$\nu_*=\nu_0=\pi/4+iE_0^N$; put $u_*:=u_0$.  In particular
\begin{equation}\label{eq:log-osc-overlap}
    \int_\R u_*(y)^2\,dy
       =2e^{-i\pi/4}\int_0^\infty
          \bigl(\phi^N_{1,0}(r)\bigr)^2\,dr
       =2e^{-i\pi/4}\ne0.
\end{equation}
\end{thmx}

\begin{remark}[A vertical spectral line]\label{rem:vertical}
The identity is unusually rigid: complex rotation does not merely deform
the logarithmic potential; it translates it by a constant.  Thus
\[
    \spec(\cL_\gamma^\sharp)
       =\frac{\gamma\pi}{4}+i\,\spec(H_\gamma^\sharp),
\]
and the real part $\gamma\pi/4$ becomes a universal spectral fingerprint
in the wave problem.
\end{remark}

\begin{thmx}[Logarithmic spectral ladder for damped waves]\label{thm:E}
Let $A>0$ and $0<x_0<1$.  Assume that $0\le a\in L^\infty(\T)$ has a
representative that is even on $(-x_0,x_0)$, is bounded below by a positive
constant on $\{|x|\ge x_0\}$, and satisfies
\begin{equation}\label{eq:log-cusp-main}
    a(x)=(\log(e/|x|))^{-A},\qquad 0<|x|<x_0.
\end{equation}
For $s>1/A$, let $R_s\in(0,1)$ be the unique solution of
\begin{equation}\label{eq:Rs-main}
    A sR_s^2\ell_s^{-A-1}=1,
    \qquad \ell_s:=\log(e/R_s),
\end{equation}
and let $\nu_n=\pi/4+i\bar E_n$ be the ladder of
Theorem~\ref{thm:D}.  For a transverse frequency $s_j$, put
$\ell_j:=\ell_{s_j}$.  Then $i\R\subset\rho(\cA)$ and the following hold.
\begin{enumerate}[label=\textup{(\alph*)},leftmargin=*]
\item Let $B\subset\C$ be a closed rectangle with
$\partial B\cap\spec(\cL_1)=\varnothing$.  For all sufficiently large
$s_j$, the spectrum of $\cA_j$ in
\begin{equation}\label{eq:ladder-window}
    \mathcal W_j(B)=
    \left\{is_j-\frac12\ell_j^{-A}
      +\frac{iA}{2}\ell_j^{-A-1}\beta:\ \beta\in B\right\}
\end{equation}
consists of exactly one algebraically simple eigenvalue $z_j^{(n)}$ for
each $\nu_n\in B$, and no others.  Rung by rung,
\begin{equation}\label{eq:ladder-asymptotic}
    z_j^{(n)}=is_j-\frac12\ell_j^{-A}
       +\frac{iA\nu_n}{2}\ell_j^{-A-1}
       +o(\ell_j^{-A-1}),
\end{equation}
so
\begin{equation}\label{eq:ladder-real-imag}
\begin{aligned}
 -\Re z_j^{(n)}
   &=\frac12\ell_j^{-A}
     +\frac{A\bar E_n}{2}\ell_j^{-A-1}+o(\ell_j^{-A-1}),\\
 \Im z_j^{(n)}-s_j
   &=\frac{A\pi}{8}\ell_j^{-A-1}+o(\ell_j^{-A-1}).
\end{aligned}
\end{equation}
\item In particular, for every sufficiently large positive transverse
frequency $s_j=\sqrt{\lambda_j}$, the principal rung
$z_j=z_j^{(0)}$ is algebraically simple as an eigenvalue of the scalar
block $\cA_j$ and satisfies
\begin{equation}\label{eq:log-resonance-asymptotic}
    z_j=is_j-\frac12\ell_j^{-A}
       +\frac{iA\nu_*}{2}\ell_j^{-A-1}
       +o(\ell_j^{-A-1}).
\end{equation}
Equivalently,
\begin{equation}\label{eq:log-real-imag}
    -\Re z_j=\frac12\ell_j^{-A}
       +\frac{AE_0^N}{2}\ell_j^{-A-1}+o(\ell_j^{-A-1}),
\end{equation}
\begin{equation}\label{eq:log-imag-shift}
    \Im z_j-s_j=\frac{A\pi}{8}\ell_j^{-A-1}
       +o(\ell_j^{-A-1}).
\end{equation}
Its algebraic multiplicity in the full generator is obtained by summing
the scalar-block multiplicities as in Lemma~\ref{lem:blocks}\textup{(iv)}:
the identical blocks attached to $\ker(-\Delta_Y-\lambda_j)$ contribute
$\dim\ker(-\Delta_Y-\lambda_j)$, and any accidental coincidence with a
different transverse block contributes additively.
Together with the global resolvent estimate, these eigenvalues give
\[
 c e^{-C t^{1/(A+1)}}
 \ \le\ \|e^{t\cA}\cA^{-1}\|
 \ \le\ C e^{-c t^{1/(A+1)}},\qquad t\ge1.
\]
\item The conclusions of \textup{(a)}--\textup{(b)} are stable under
fixed measurable lower-order relative perturbations.  More precisely,
let $\sigma>0$ and let $0\le\tilde a\in L^\infty(\T)$ have a
representative such that, for some $x_1\in(0,1)$ and $C_W>0$,
\[
 \tilde a(x)=(\log(e/|x|))^{-A}(1+w(x)),\qquad
 |w(x)|\le C_W(\log(e/|x|))^{-1-\sigma}
\]
for a.e.\ $0<|x|<x_1$, and assume
\[
  \operatorname*{ess\,inf}_{\{\rho\le|x|\}}\tilde a>0
  \qquad\text{for every }\rho\in(0,x_1].
\]
Then $i\R\subset\rho(\cA_{\tilde a})$, and \textup{(a)}--\textup{(b)}
hold with $\tilde a$ in place of $a$, with the same scales $R_s,\ell_s$,
the same limiting ladder, and the same two leading coefficients.  The
finite-$s$ eigenfunctions need not have parity unless $\tilde a$ is even.
\end{enumerate}
\end{thmx}

Part \textup{(a)} is the central spectral statement.  It is stronger than
a quasimode, a pseudospectral lower bound, or the construction of one
eigenvalue: every fixed model window transfers with its full algebraic
count, and nothing else enters the corresponding transverse block window.
Each resulting block eigenvalue embeds in the spectrum of the full
generator; repeated transverse multiplicities and accidental coincidences
between distinct blocks are governed by Lemma~\ref{lem:blocks}.  The
vertical model ladder becomes a horizontal ladder of decay rates beneath
each transverse frequency, while every fixed rung has frequency shift
$A\pi\ell_j^{-A-1}/8+o(\ell_j^{-A-1})$, with universal leading coefficient
$A\pi/8$.  Part \textup{(b)} closes the dynamical chain by matching the
principal rung with the global resolvent bound, and
part \textup{(c)} shows that neither conclusion depends on exact parity or
smoothness of the lower-order modulation.  The exact identity
\eqref{eq:wave-depth-identity} gives the leading real part a direct
interpretation: every rung has depth equal to one half of the damping mass
sampled by its eigenfunction, so the coefficient $1/2$ is structural.
The remainders $o(\ell_j^{-A-1})$ are qualitative.  Their non-effectivity
comes from compactness-based convergence of the rescaled Riesz projections
to the logarithmic oscillator; no convergence rate is assumed or claimed
under the bounded and measurable hypotheses used here.  The two displayed
coefficients and the finite-window algebraic count are unaffected.

\subsection{Geometric scope and limitations}

The decisive hypothesis is exact tensor-sum separation, allowing the twisted
boundary conditions induced by isometric holonomy, rather than global
product topology.  Section~\ref{sec:mapping-torus}
proves that Theorems~\ref{thm:A}--\ref{thm:C} and the logarithmic ladder
persist on every isometric mapping torus.  Holonomy changes the scalar
periodic boundary condition to a Floquet condition, but leaves the
estimates, the fixed-window count, and the two leading spectral
coefficients unchanged; no finite-order assumption on the holonomy is
needed.  The full statement is Theorem~\ref{thm:F}.

That exact separation is nevertheless essential to classification by
$\Theta_a$ alone.  For the warped metrics, with $|\eps|<1$,
\[
    g_\eps=dx^2+(2+\eps\sin x)^2dy^2
\]
on $\T^2$, Proposition~\ref{prop:warped-sharpness} and
Corollary~\ref{cor:warped-log-cusp} keep the damping and its lower interval
mass fixed, but change the logarithmic cusp from stretched-exponential
decay at $\eps=0$ to exponential stabilization for every
$0<|\eps|<1$.  Thus isometric holonomy preserves the scalar mechanism,
whereas arbitrarily small smooth warping can change it.  Damping
depending on both variables, and higher-dimensional rough bases, lie
outside that finite-cluster mechanism.

\subsection{Further comparisons with prior work}\label{subsec:literature}

Sharp and near-sharp decay estimates for trapped waves on tori, partially
rectangular domains, and product manifolds have been derived from the
geometry of the undamped set and from prescribed vanishing, growth, or
derivative conditions
\cite{BurqHitrik2007,ALN,Stahn2017,LL,BurqZuily2015,Kleinhenz2019,
DK2020,Sun2023,Kleinhenz2025,DKP,AK}.  In the polynomial regime,
Kleinhenz--Wang treat controlled unidirectional damping on the torus,
including unbounded polynomial singularities, and prove sharp
polynomial decay \cite{KleinhenzWang2026}; recent work also treats
non-polynomial derivative bounds and finer geometric effects
\cite{Kleinhenz2025,DKP,AK}.  Those results begin with a specified
pointwise profile or damped-set geometry and derive a decay estimate.  The
coefficient here is an arbitrary bounded measurable function, the datum is
its worst local average at every scale, and the conclusion is inverse as
well as direct.  The novelty is therefore not the appearance of a
logarithm by itself, but the exact mass--resolvent equivalence, its recovery
from one resonant channel per octave, and the counterexamples showing that
the reverse resolvent-to-mass implication fails for every sublinear power
gauge.

Rough damping also has a geometric stabilization theory on tori:
Burq--G\'erard establish stabilization results for polygonal dampings on
$\T^2$, and Rouveyrol develops higher-dimensional sufficient conditions
\cite{BG,Rouveyrol2024}.  Fixed-scale lower averages and
thick-set uncertainty principles form another related line.  Green obtains
fractional-wave decay from a uniform lower average on intervals of one
fixed length and proves the corresponding necessity for exponential decay
\cite{Green2020}; Green--Jaye--Mitkovski prove annular uncertainty
principles for geometrically controlling sets
\cite{GreenJayeMitkovski}.  Our upper bound uses this uncertainty-principle
philosophy in the finite propagating cluster, through Nazarov's measurable
Tur\'an estimate \cite{Nazarov}.  The converse is different: an absorbing
ODE quasimode converts a bound on one resonant scalar channel per octave
into lower mass at every sufficiently small scale.  This sparse dynamical
recovery is the inverse component absent from a spectral inequality alone.
It is also distinct from the classical inverse-spectral problem for a
regular one-dimensional damping, where full spectral or trace data are
used to recover or constrain the coefficient \cite{BorisovFreitas2009}:
here the data are sparse high-frequency resolvent samples, and the
recovered object is the rough multiscale invariant $\Theta_a$, not the
coefficient itself.

Rough Schr\"odinger observability provides a complementary fixed-scale
comparison.  Burq--Zworski prove observability on rectangular two-tori from
arbitrary measurable sets of positive measure \cite{BurqZworski2019}.
Burq--Zhu prove a general toral criterion, resolve the positive-measure
conjecture in one dimension, and obtain product classes of measurable
observation domains in arbitrary dimensions \cite{BurqZhu2025}.  These
results concern fixed observation sets; Theorem~\ref{thm:B} instead identifies a quantitative
shrinking-scale invariant and recovers it from sparse high-frequency
resolvent samples.

A direct damped-wave spectral precedent is Nonnenmacher's appendix to
Anantharaman--L\'eautaud \cite{ALN}, where characteristic strip damping
produces a near-axis branch at polynomial depth; Stahn later proves the
corresponding exact strip-decay rate on the Dirichlet square
\cite{Stahn2017}.  The present regime
has no open undamped strip: every punctured neighborhood of the trapped
slice carries positive damping.  Its spectral depth is logarithmic rather
than polynomial, and every fixed model window is transferred with its
complete blockwise algebraic count and two-term coefficients.

Classical damped-wave spectral theory describes high-frequency
bands, Weyl laws, and deviations controlled by geodesic averages of the
damping
\cite{Sjostrand2000,Hitrik2002,Hitrik2003,AschLebeau2003,
Anantharaman2010}.  Theorem~\ref{thm:E} is narrower geometrically and finer
near the axis: it identifies an exact singular model, transfers every
fixed portion of its spectrum with algebraic multiplicity, and excludes
additional eigenvalues from the corresponding wave window.  Hitrik's
elliptic-orbit construction produces eigenfrequencies converging rapidly
to the axis when the damping misses a neighborhood of the trapped orbit
\cite{Hitrik2003}.  Here every neighborhood of the trapped slice carries
positive damping mass, and the spectrum instead lies at the finite
logarithmic depth dictated by that mass.  The boundary deformation used
below belongs to the classical Aguilar--Combes/Balslev--Combes
complex-scaling lineage \cite{AguilarCombes,BalslevCombes}.  Its natural
nonnormal relatives include the rotated oscillator, the complex Airy
operator, and the semiclassical spectral theory of smooth purely imaginary
potentials \cite{Davies,DSZ,AlmogHenry2016,AlmogGrebenkovHelffer2019}.
Arnaiz--Bony--Michel analyze leftmost eigenvalues of purely imaginary
semiclassical Schr\"odinger operators through finite-order homogeneous
critical models \cite{ArnaizBonyMichel}.  The cusp here is neither $C^1$
at its zero nor homogeneous of finite order: complex dilation translates
the logarithmic potential and fixes its complete vertical spectrum.  The
remaining damped-wave step is a multiplicity-preserving transfer of that
spectrum through a nonnormal quadratic pencil.

\subsection{Proof architecture and organization}

The rough dictionary is driven by a two-part one-dimensional mechanism.
The upper branch combines a weighted interval Poincar\'e inequality with a
Nazarov--Tur\'an observability estimate for the propagating Fourier window;
the lower branch is an absorbing ODE quasimode on an interval of small
damping mass.  At a critical gauge, the two branches balance at
$\rho_s=s^{-1}L(s)^{1/2}$, and scale stability makes the inverse argument
close.  At power scales the balance is no longer rigid, producing the
breakdown in Theorem~\ref{thm:C}.

For the spectral argument, Proposition~\ref{prop:spectral-calibration}
first turns the rough mass bound into a spectral wall and identifies wave
depth with sampled damping mass.  The cusp scaling dictated by
\eqref{eq:Rs-main} produces the full-line operator
$\cL_1=-\partial_y^2+i\log|y|$.  Liouville--Green control and boundary
complex dilation solve the model; compactness and Riesz projection
convergence transfer its ladder to $T_s$; complex dilation supplies the
nonzero bilinear overlap, and complex symmetry uses it to exclude higher
geometric and algebraic multiplicity; finally the two Gohberg--Sigal
transfer propositions carry the isolated rungs and their total
finite-window count through the quadratic wave pencil.  The eigenvalues
and the global resolvent estimate then meet in the
envelope lemma to give sharp regularized decay.

Part~\ref{part:invariant} proves Theorems~\ref{thm:A}--\ref{thm:C}.
Part~\ref{part:model} develops the pencil transfer, solves the logarithmic
oscillator, proves Theorem~\ref{thm:E}, and establishes stability under
fixed measurable lower-order relative perturbations.  Its final section
uses the joint Floquet decomposition to prove the mapping-torus extension,
Theorem~\ref{thm:F}.  General notation and analytic inputs are collected at
the beginning of Part~\ref{part:invariant}.

\paragraph{Relation to the first preprint version.}
This article supersedes its
\href{https://arxiv.org/abs/2606.15218v1}{first arXiv version}.
Relative to that version, the rough inverse theorem is sharpened to sparse
recovery from scalar blocks, and the logarithmic oscillator,
multiplicity-preserving spectral realization, and stable perturbation theory
are added.  Together these results give a self-contained passage from
measurable mass to true spectrum and decay.

\clearpage
\tableofcontents
\part{The rough local-mass invariant}\label{part:invariant}

\section{Preliminaries and analytic inputs}\label{sec:prelim}

\subsection{Conventions}\label{subsec:notation}

Hilbert inner products are linear in the first argument.  Constants $c,C$
may change from line to line and depend on the quantitative hypotheses
currently in force.  For a fixed nontrivial damping they may depend on
$\|a\|_\infty$, a positive-mass scale, and the geometry of $Y$; for a
critical gauge they may also depend on its full gauge data.  We write
$A\lesssim B$, $A\gtrsim B$, and $A\asymp B$ in the usual way.  If $\Phi$
is positive and increasing,
\[
    F(t)\asymp_{\exp}e^{-\Phi(t)}
\]
means that $c e^{-C\Phi(Ct)}\le F(t)\le C e^{-c\Phi(ct)}$ for large $t$.

On $\T$ we use
\[
    \operatorname{dist}_\T(x,y)=\min_{k\in\Z}|x-y+2\pi k|,
    \qquad |x|=\operatorname{dist}_\T(x,0).
\]
An interval or arc is the image of a connected real interval, with periodic
Lebesgue measure.  We write $D(z,r)$ and $\overline D(z,r)$ for open and
closed discs, $K\Subset U$ for compact containment, and use essential
infima for measurable coefficients.  For a semibounded self-adjoint
operator with compact resolvent, $\lambda_k(B)$ denotes its eigenvalues in
nondecreasing order with multiplicity; $\lambda_j$ without an operator
argument is reserved for the transverse eigenvalues of $-\Delta_Y$.

Recall that a set $\cS\subset(0,\infty)$ is dyadically syndetic if it meets every
sufficiently high interval $[2^k,2^{k+1})$.  Dyadic syndeticity always
concerns distinct positive transverse frequencies.  The Weyl law implies
that every sufficiently high dyadic block contains such a frequency:
if $N_Y(\Lambda):=\#\{j:\lambda_j\le\Lambda\}$ is the counting function
with multiplicity and $d=\dim Y$, then
\[
 N_Y\!\left(\left(\tfrac32R\right)^2\right)-N_Y(R^2)
   \sim c_Y\bigl((\tfrac32)^d-1\bigr)R^d>0.
\]
Consequently, for every sufficiently large $R$, there is a distinct
positive transverse frequency in $(R,\tfrac32R]\subset[R,2R)$.
Once the exact tensor-product decomposition has been made, this dyadic
nonlacunarity is the only quantitative spectral input from $Y$ used
below.  A closed rectangle in $\C$ includes its boundary and has nonempty
interior.

\subsection{External analytic inputs used below}

We record the imported results in the forms required by the two main
arguments.  The inputs for the rough inverse theorem are Nazarov's
measurable Tur\'an estimate and the standard semigroup comparison.  The
spectral argument additionally uses the fixed-domain Gohberg--Sigal theorem, the
Batty--Duyckaerts resolvent-to-decay estimate, Olver's finite-endpoint
Liouville--Green theorem, and the rank-one comparison of separated
half-line extensions.  Whenever a displayed statement is an adapted
corollary, the reduction to the cited result is included here; the
logarithmic endpoint-at-infinity analysis needed for Theorem~\ref{thm:D}
is derived below rather than imported.

\begin{lemma}[Nazarov--Tur\'an estimate for exponential polynomials]\label{lem:nazarov-input}
Let $I\subset\R$ be a bounded interval, let $E\subset I$ be measurable with $|E|>0$, and let
\[
    p(x)=\sum_{m=1}^N c_m e^{i\xi_m x},
\]
where $N\ge1$, the coefficients $c_m\in\C\setminus\{0\}$, and the frequencies $\xi_m\in\R$ are distinct.  Equivalently, for an arbitrary finite representation one first combines equal frequencies and deletes zero combined coefficients, and $N$ denotes the number of remaining distinct terms.  There is an absolute constant $C_{\mathrm{TN}}\ge1$ such that
\begin{equation}\label{eq:nazarov-input}
    \sup_I |p|\le
       \Bigl(\frac{C_{\mathrm{TN}}|I|}{|E|}\Bigr)^{N-1}
       \operatorname*{ess\,sup}_{E}|p| .
\end{equation}
For $N=1$ the exponent is zero and the assertion is the constant-modulus
identity for a single exponential.  Nazarov's theorem is stated with the
pointwise supremum on $E$.  To obtain the stronger essential-supremum form
displayed above, let
\[
 E_\varepsilon
   =\{x\in E:|p(x)|\le
       \operatorname*{ess\,sup}_{E}|p|+\varepsilon\},
 \qquad \varepsilon>0.
\]
Then $E_\varepsilon$ is measurable, $|E_\varepsilon|=|E|$, and
$\sup_{E_\varepsilon}|p|\le
\operatorname*{ess\,sup}_{E}|p|+\varepsilon$.  Apply the pointwise theorem
to $E_\varepsilon$ and let $\varepsilon\downarrow0$.  Conversely, replacing
the essential supremum in \eqref{eq:nazarov-input} by the pointwise supremum
gives an immediate weaker form.  The constant is independent of the
interval length and of the locations or separation of the distinct
frequencies.  This is the measurable Tur\'an theorem of Nazarov
\cite[\S1, Theorem~1]{Nazarov}, after translation and affine rescaling of
the interval; the general exponential factor in Nazarov's statement is one
because the exponents $i\xi_m$ have zero real part.  In the only application
below, $I$ is a length-$2\pi$ representative of $\T$ and, after combining
equal Fourier modes, $N\le10$; hence the exponent is at most $9$.
\end{lemma}

We use the standard conventions of \cite{GohbergSigal} for a holomorphic family $F(z)$ of index-zero Fredholm operators: $z_0$ is a \emph{characteristic value} if $F(z_0)$ is not invertible; a \emph{root function} at $z_0$ is a holomorphic vector function $\phi$ with $\phi(z_0)\neq0$ and $F(z)\phi(z)$ vanishing at $z_0$, its \emph{order} being the order of that zero; and a \emph{Keldysh chain} of length $m$ at $z_0$ is a tuple $(u_0,\dots,u_{m-1})$ with $u_0\neq0$ and $\sum_{i=0}^{k}\tfrac1{i!}F^{(i)}(z_0)\,u_{k-i}=0$ for $0\le k\le m-1$ --- the condition that the coefficients of $F(z)\bigl(u_0+(z-z_0)u_1+\cdots+(z-z_0)^{m-1}u_{m-1}\bigr)$ vanish to order $m$ at $z_0$, so a chain of length $m$ is the same as a root function of order at least $m$ with leading vector $u_0$.  A \emph{canonical system} at $z_0$ consists of chains whose leading vectors form a basis of $\ker F(z_0)$, each of maximal length among chains whose leading vector lies outside the span of the preceding leading vectors; the \emph{Gohberg--Sigal multiplicity} of $z_0$ is the sum of the lengths of a canonical system, a quantity independent of the choices; it is the multiplicity appearing in the following lemma.

For later use, if $T$ is a closed operator and
$F(\zeta)=T-\zeta$ on the fixed domain $\mathcal D(T)$, then the Keldysh
chain condition at an isolated eigenvalue $\lambda$ is exactly
\[
 (T-\lambda)u_0=0,\qquad
 (T-\lambda)u_k=u_{k-1}\quad(1\le k<m).
\]
Thus Keldysh chains of the linear pencil are precisely Jordan chains of
$T$, with the same lengths, and the Gohberg--Sigal multiplicity of
$\lambda$ equals its algebraic multiplicity as an eigenvalue of $T$.

\begin{lemma}[Gohberg--Sigal operator Rouch\'e theorem, fixed-domain homotopy form]\label{lem:GS-input}
Let $\Omega\subset\C$ be a bounded domain with piecewise $C^1$ boundary,
and let $U\subset\C$ be a connected open set with
$\overline\Omega\Subset U$.  Let $X,Y$ be Banach spaces.  For
$0\le\tau\le1$, let
\[
    F_\tau\in\operatorname{Hol}\bigl(U,\mathcal L(X,Y)\bigr)
\]
be Fredholm of index zero at every $z\in U$, and assume that $(\tau,z)\mapsto F_\tau(z)$ is norm-continuous on $[0,1]\times\overline\Omega$.  If
\[
    F_\tau(z)\text{ is invertible for every }(\tau,z)\in[0,1]\times\partial\Omega,
\]
then, for each $\tau$, the characteristic values of $F_\tau$ in $\Omega$ are finite in number, counted with Gohberg--Sigal multiplicity, and their total multiplicity is independent of $\tau$.

The same conclusion holds for closed-operator families $A_\tau(z):D\to Y$ when the common domain $D$ is equipped with one fixed Banach norm, independent of $z$ and $\tau$, for which every $A_\tau(z)$ is bounded and the resulting family has the properties above with $X=D$.  Equivalently, one may precompose with any fixed Banach-space isomorphism $J:X_0\to D$; the families $A_\tau(z)$ and $A_\tau(z)J$ have corresponding root functions, Keldysh chains, orders, and multiplicities.  In the applications, $D=H^2(\T)$ carries the graph norm of a fixed reference operator.  No moving-domain version is used.
\end{lemma}

\begin{proof}[Reduction to the cited operator Rouch\'e theorem]
Fix $\tau_0\in[0,1]$.  Boundary invertibility and compactness give a
neighborhood of $\partial\Omega$ on which $F_{\tau_0}$ is invertible.
Since $U$ is connected and $F_{\tau_0}$ is a holomorphic Fredholm family
which is invertible at a boundary point, analytic Fredholm theory
\cite[Theorem~1.16, p.~15]{AKL} shows that its characteristic values are
discrete in $U$.  They are therefore finite in number in $\Omega$: after
removing the invertible boundary neighborhood, they lie in a compact subset
of $U$.  Write them as $z_1,\dots,z_m$.

If $m=0$, then $F_{\tau_0}$ is invertible on $\overline\Omega$; compactness
and norm-continuity imply that $F_\tau$ remains invertible there for all
$\tau$ in a neighborhood of $\tau_0$, so the local constancy assertion is
immediate.  Suppose $m\ge1$.  Choose pairwise disjoint closed discs
$\overline D_j\Subset\Omega$, centered at $z_j$, such that $z_j$ is the
only characteristic value of $F_{\tau_0}$ in $\overline D_j$, and set
\[
 K=\overline\Omega\setminus\bigcup_{j=1}^m D_j.
\]
The operator $F_{\tau_0}(z)$ is invertible for every $z\in K$, and hence
\[
 \sup_{z\in K}\|F_{\tau_0}(z)^{-1}\|<\infty.
\]
Norm-continuity gives a neighborhood $I_{\tau_0}$ of $\tau_0$ such that,
for $\tau\in I_{\tau_0}$,
\[
 \sup_{z\in K}
 \bigl\|F_{\tau_0}(z)^{-1}
       \bigl(F_\tau(z)-F_{\tau_0}(z)\bigr)\bigr\|<1.
\]
The Neumann-series criterion makes $F_\tau$ invertible on $K$, so every
characteristic value of $F_\tau$ in $\Omega$ lies in one of the discs
$D_j$.  On each simply connected disc, the operator Rouch\'e theorem of
Gohberg--Sigal \cite{GohbergSigal}, in the fixed-Banach-space formulation
\cite[Theorem~1.15, p.~14]{AKL}, gives equality of the total multiplicities of
$F_\tau$ and $F_{\tau_0}$.  Summing over $j$ proves local constancy of the
total multiplicity in the original, possibly multiply connected domain
$\Omega$.  It is therefore constant on the connected interval $[0,1]$.
The fixed-domain assertion follows by taking $X=D$; precomposition by a
fixed isomorphism carries each root function and Keldysh chain bijectively
to one of the same order.
\end{proof}

\begin{lemma}[Quantified resolvent-to-decay theorem]\label{lem:BD-RST-input}
Let $T(t)$ be a bounded $C_0$-semigroup on a Hilbert space with generator $A$, and assume $i\R\subset\rho(A)$.  Let $L$ be a critical gauge and let
\[
    M(S)=C_0\max\{L(S),M_0\},\qquad S\ge0,
\]
with arbitrary constants $C_0,M_0\ge1$, after any harmless nondecreasing modification on a compact interval.  If
\[
    \|(is-A)^{-1}\|\le M(|s|),\qquad s\in\R,
\]
then there are constants $c,C,t_0>0$, depending only on the semigroup bound and on the displayed majorant data, such that
\begin{equation}\label{eq:BD-RST-input}
    \|T(t)A^{-1}\|\le \frac{C}{M_{\log}^{-1}(ct)},
    \qquad t\ge t_0.
\end{equation}
Here
\[
    M_{\log}(S)=M(S)\bigl(\log(1+M(S))+\log(1+S)\bigr),
\]
and the generalized right inverse is defined for every $r\ge0$ by
\[
    M_{\log}^{-1}(r)
       :=\sup\Bigl(\{S\ge0:M_{\log}(S)\le r\}\cup\{0\}\Bigr).
\]
Since $M_{\log}$ is nondecreasing and tends to infinity, this supremum is finite; $t_0$ is chosen large enough that $M_{\log}^{-1}(ct)\ge1$ for $t\ge t_0$, so the denominator in \eqref{eq:BD-RST-input} is positive.

This follows from the Batty--Duyckaerts $M_{\log}$ theorem \cite[Theorem~1.5]{BD}, stated in the present Hilbert-space notation as \cite[Theorem~3.1]{RST}, and specialized to the critical majorants used in Lemma~\ref{lem:envelope}.  The only convention change is the right-inverse notation.  If the source statement is applied first to a continuous increasing majorant, use the following upper regularization.  For large dyadic $2^k$ define $M^\#$ on $[2^k,2^{k+1}]$ by linear interpolation between the shifted values $M(2^{k+1})$ and $M(2^{k+2})$, and connect it continuously to a large constant on the remaining compact interval.  Then $M^\#$ is continuous, nondecreasing, and for $2^k\le S\le2^{k+1}$,
\[
    M(S)\le M(2^{k+1})\le M^\#(S)\le M(2^{k+2})\le M(4S).
\]
Since $M$ is critical up to a fixed compact-frequency maximum, scale stability and doubling give $M(4S)\le C_\#M(S)$ for $S\ge S_*$, with $C_\#$ depending only on the relevant full gauge data and the compact-frequency constant.  Hence
\[
    M_{\log}(S)\le (M^\#)_{\log}(S)\le C_\# M_{\log}(4S),\qquad S\ge S_*.
\]
The generalized inverses are therefore comparable in the form needed below: for some constants $c,C>0$,
\[
    \bigl((M^\#)_{\log}\bigr)^{-1}(t)\ge c\,M_{\log}^{-1}(ct),
    \qquad
    \frac1{\bigl((M^\#)_{\log}\bigr)^{-1}(t)}\le \frac{C}{M_{\log}^{-1}(ct)}
\]
for all sufficiently large $t$.  Applying the continuous-majorant theorem to $M^\#$ gives \eqref{eq:BD-RST-input} for $M$ for $t\ge t_0$.  The shifted interpolation majorizes $M$ on each dyadic interval and is insensitive to jumps or flat portions of the original nondecreasing majorant.
\end{lemma}

\begin{lemma}[Finite-endpoint Olver theorem and logarithmic endpoint-at-infinity reduction]\label{lem:olver-input}
The external input is Olver's finite-endpoint complex
Liouville--Green theorem and its error-control equation
\cite[Ch.~6, Theorem~11.1 and \S\S 11.2, 13, pp.~220--227]{Olver}.  In the notation
below it gives the relative function and derivative estimates
\eqref{eq:olver-u}--\eqref{eq:olver-du}.  We then derive the
endpoint-at-infinity consequences for the logarithmic potentials
\eqref{eq:olver-log-q}.  Let $D\subset\C$ be simply connected, let $q$ be
holomorphic and zero-free on $D$, fix a branch of $q^{1/2}$ and a
holomorphic fourth root $q^{1/4}$ with $(q^{1/4})^2=q^{1/2}$, and put
$\xi(z)=\int^z q^{1/2}$.  If $z_{\rm ref}$ is a finite reference endpoint,
let $H(z_{\rm ref})\subset D$ be a connected open progressive domain: for
each $z\in H(z_{\rm ref})$ a chosen path $\Gamma_z\subset D$, consisting of
a finite chain of $R_2$-arcs (arcs admitting regular $C^2$
parametrizations), satisfies Olver's progressivity hypothesis and joins $z$
to $z_{\rm ref}$.  Every finite-endpoint path used below is a straight
radial segment; its admissibility is automatic, and its progressivity is
verified from the phase in the relevant application.  In the logarithmic applications
$H(\infty)$ denotes the sectorial tail whose
points are joined to infinity by the rays constructed in the proof; those
endpoint-at-infinity assertions are established directly by the Volterra
argument below rather than imported from Olver.  A path is
$\xi$-progressive for the recessive exponential when $\Re\xi$ is
nondecreasing toward the reference endpoint; for the dominant exponential
the direction is reversed.  For $u''=q(z)u$, Olver's exact error-control
differential is
\[
 dF_q(z)=\left(\frac{5q'(z)^2}{16q(z)^{5/2}}
              -\frac{q''(z)}{4q(z)^{3/2}}\right)dz.
\]
For estimates we use the absolute majorant
\[
    \omega_q(z)|dz|
    :=\left(\left|\frac{q'(z)^2}{q(z)^{5/2}}\right|+
       \left|\frac{q''(z)}{q(z)^{3/2}}\right|\right)|dz|,
\]
and define the majorizing error-control variation of the chosen family by
\[
    \mathcal V=\sup_{z\in H(z_{\rm ref})}
       \int_{\Gamma_z}\omega_q(t)|dt| .
\]
For finite reference endpoint $z_{\rm ref}$, if the paths are progressive
and $\mathcal V<\infty$, Olver's theorem gives on $H(z_{\rm ref})$ a
solution of $u''=q(z)u$ of the form
\begin{equation}\label{eq:olver-u}
    u(z)=q(z)^{-1/4}\exp(\mp\xi(z))(1+\eta_0(z)),
    \qquad
    \sup_{H(z_{\rm ref})} |\eta_0|
       \le C_0\bigl(e^{C_0\mathcal V}-1\bigr).
\end{equation}
The sign is chosen according to the progressive direction.  In the same notation the differentiated estimate is
\begin{equation}\label{eq:olver-du}
\begin{aligned}
    u'(z)&=\mp q(z)^{1/4}\exp(\mp\xi(z))(1+\eta_1(z)),\\
    \sup_{H(z_{\rm ref})}|\eta_1|&\le
       C_1\biggl(e^{C_1\mathcal V}-1
          +\sup_{H(z_{\rm ref})}\frac{|q'|}{|q|^{3/2}}\biggr).
\end{aligned}
\end{equation}
Here $C_0,C_1$ are absolute in the fixed-sector applications below.  Since
$|dF_q|\le\omega_q(z)|dz|$, Olver's source variation is bounded by
$\mathcal V$, so using this majorant only enlarges the error bound.  The
derivative of $q^{-1/4}$ is the only normalization change: it contributes
the relative term $O(|q'|/|q|^{3/2})$, while the differentiated Volterra
error is controlled by the same majorizing variation.

For the only potentials used below,
\begin{equation}\label{eq:olver-log-q}
    q(z)=e^{i\alpha}\gamma\Log z-E,
    \qquad
    \alpha\in\{0,\pi/2\},\qquad \gamma>0,
\end{equation}
with $E$ in a fixed compact subset of $\C$, the endpoint at infinity, compact-parameter uniformity, and the raywise dominant construction are consequences of the finite-endpoint theorem and the sectorial Volterra construction displayed in the proof below; no endpoint-at-infinity assertion for a general potential is used or claimed.  For the endpoint-at-infinity assertion, write the angular interval of the sectorial tail as $[\theta_-,\theta_+]$ and require
\begin{equation}\label{eq:olver-phase-condition}
    c_0:=\inf_{\theta\in[\theta_-,\theta_+]}
       \cos(\theta+\alpha/2)>0.
\end{equation}
The sign of the square-root branch is fixed by the uniform asymptotic normalization
\begin{equation}\label{eq:olver-sqrt-normalization}
    q(re^{i\theta})^{1/2}
      =e^{i\alpha/2}(\gamma\log r)^{1/2}
         \bigl(1+O((\log r)^{-1})\bigr),
    \qquad r\to\infty,
\end{equation}
on that angular interval.  The constants may depend on $\gamma$, $c_0$, and
the compact set for $E$; they are uniform if $\gamma$ is restricted to a
compact subinterval of $(0,\infty)$.  Thus all later uses of this input are
reduced to \eqref{eq:olver-u}--\eqref{eq:olver-du} on the logarithmic
sectors verified explicitly in Lemma~\ref{lem:sector-LG}; no other
complex-domain Liouville--Green variant is invoked, and the notation
$H(\infty)$ below refers only to those explicitly constructed logarithmic
sectorial tails satisfying
\eqref{eq:olver-phase-condition}--\eqref{eq:olver-sqrt-normalization}.
\end{lemma}

\begin{proof}[Reduction in the logarithmic sectors]
Only the reductions after Olver's finite-endpoint theorem are proved here.  Fix a compact set $K_E\subset\C$ for $E$, fix $\gamma>0$, and fix a closed logarithmic sector
\[
   \{re^{i\theta}:r\ge R_0,\ \theta\in[\theta_-,\theta_+]\}
\]
having a positive angular margin inside the principal branch domain of $\Log$ and satisfying \eqref{eq:olver-phase-condition}.  After slightly enlarging the angular interval while preserving a lower bound $c_0/2$, and then increasing $R_0$ if necessary, $q$ in \eqref{eq:olver-log-q} is zero-free on the resulting open sector.  The branches of $q^{1/2}$ and $q^{1/4}$ are fixed as in the statement.  Uniformly for $E\in K_E$ and for sectors with these fixed margins,
\begin{equation}\label{eq:olver-log-estimates}
\begin{gathered}
    |q(z)|\asymp_{\gamma,K_E}\log |z|,
    \qquad
    q'(z)=\frac{e^{i\alpha}\gamma}{z},
    \qquad
    q''(z)=-\frac{e^{i\alpha}\gamma}{z^2},\\
    \frac{|q'(z)|}{|q(z)|^{3/2}}
       \le \frac{C}{|z|(\log |z|)^{3/2}},\qquad
    \frac{|q'(z)|^2}{|q(z)|^{5/2}}
      +\frac{|q''(z)|}{|q(z)|^{3/2}}
       \le \frac{C}{|z|^2(\log |z|)^{3/2}} .
\end{gathered}
\end{equation}
For $z=re^{i\theta}$ choose $\Gamma_z$ to be the outward radial ray $\{\rho e^{i\theta}:\rho\ge r\}$.  After increasing $R_0$, these rays are recessive-progressive, as follows from the phase calculation below, and their error-control variation satisfies
\begin{equation}\label{eq:olver-tail-variation}
    \sup_{\substack{z\text{ in the sector}\\ |z|\ge R}}
       \int_{\Gamma_z}\omega_q(t)|dt|
       \le C\int_R^\infty r^{-2}(\log r)^{-3/2}\,dr,
\end{equation}
which is finite and tends to zero as $R\to\infty$, uniformly in $E\in K_E$ and in the sector within the fixed angular margin.  Only this chosen radial family, whose Euclidean arclength element is $d\rho$, is used in the variation estimate; no corresponding bound is asserted for arbitrary progressive paths of uncontrolled Euclidean length.  These estimates also imply
\[
    \sup_{|z|\ge R}\frac{|q'|}{|q|^{3/2}}\longrightarrow0
    \qquad(R\to\infty),
\]
again uniformly in the same parameters.  Therefore the relative errors in \eqref{eq:olver-u} and \eqref{eq:olver-du} can be made as small as required on tails, and they tend to zero along the chosen progressive approaches to infinity.

We now spell out the endpoint-at-infinity passage in the logarithmic case.  Let $D_R$ be a connected, simply connected open sectorial tail on which the above estimates hold and whose rays are recessive-progressive.  On $D_R$ set
\[
    \phi_-(z)=q(z)^{-1/4}e^{-\xi(z)},\qquad
    \mathcal E_q(z)=\frac{5q'(z)^2}{16q(z)^2}-\frac{q''(z)}{4q(z)} .
\]
Then $\phi_-''=(q+\mathcal E_q)\phi_-$.  Writing $u=\phi_-(1+h)$ and imposing the recessive normalization $h(\infty)=0$ gives
\[
    (\phi_-^2 h')'=-\mathcal E_q\phi_-^2(1+h).
\]
Equivalently, initially taking $\Gamma_z$ to be the chosen outward radial ray from $z$ to infinity in $D_R$,
\begin{equation}\label{eq:olver-volterra-tail}
    h(z)=-\int_{\Gamma_z}
      \frac{1-e^{-2(\xi(t)-\xi(z))}}{2q(t)^{1/2}}
      \mathcal E_q(t)(1+h(t))\,dt .
\end{equation}
The sign in \eqref{eq:olver-volterra-tail} is forced by the displayed differential identity.  Conversely, any bounded solution of \eqref{eq:olver-volterra-tail} gives an exact solution $u''=qu$ with the stated normalization.  The integral is path-independent within $D_R$ among paths for which the improper integral is defined.  Indeed, the integrand is holomorphic in $t$ for fixed $z$, and it remains only to justify the limit at the endpoint at infinity when two truncated progressive paths are connected by a large circular arc.  The phase hypothesis now gives the required uniform coercivity rather than merely raywise progressivity: uniformly on the angular interval,
\[
\begin{aligned}
 \frac{d}{d\rho}\Re\xi(\rho e^{i\theta})
   &=\Re\!\left(e^{i\theta}q(\rho e^{i\theta})^{1/2}\right)\\
   &=(\gamma\log\rho)^{1/2}
      \left(\cos(\theta+\alpha/2)
             +O((\log\rho)^{-1})\right).
\end{aligned}
\]
Thus, after increasing $R$, this derivative is at least $\frac{c_0}{4}(\gamma\log\rho)^{1/2}$.  Integration along the ray, with the additive normalization of $\xi$ absorbed into the constant, shows that there are constants $c,C>0$ such that
\[
    \Re \xi(\rho e^{i\theta})\ge c\rho(\log\rho)^{1/2}-C
\]
uniformly for all angles of the tail.  Hence, for fixed $z$, the factor $\exp(-2(\xi(t)-\xi(z)))$ is uniformly bounded, and in fact exponentially small, on the connecting arcs $|t|=\rho$ as $\rho\to\infty$.  Since
\[
    \frac{|\mathcal E_q(t)|}{|q(t)|^{1/2}}
       \le C |t|^{-2}(\log |t|)^{-3/2},
\]
the contribution of such an arc is
\[
    O\bigl(\rho^{-1}(\log\rho)^{-3/2}\bigr)+
    O\bigl(e^{-c\rho(\log\rho)^{1/2}}\bigr)=o(1),
\]
with constants depending on the fixed initial point $z$ and the sector.
Thus the improper Volterra integral is independent of the choice among
progressive paths for which it is defined.  We record explicitly why the
resulting function is holomorphic in the initial point, despite the radial
path depending on that point.  Fix $z_0\in D_R$ and a disc
$D_0\Subset D_R$ centered at $z_0$.  For $z\in D_0$, path independence
allows us to replace the radial path $\Gamma_z$ by the concatenation of a
path in $D_0$ from $z$ to $z_0$ and the fixed path $\Gamma_{z_0}$.  For a
bounded holomorphic input $h$, the integral over the finite first piece is
holomorphic in $z$ by the primitive theorem, since its kernel is jointly
holomorphic in $(z,t)$.  On the fixed tail $\Gamma_{z_0}$, the estimates
above give a locally uniform integrable majorant for $z\in D_0$; hence its
truncations are holomorphic in $z$ and converge locally uniformly as the
outer endpoint tends to infinity.  The Weierstrass theorem therefore shows
that the Volterra operator maps bounded holomorphic functions on $D_R$ to
bounded holomorphic functions.  Picard iterates starting from $0$ are
holomorphic, and their norm convergence gives a holomorphic fixed point.

On each chosen radial path, $\Re(\xi(t)-\xi(z))\ge0$ for $t$ farther out than $z$, so the kernel in \eqref{eq:olver-volterra-tail} is bounded by a constant times $|q(t)|^{-1/2}$.  Hence the operator norm on $H^\infty(D_R)$ is bounded by $C\mathcal V_R$, where
\[
    \mathcal V_R:=\sup_{z\in D_R}\int_{\Gamma_z}\omega_q(t)|dt|\longrightarrow0.
\]
For $R$ large this norm is $<1/2$, and the contraction equation has a unique bounded holomorphic solution with
\[
    \|h\|_{H^\infty(D_R)}\le C\mathcal V_R.
\]
The normalization is
\[
    u(z)q(z)^{1/4}e^{\xi(z)}=1+h(z)\longrightarrow1
    \qquad (z\to\infty\text{ along the chosen radial tails}),
\]
uniformly on compact subsectors of the tail.  On compact subsectors below the tail radius, the Cauchy data at the entrance circle determine the continuation uniquely; standard continuous dependence for the ODE gives $C^1_{\mathrm{loc}}$ convergence and holomorphy there.  Thus the construction produces one sectorial recessive solution, not merely a raywise family.

The derivative estimate is obtained from the first-order identity above.  Since
\[
    \phi_-^2(z)h'(z)=\int_{\Gamma_z}
        \mathcal E_q(t)\phi_-^2(t)(1+h(t))\,dt,
\]
the same progressive bound gives
\[
    |h'(z)|\le C|q(z)|^{1/2}\mathcal V_R
    \qquad (z\in D_R),
\]
after increasing $R$ if necessary.  Combining this with
\[
    \frac{\phi_-'(z)}{-q(z)^{1/2}\phi_-(z)}=1+O\!\left(\frac{|q'|}{|q|^{3/2}}\right)
\]
gives the differentiated estimate \eqref{eq:olver-du}.  The constants are uniform for $E$ in the prescribed compact set and for $\gamma$ in any fixed compact subinterval of $(0,\infty)$, because all lower bounds for $|q|$, all tail variations, and all progressive-path constants are uniform after increasing the lower radius.

For the dominant solution later used on a fixed ray no sector-wide dominant
branch is required or asserted.  Fix a ray $z=re^{i\theta}$, $r\ge y_1$,
along which $\Re\xi$ is increasing, and set $z_1=y_1e^{i\theta}$.  With
$\phi_+(z)=q(z)^{-1/4}e^{+\xi(z)}$, solve on this one ray the initial
Volterra equation
\begin{equation}\label{eq:olver-ray-dominant}
    h_-(z)=-\int_{z_1}^z
      \frac{1-e^{-2(\xi(z)-\xi(t))}}{2q(t)^{1/2}}
      \mathcal E_q(t)(1+h_-(t))\,dt,
\end{equation}
where the integral is along the ray.  This is the same
variation-of-constants identity for $u_-:=\phi_+(1+h_-)$ with
$h_-(z_1)=h_-'(z_1)=0$.  The kernel in
\eqref{eq:olver-ray-dominant} is bounded by $C|q(t)|^{-1/2}$ on the
reversed-progressive ray, and the ray variation is bounded by
\eqref{eq:olver-tail-variation}; for $y_1$ large the Volterra operator is a
contraction on bounded functions of $r\in[y_1,\infty)$.  Hence
\[
    \sup_{r\ge y_1}|h_-(re^{i\theta})|+
    \sup_{r\ge y_1}\frac{|h_-'(re^{i\theta})|}{|q(re^{i\theta})|^{1/2}}
       \le C\mathcal V_{y_1},
\]
so the relative function and derivative errors can be made arbitrarily small on the whole ray.  The Cauchy data of this ray solution at any point of the ray extend holomorphically to a small sector by the local ODE theorem.  This raywise construction is exactly what is used in Lemma~\ref{lem:sector-LG}.

Finally, the compatibility needed below is compatibility of the \emph{normalized} recessive construction, not uniqueness from one scalar value.  Suppose two constructions are defined on sectorial progressive domains with a common tail, use the same branches of $q^{1/2}$ and $q^{1/4}$ and the same additive normalization of $\xi$, and both satisfy
\[
    u(z)q(z)^{1/4}e^{\xi(z)}\longrightarrow1
\]
along the common progressive tail.  After restricting to a sufficiently far-out common tail, both are bounded fixed points of the same path-independent contraction equation \eqref{eq:olver-volterra-tail}; hence they coincide there.  Their Cauchy data consequently agree at every point of that tail, and ordinary Cauchy uniqueness for the second-order equation, followed by analytic continuation, propagates the equality throughout the connected component of their common domain.  More generally, equality of both value and derivative at one point suffices; equality of the scalar values alone does not.  This is precisely the uniqueness and compatibility used below.
\end{proof}

\begin{lemma}[Rank-one endpoint comparison for half-line Schr\"odinger extensions]\label{lem:rank-one-extension-input}
Let
\[
    \tau=-\frac{d^2}{dx^2}+V(x)
\]
on $(0,\infty)$, where $V\in L^1_{\mathrm{loc}}(0,\infty;\R)$ and $V\in L^1(0,1)$, so that $0$ is a regular endpoint, and suppose that $+\infty$ is limit point.  Let $H^N$ and $H^D$ be the self-adjoint realizations with boundary conditions $u'(0)=0$ and $u(0)=0$, respectively.  Assume that their standard quadratic forms are closed and bounded below, that
\[
    \mathcal Q(H^D)
       =\{u\in\mathcal Q(H^N):u(0)=0\}
       \subset\mathcal Q(H^N),
\]
and that both operators have compact resolvent.  Then for every real
\[
    \mu<\min\{\inf\spec H^N,\inf\spec H^D\}
\]
the resolvent difference
\[
    (H^N-\mu)^{-1}-(H^D-\mu)^{-1}
\]
has rank one.  If the eigenvalues are listed increasingly with multiplicity, then
\[
    E^N_n\le E^D_n\le E^N_{n+1},\qquad n\ge0,
\]
and equivalently the two eigenvalue counting functions differ by at most one at every real energy.  Strictness is not asserted here; in the application it is obtained separately from spectral disjointness.  This is the regular-left, limit-point-right Schr\"odinger specialization of Krein's resolvent formula for separated self-adjoint extensions; see, for instance, \cite[Sec.~7, pp.~110--125]{Weidmann}.
\end{lemma}

\begin{proof}
Fix $\mu$ below both spectra.  Let $\psi_\mu$ be a nonzero real solution of
\[
    -u''+(V-\mu)u=0
\]
which is square-integrable at $+\infty$.  The limit-point hypothesis makes its span the unique square-integrable solution space there.  Let $u_N,u_D$ be the real solutions normalized by
\[
    u_N(0)=1,\quad u_N'(0)=0,
    \qquad
    u_D(0)=0,\quad u_D'(0)=1.
\]
For $\sharp\in\{N,D\}$, the Wronskian $W(u_\sharp,\psi_\mu)$ is nonzero: otherwise a nonzero solution would satisfy the $\sharp$ boundary condition and be square-integrable at infinity, contradicting $\mu\in\rho(H^\sharp)$.  The Green kernel of $(H^\sharp-\mu)^{-1}$ therefore has the standard separated form
\[
    G^\sharp_\mu(x,y)
      =c_\sharp\,u_\sharp(x_<)\psi_\mu(x_>),
    \qquad
    x_<=\min(x,y),\quad x_>=\max(x,y),
\]
where the sign of the reciprocal-Wronskian constant $c_\sharp$ is fixed by the derivative jump for $-d^2/dx^2$.

The kernel $K_\mu:=G^N_\mu-G^D_\mu$ has no derivative jump across $x=y$ and hence solves the homogeneous equation distributionally in each variable.  For fixed $y$, it is square-integrable at $+\infty$ as a function of $x$, so the limit-point property gives $K_\mu(x,y)=c(y)\psi_\mu(x)$.  Self-adjoint symmetry of the real Green kernels then yields
\[
    K_\mu(x,y)=c\,\psi_\mu(x)\psi_\mu(y)
\]
for a real constant $c$.  Thus the resolvent difference has rank at most one.  The coefficient cannot vanish: otherwise the two resolvents, and therefore the two self-adjoint operators, would coincide, contradicting the distinct Neumann and Dirichlet endpoint extensions.  Hence the rank is exactly one.

The form-domain inclusion gives $E^N_n\le E^D_n$ by the min--max principle.  Since the two positive compact resolvents differ by rank one, their min--max values interlace; under the strictly decreasing map $E\mapsto(E-\mu)^{-1}$ this gives $E^D_n\le E^N_{n+1}$.  The same interlacing is equivalent to the stated counting-function bound.
\end{proof}

\begin{lemma}[Standard semigroup and positivity inputs]\label{lem:standard-tools}
The following standard facts are used below in precisely these forms.
\begin{enumerate}[label=\textup{(\roman*)},leftmargin=*]
\item \textup{(Gearhart--Pr\"uss--Huang.)}  For a bounded $C_0$-semigroup on a Hilbert space with generator $A$, exponential stability is equivalent to $i\R\subset\rho(A)$ together with $\sup_{s\in\R}\|(is-A)^{-1}\|<\infty$; when the resolvent is compact, excluding point spectrum on a bounded imaginary segment and proving a high-frequency bound give the bounded-frequency part by continuity of the resolvent.  This is the Gearhart--Pr\"uss--Huang theorem \cite{Gearhart,Pruss,Huang}.
\item \textup{(Quantitative transverse input.)}  After the exact block
decomposition of Lemma~\ref{lem:blocks}, the only quantitative spectral
input from $Y$ is the Weyl-law consequence recorded above: all
sufficiently large dyadic frequency blocks contain a distinct positive
transverse eigenfrequency.
\item \textup{(Compact-circle positivity.)}  Let $V\in L^\infty(\T;\R)$ and let $H=-\partial_x^2+V$ be the form realization on $L^2(\T)$.  Then $H$ is a bounded form perturbation of the compact-resolvent Laplacian; hence $H$ has compact resolvent and $e^{-tH}$ is compact for every $t>0$.  The Feynman--Kac formula is first used for bounded nonnegative Borel data, with Brownian motion normalized to have generator $\partial_x^2$, and the strict positivity of the circle heat kernel extends positivity improvement to every nonzero $f\in L^2_+$.  Compact positivity improvement gives, by the Krein--Rutman/Perron--Frobenius theorem, a simple lowest eigenvalue with a strictly positive eigenfunction.
\item \textup{(Half-line Perron--Sturm.)}  For a real, locally integrable Sturm--Liouville expression on $(0,\infty)$ with a regular separated endpoint at $0$, limit-point endpoint at $+\infty$, semibounded separated form, and compact resolvent, the lowest eigenvalue in each separated self-adjoint realization is simple and the corresponding eigenfunction has a representative strictly positive on $(0,\infty)$ after multiplying by a scalar.  This is the usual one-dimensional Perron--Sturm/oscillation consequence \cite[Secs.~13--14, pp.~194--226]{Weidmann}; in the logarithmic case the endpoint and form-domain details are supplied by Lemma~\ref{lem:log-form-realization} before this item is invoked.
\end{enumerate}
\end{lemma}

\begin{lemma}[Compact-resolvent and Riesz-projection facts]\label{lem:riesz-tools}
Let $A$ be a densely defined closed operator on a complex Hilbert space $H$.
\begin{enumerate}[label=\textup{(\roman*)},leftmargin=*]
\item If $A$ has compact resolvent, then every finite point of $\spec(A)$ is an isolated eigenvalue of finite algebraic multiplicity, and finite spectral points have no finite accumulation point.  For a positively oriented rectifiable Jordan curve $\Gamma\subset\rho(A)$, or a finite union of such curves with pairwise disjoint enclosed regions, the Riesz projection
\[
    P_\Gamma(A)=\frac{1}{2\pi i}\int_\Gamma (z-A)^{-1}\,dz
\]
has finite rank; its rank is the total algebraic multiplicity of the eigenvalues enclosed by $\Gamma$.  No completeness of the total root-vector system is asserted.
\item One has $\rho(A^*)=\overline{\rho(A)}$ and
\[
    P_\Gamma(A)^*=P_{\overline\Gamma}(A^*),
\]
where $\overline\Gamma$ is given the positive orientation around the conjugate enclosed region.  Consequently an isolated eigenvalue $\lambda$ of $A$ and the isolated eigenvalue $\bar\lambda$ of $A^*$ have the same algebraic multiplicity.
\item Let $A_n$ be densely defined closed operators on the same Hilbert space, and suppose that $\Gamma\subset\rho(A)\cap\rho(A_n)$ for all sufficiently large $n$ and
\[
    \sup_{z\in\Gamma}
       \bigl\|(z-A_n)^{-1}-(z-A)^{-1}\bigr\|\longrightarrow0.
\]
Then $P_\Gamma(A_n)\to P_\Gamma(A)$ in operator norm.  If $P_\Gamma(A)$ has finite rank, then for all sufficiently large $n$ the two projections have the same rank; in particular, total algebraic multiplicity inside the contour is stable under this contour-uniform norm-resolvent convergence.
\end{enumerate}
These facts are the standard compact-resolvent and Riesz-projection results
in \cite[Ch.~III, \S6.4--8, pp.~178--188; Ch.~IV, \S3.3--5,
pp.~210--214]{Kato}.
\end{lemma}

\begin{proof}
For \textup{(i)}, compactness of one resolvent converts the local spectral problem for $A$ into that of a compact bounded operator by the resolvent identity; the nonzero spectral points of a compact operator are isolated eigenvalues of finite algebraic multiplicity, and the contour integral is the corresponding finite-rank spectral projection.  For \textup{(ii)},
\[
    \bigl((z-A)^{-1}\bigr)^*=(\bar z-A^*)^{-1},
\]
and conjugation reverses the contour orientation, yielding the displayed projection identity and equality of ranks.  For \textup{(iii)}, integration of the assumed uniform resolvent convergence gives norm convergence of the projections.  Once $\|P_\Gamma(A_n)-P_\Gamma(A)\|<1$, restriction of either projection to the range of the other is injective; hence the two finite-dimensional ranges have the same dimension.
\end{proof}

\subsection{Generator and block reduction}

\begin{lemma}[Well-posedness and block structure]\label{lem:blocks}
\leavevmode
\begin{enumerate}[label=\textup{(\roman*)}]
\item $\cA$ is maximally dissipative on $\cH$, generates a contraction semigroup, and has compact resolvent.
\item The transverse decomposition $L^2(X)=\bigoplus_{j\ge0}L^2(\T)\otimes\C\phi_j$, where $-\Delta_Y\phi_j=\lambda_j\phi_j$, reduces $\cA$; here $(\phi_j)_{j\ge0}$ is a fixed orthonormal eigenbasis of $-\Delta_Y$, repeated eigenvalues producing repeated identical scalar blocks, and all block statements are per basis element --- global multiplicities thus inherit the transverse multiplicity --- while dyadic syndeticity always refers to the set of \emph{distinct} positive frequencies.  Write $\cA_j$ for the block generators.  For $j\ge1$ the block energy space is
\[
    \cH_j=H^1(\T)\times L^2(\T),\qquad
    \|(u,v)\|_{\cH_j}^2=\|u'\|_2^2+\lambda_j\|u\|_2^2+\|v\|_2^2,
\]
the first-component norm being inherited from $\dot H^1(X)$, while the $j=0$ block carries $\cH_0=(H^1(\T)/\C)\times L^2(\T)$ with the quotient first-component norm.  On these spaces
\[
    \cA_j(u,v)=(v,u''-\lambda_j u-av),\qquad j\ge1,
\]
with domain $H^2(\T)\times H^1(\T)$, and the same formula modulo constants on $(H^2(\T)/\C)\times H^1(\T)$ for $j=0$.  A number $z\in\C\setminus\{0\}$ is an eigenvalue of $\cA_j$ if and only if the pencil equation
\begin{equation}\label{eq:pencil-block}
    -v''+(\lambda_j+za+z^2)v=0
\end{equation}
has a nontrivial solution $v\in H^2(\T)$; at $z=0$ no division by the first block equation is made, and the equation $\cA_jU=0$ is understood directly on the corresponding quotient block.  We shall also use the standard product Sobolev characterization: if $u=\sum_j u_j\otimes\phi_j$, then
\begin{equation}\label{eq:product-H2-block}
    \|u\|_{H^2(X)}^2\asymp \sum_{j\ge0}\Bigl(\|u_j''\|_2^2+(1+\lambda_j)\|u_j'\|_2^2+(1+\lambda_j)^2\|u_j\|_2^2\Bigr),
\end{equation}
and analogously $\|u\|_{H^1(X)}^2\asymp\sum_j(\|u_j'\|_2^2+(1+\lambda_j)\|u_j\|_2^2)$.
\item For $s\in\R\setminus\{0\}$, with $P_X(is)$ realized as the closed operator $H^2(X)\to L^2(X)$ and each $Q_{s,\mu}$ as the closed periodic operator $H^2(\T)\to L^2(\T)$, $is\in\rho(\cA)$ if and only if $P_X(is)$ is invertible on $L^2(X)$, and
\begin{equation}\label{eq:P-block}
    \|P_X(is)^{-1}\|_{L^2\to L^2}=\sup_{j\ge0}\|Q_{s,\mu_j}^{-1}\|_{L^2(\T)\to L^2(\T)},
    \qquad \mu_j=s^2-\lambda_j .
\end{equation}
The displayed norms are the $L^2$ operator norms of the inverses.
\item More generally, for $z\in\C$, the reduced direct sum satisfies
\begin{equation}\label{eq:direct-sum-resolvent}
    z\in\rho(\cA)\quad\Longleftrightarrow\quad
    z\in\rho(\cA_j)\ \text{for every }j\ \text{and}\ \sup_j\|(z-\cA_j)^{-1}\|_{\cH_j\to\cH_j}<\infty .
\end{equation}
The norms in \eqref{eq:direct-sum-resolvent} are the operator norms on the block energy spaces just defined.  The resolvent criterion is separate from the finite-point spectral assertion: since $\cA$ has compact resolvent, every finite spectral point is an eigenvalue; decomposing an eigenvector into transverse blocks shows that at least one block has the same spectral point.  Its algebraic multiplicity is the sum of the algebraic multiplicities over the finitely many transverse basis blocks producing the same $z$; in particular repeated transverse eigenvalues add identical scalar-block multiplicities.  No classification of accidental coincidences between different transverse frequencies is asserted.
\end{enumerate}
\end{lemma}

\begin{proof}
(i) $\dot H^2(X)$ is dense in $\dot H^1(X)$ and $H^1(X)$ is dense in $L^2(X)$, so $\mathcal D(\cA)$ is dense in $\cH$.  For $U=([u],v)\in\mathcal D(\cA)$, $\Re\langle\cA U,U\rangle_\cH=-\int_Xa|v|^2\le0$.  For surjectivity of $1-\cA$, let $([f_1],f_2)\in\cH$ and fix the mean-zero representative $f_1$.  The equations $[u]-[v]=[f_1]$, $v-\Delta_Xu+av=f_2$ reduce, upon choosing the representative $v=u-f_1$ (any additive constant in the first equation is absorbed into the representative of $u$), to the single elliptic problem
\[
    (1+a)u-\Delta_Xu=f_2+(1+a)f_1\ \in L^2(X),
\]
which has a unique solution $u\in H^1(X)$ by Lax--Milgram (the form $\int|\nabla u|^2+\int(1+a)|u|^2$ is coercive); since $u\in H^1\subset L^2$ and $a\in L^\infty$, $(1+a)u\in L^2$, so the equation gives $\Delta_Xu\in L^2$ and elliptic regularity yields $u\in H^2(X)$; then $v=u-f_1\in H^1$, $([u],v)\in\mathcal D(\cA)$, and $(1-\cA)([u],v)=([f_1],f_2)$.  Replacing $f_1$ by $f_1+c$ replaces $u$ by $u+c$ and leaves $[u]$ and $v$ unchanged, so the construction is representative-independent; Lumer--Phillips applies.  Compactness: if $\|U\|_\cH+\|\cA U\|_\cH\le1$ with $U=([u],v)$, $\cA U=([v],\Delta_Xu-av)$, then $v$ is bounded in $H^1$ (the $L^2$ bound from $U$, the gradient bound from the first component of $\cA U$), $\|\Delta_Xu\|_{L^2}\le1+\|a\|_\infty\|v\|\le C$, and elliptic regularity with the Poincar\'e inequality bounds the mean-zero representative $u$ in $H^2$.  So graph-bounded sets are bounded in $(H^2\cap\{\int_Xu=0\})\times H^1\Subset(H^1\cap\{\int_Xu=0\})\times L^2=\cH$ on the compact manifold $X$, and the resolvent is compact.  (ii) Since $a=a(x)$, $\cA$ commutes with $I\otimes\Pi_j$, where $\Pi_j$ denotes the orthogonal projection of $L^2(Y)$ onto $\C\phi_j$; the eigenvalue equation $\cA_j(u,v)=z(u,v)$ gives $v=zu$ --- for $j=0$ after normalizing the quotient representative, legitimate since $z\neq0$; the case $z=0$ is always treated separately --- and \eqref{eq:pencil-block}.  (iii) We spell out the quotient normalization because this is the only point where the stationary operator acts on actual representatives.  Let $([f_1],f_2)\in\cH$ and choose the mean-zero representative of $f_1$.  If $([u],v)$ solves $(is-\cA)([u],v)=([f_1],f_2)$, the first component gives $[v]=[isu-f_1]$.  Since $s\ne0$, there is a unique representative of $[u]$ for which
\begin{equation}\label{eq:resolvent-v-rep}
    v=isu-f_1\qquad\text{in }L^2(X).
\end{equation}
Indeed, starting from any representative $u_0$ one has $v=isu_0-f_1+c$ for some constant $c$; replacing $u_0$ by $u=u_0+c/(is)$ gives \eqref{eq:resolvent-v-rep}.  If two representatives $u$ and $u+e$ both satisfy \eqref{eq:resolvent-v-rep}, then $ise=0$, hence $e=0$.  With this uniquely normalized representative fixed, substituting \eqref{eq:resolvent-v-rep} into the second component gives
\begin{equation}\label{eq:resolvent-substitution}
    P_X(is)u=f_2+(is+a)f_1.
\end{equation}
Conversely, if $u\in H^2(X)$ solves \eqref{eq:resolvent-substitution} and $v$ is defined by \eqref{eq:resolvent-v-rep}, then $v\in H^1(X)$ and $([u],v)$ solves the generator resolvent equation.  Thus $is\in\rho(\cA)$ is equivalent to invertibility of $P_X(is):H^2(X)\subset L^2(X)\to L^2(X)$; since this operator is closed, its inverse, when it exists, is bounded from $L^2$ to $H^2$ with the graph norm.

For the block identity, write $f=\sum_j f_j\otimes\phi_j$ and set $u_j=Q_{s,\mu_j}^{-1}f_j$ whenever the block inverses exist with a uniform $L^2$ bound.  Then $u=\sum_j u_j\otimes\phi_j\in L^2(X)$ and $P_X(is)u=f$ distributionally.  Hence $-\Delta_Xu=f+(s^2-isa)u\in L^2(X)$, and elliptic regularity on the compact product gives $u\in H^2(X)$; equivalently, the product characterization \eqref{eq:product-H2-block} holds for the assembled vector.  This proves that the direct sum of the block inverses is the inverse of the closed operator $P_X(is)$, with norm equal to the supremum of the $L^2$ block norms, and gives \eqref{eq:P-block}.

(iv) Since the transverse projections reduce $\cA$, the operator is the Hilbert direct sum of the closed blocks $\cA_j$ on the corresponding block energy spaces.  For a direct sum of closed operators, $(z-\cA)^{-1}$ is the direct sum of $(z-\cA_j)^{-1}$ precisely when all block inverses exist and their norms are uniformly bounded; this gives \eqref{eq:direct-sum-resolvent}.  Conversely, any vector supported in one block shows that failure of a block inverse prevents invertibility of $z-\cA$.  Because $\cA$ has compact resolvent by (i), every finite spectral point is an eigenvalue of finite algebraic multiplicity and finite spectral points cannot accumulate.  For the algebraic statement, let $\Gamma$ be a small circle about $z$ with $\Gamma\subset\rho(\cA)$; then $\Gamma\subset\rho(\cA_j)$ for every $j$, with $\sup_j\sup_{\zeta\in\Gamma}\|(\zeta-\cA_j)^{-1}\|\le\sup_{\zeta\in\Gamma}\|(\zeta-\cA)^{-1}\|<\infty$ by \eqref{eq:direct-sum-resolvent}.  The Riesz projection $\Pi_\Gamma(\cA)=\frac1{2\pi i}\oint_\Gamma(\zeta-\cA)^{-1}\,d\zeta$ therefore commutes with every $I\otimes\Pi_j$ and acts blockwise, $\Pi_\Gamma(\cA)=\bigoplus_j\Pi_\Gamma(\cA_j)$, the contour integral converging blockwise by the uniform bound.  Its range is the generalized eigenspace at $z$ and is finite-dimensional by compactness of the resolvent, so only finitely many $\Pi_\Gamma(\cA_j)$ are nonzero and $\operatorname{rank}\Pi_\Gamma(\cA)=\sum_j\operatorname{rank}\Pi_\Gamma(\cA_j)$: eigenvectors and generalized eigenvectors decompose into block components, the generalized eigenspace is the algebraic direct sum of the block generalized eigenspaces at $z$, and algebraic multiplicities add over the finitely many contributing blocks.
\end{proof}

For the notation $\cA_d$ and $\cA_{d,j}$ introduced above, the same
block definitions and pencil correspondence are always understood with
$a$ replaced by $d$.  After fixing a transverse block $j$ and writing
$s=s_j$, abbreviate
\[
 P_d(z):=P_{d,j}(z)=-\partial_x^2+s^2+zd+z^2.
\]
Closedness and compact resolvent follow by the same argument as
Lemma~\ref{lem:blocks}, with the Lax--Milgram resolvent point chosen at a
real parameter $\lambda>\|d\|_\infty$ when $d$ is sign-changing
(equivalently, $\cA_d$ is a bounded perturbation of the undamped block
generator); the direct-sum and pencil-correspondence statements are
identical.  Dissipativity and nonpositive-real-part conclusions require
$d\ge0$, and any axis-exclusion or unique-continuation step requires the
corresponding nontriviality condition on $d$.

\begin{remark}[Abstract transverse scope]\label{rem:abstract-transverse}
The transverse manifold is a geometric realization of a slightly more
general untwisted tensor-sum input.  Let $B$ be a nonnegative
self-adjoint operator with compact resolvent on a Hilbert space $K$, set
\[
 H_B=-\partial_x^2\otimes I+I\otimes B
 \quad\text{on }L^2(\T)\otimes K,
\]
and put
\[
 \dot{\mathcal D}(H_B^{1/2})
 :=\mathcal D(H_B^{1/2})/\ker H_B,
 \qquad
 \mathcal E_B:=\dot{\mathcal D}(H_B^{1/2})\times L^2(\T;K).
\]
The first quotient is the homogeneous Hilbert quotient: explicitly,
\[
 \|[u]\|_{\dot{\mathcal D}(H_B^{1/2})}
    :=\|H_B^{1/2}u\|,
 \qquad
 \langle([u],v),([w],z)\rangle_{\mathcal E_B}
 :=\langle H_B^{1/2}u,H_B^{1/2}w\rangle+\langle v,z\rangle.
\]
The zero eigenspace is isolated because $H_B$ has compact resolvent, so
this quotient norm is complete.  With $a=a(x)$ acting by multiplication
on $L^2(\T;K)$, the corresponding generator is
\[
 \mathcal A_B([u],v)=([v],-H_Bu-av),\qquad
 \mathcal D(\mathcal A_B)
 =\bigl(\mathcal D(H_B)/\ker H_B\bigr)
      \times\mathcal D(H_B^{1/2}).
\]
This action is representative-independent, since $H_B$ annihilates its
kernel, and $[v]$ is defined for every second component in the displayed
domain.
If the distinct positive square roots of $\spec B$ are dyadically
syndetic, the scalar decomposition and the proofs of
Theorems~\ref{thm:A}--\ref{thm:E} remain valid for the corresponding
abstract damped wave system; multiplicities of $B$ are inherited by the
scalar blocks.  This includes smooth compact transverse manifolds with
the Dirichlet or Neumann Laplacian and the corresponding boundary
condition.  We retain the product-manifold formulation because it exposes
the PDE geometry; Theorem~\ref{thm:F} gives its nonproduct flat-holonomy
extension.
\end{remark}

\begin{lemma}[Axis cleanness for nontrivial transverse damping]\label{lem:axis-cleanness}
Let $0\le d\in L^\infty(\T)$ be real and nontrivial, equivalently $\int_\T d>0$.  Then the corresponding generator $\cA_d$ satisfies $i\R\subset\rho(\cA_d)$.
\end{lemma}

\begin{proof}
The preceding replacement convention gives compact resolvent for $\cA_d$, so it is enough to exclude imaginary-axis eigenvalues.  Let $i\sigma$ be an eigenvalue.  If $\sigma\ne0$, the block reduction gives a nonzero scalar solution of
\[
    -v''+(\lambda_j+i\sigma d-\sigma^2)v=0
\]
for some transverse block.  Taking imaginary parts after pairing with $v$ gives $\sigma\int_\T d|v|^2=0$, hence $dv=0$ a.e. and $v=0$ a.e. on the positive-measure set $\{d>0\}$.  The equation then reduces distributionally on the whole circle to the constant-coefficient equation $-v''+(\lambda_j-\sigma^2)v=0$.  Thus $v$ has a real-analytic representative; since it vanishes on a set of positive measure, it vanishes identically.  This rules out nonzero eigenvectors for $\sigma\ne0$.  At $\sigma=0$, the equation $\cA_d([u],v)=0$ forces $v$ to be constant and $\Delta_Xu=vd$; integrating over $X$ gives $v\int_Xd=0$, hence $v=0$, and then $u$ is constant, so $[u]=0$ in the quotient energy space.  Thus $0$ is not an eigenvalue.  Compact resolvent excludes residual and continuous spectrum on the axis, proving the claim.
\end{proof}

\subsection{The stationary--generator comparison}

The following comparison is standard for damped waves; we include the proof because the paper uses both directions quantitatively, and because the source term $isf_1$ in \eqref{eq:resolvent-substitution} must be controlled through the resolvent identity.

\begin{lemma}[Stationary--generator comparison]\label{lem:comparison}
Let $|s|\ge1$ and suppose $P_X(is)$ is invertible with $m(s)=\|P_X(is)^{-1}\|_{L^2\to L^2}$.  Then $is\in\rho(\cA)$ and
\begin{equation}\label{eq:stat-gen-comparison}
    \|(is-\cA)^{-1}\|_{\cH\to\cH}\le C\bigl(1+|s|\,m(s)\bigr).
\end{equation}
The constant $C$ depends only on the product geometry, the Poincar\'e constants implicit in the quotient convention, and $\|a\|_{L^\infty}$, and is independent of $s$, $m(s)$, and the source terms.  Conversely, if $is\in\rho(\cA)$ then $P_X(is)$ is invertible and
\begin{equation}\label{eq:stat-gen-reverse}
    m(s)\le |s|^{-1}\|(is-\cA)^{-1}\|_{\cH\to\cH}.
\end{equation}
\end{lemma}

\begin{proof}
Write $P_+=P_X(is)$, $P_-=P_X(-is)$, $P=P_+$, and $m=m(s)$; throughout the proof, first components such as $f_1$ are represented by their mean-zero representatives, as in Lemma~\ref{lem:blocks}, so that expressions like $af_1$ and $\nabla f_1$ are well defined.  Let $g\in L^2$, $u=P^{-1}g$.  Pairing $Pu=g$ with $u$,
\begin{equation}\label{eq:P-pairings}
    \|\nabla u\|_2^2-s^2\|u\|_2^2=\Re\langle g,u\rangle,
    \qquad
    s\int_Xa|u|^2=\Im\langle g,u\rangle,
\end{equation}
whence $\|\nabla P^{-1}g\|\le(|s|m+m^{1/2})\|g\|$.  Since $P_+^*=P_-$ and $P_- =\overline{P_+\overline{\,\cdot\,}}$, the same bound holds for $P_-^{-1}$.  Define $P_+^{-1}\operatorname{div}:L^2(X;T^*X)\to L^2(X)$ by duality, $\langle P_+^{-1}\operatorname{div}F,\,g\rangle:=-\langle F,\,\nabla P_-^{-1}g\rangle$ for $g\in L^2(X)$ --- the bounded extension of $F\mapsto P_+^{-1}(\operatorname{div}F)$ from fields with $\operatorname{div}F\in L^2(X)$, as integration by parts on the closed manifold shows.  Then
\begin{equation}\label{eq:P-div}
    \|P^{-1}\operatorname{div}\|_{L^2(X;T^*X)\to L^2(X)}\le C(|s|m+m^{1/2}).
\end{equation}
Now let $(f_1,f_2)\in\cH$ with $F=\|(f_1,f_2)\|_\cH$; recall $\|f_1\|_2\le C\|\nabla f_1\|_2\le CF$ by Poincar\'e--Wirtinger for the mean-zero representative.  Solve \eqref{eq:resolvent-substitution} as $u=u_1+u_2$ with $u_2=P^{-1}(f_2+af_1)$, so $\|u_2\|\le CmF$, and $u_1=P^{-1}(isf_1)$.  With the convention $\langle\operatorname{div}F,\varphi\rangle=-\langle F,\nabla\varphi\rangle$, one has $-\Delta_X f_1=-\operatorname{div}\nabla f_1$ in distributions.  We first justify the next resolvent identity at the $H^{-1}$ level.  Choose mean-zero $f_{1,n}\in H^2(X)$ with $f_{1,n}\to f_1$ in $H^1(X)$.  For $f_{1,n}$ the identity
\[
    P^{-1}(isf_{1,n})=\frac{i}{s}\Bigl(-f_{1,n}-P^{-1}\operatorname{div}(\nabla f_{1,n})+isP^{-1}(af_{1,n})\Bigr)
\]
holds in $L^2$, since $s^2P^{-1}=-I+P^{-1}(-\Delta_X+isa)$ on $H^2$ data and $-\Delta_Xf_{1,n}=-\operatorname{div}\nabla f_{1,n}$.  Passing to the limit uses $P^{-1}:L^2\to L^2$, $a\in L^\infty$, and the bounded extension \eqref{eq:P-div} of $P^{-1}\operatorname{div}:L^2(T^*X)\to L^2$.  Hence
\[
    u_1=\frac{i}{s}\Bigl(-f_1-P^{-1}\operatorname{div}(\nabla f_1)+isP^{-1}(af_1)\Bigr),
\]
with the middle term defined via \eqref{eq:P-div} since $f_1\in\dot H^1$ only.  By duality: with $\langle\operatorname{div}F,\varphi\rangle:=-\langle F,\nabla\varphi\rangle$, one has $\langle P_+^{-1}\operatorname{div}F,\,g\rangle=-\langle F,\,\nabla P_-^{-1}g\rangle$ ($P_+^*=P_-$, $a$ being real), so this operator norm equals that of $\nabla P_-^{-1}$, bounded by the gradient estimate above.
Thus by \eqref{eq:P-div},
\[
    \|u_1\|\le \frac{C}{|s|}F+\frac{C}{|s|}(|s|m+m^{1/2})F+CmF
    \le C\Bigl(m+\frac1{|s|}\Bigr)F .
\]
Hence $|s|\,\|u\|\le C(1+|s|m)F$ and $\|v\|=\|isu-f_1\|\le C(1+|s|m)F$.  Finally, by the first identity in \eqref{eq:P-pairings} applied to $g=f_2+(is+a)f_1$ (so $\|g\|\le C|s|F$),
\[
    \|\nabla u\|^2\le s^2\|u\|^2+\|g\|\|u\|\le C(1+|s|m)^2F^2 ,
\]
which proves \eqref{eq:stat-gen-comparison}.  For the converse, given $g\in L^2$ apply $(is-\cA)^{-1}$ to $(0,g)$: with the representative of the first component normalized as in Lemma~\ref{lem:blocks}(iii), $v=isu$ holds in $L^2$ and $u$ solves $Pu=g$, so $\|u\|=\|v\|/|s|\le|s|^{-1}\|(is-\cA)^{-1}\|\,\|g\|$.
\end{proof}

\subsection{Scale stability of critical gauges}

\begin{lemma}[Scale stability]\label{lem:scale-stability}
Let $L(S)=\ell(\log(eS))$ be critical.  Then for every $C_0>0$, $\alpha>0$, and $\gamma\in\R$ there are $C_1$, depending only on $\alpha$ and the doubling constant of $\ell$ --- \emph{not} on $C_0$, $\gamma$, or the normalization of $\ell$ --- and $S_1\ge S_0$, depending in addition on $C_0$, $\gamma$, and on the full gauge data, such that
\begin{equation}\label{eq:scale-stability}
    C_1^{-1}L(S)\le L\bigl(C_0S^\alpha L(S)^{\gamma}\bigr)\le C_1L(S),
    \qquad S\ge S_1 .
\end{equation}
In particular, for a large parameter $K\ge1$ in place of $C_0$, only the threshold $S_1(K)$ depends on $K$; the comparison constant does not.
\end{lemma}

\begin{proof}
For $u\ge u_0$, put $m=\lceil\log_2(u/u_0)\rceil$.  Monotonicity and
iteration of the specified eventual-doubling inequality give
\[
 \ell(u)\le \ell(2^m u_0)
   \le C_\ell^m\ell(u_0)
   \le C_\ell\ell(u_0)(u/u_0)^{p_\ell},
 \qquad p_\ell:=\log_2C_\ell.
\]
Thus $\ell(u)\le Cu^{p_\ell}$, with exponent depending only on the
doubling constant and $C$ depending on the full gauge data; hence
$\log L(S)=O(\log\log S)$ with constants allowed in the eventual
threshold.  Therefore
\[
    \begin{aligned}
    \log\bigl(eC_0S^\alpha L(S)^\gamma\bigr)
    &=\alpha\log(eS)+O_\gamma(\log\log S)+O_{C_0}(1)\\
    &\in[\tfrac\alpha2\log(eS),\,2\alpha\log(eS)]
    \end{aligned}
\]
for $S\ge S_1(C_0,\gamma)$ --- the containment needs only $|O_\gamma(\log\log S)|+O_{C_0}(1)\le\tfrac\alpha2\log(eS)$, and we also enlarge $S_1$ so that $C_0S^\alpha L(S)^\gamma\ge S_0$ before $L$ is evaluated, which is possible because $L(S)^\gamma\ge C(\log(eS))^{-p_\ell|\gamma|}$ when $\gamma<0$ and the left side still tends to $+\infty$ --- and monotonicity plus finitely many doublings of $\ell$ give \eqref{eq:scale-stability}.  The bracket $[\tfrac\alpha2\log(eS),\,2\alpha\log(eS)]$ does not involve $C_0$, $\gamma$, or the normalization of $\ell$, so the number of doublings, hence $C_1$, depends only on $\alpha$ and the doubling constant; $C_0$ (or $K$), $\gamma$, and the full gauge normalization enter only through the threshold $S_1$, the comparison constant being independent of them.
\end{proof}

\section{The local-mass resolvent mechanism}\label{sec:machine}

This section contains the two analytic engines of the rough inverse theorem.  Weighted
Poincar\'e control gives an absorption estimate away from propagation,
while Nazarov's inequality controls the uniformly finite Fourier cluster
near propagation; taking the stronger of these two estimates yields the
upper local-mass engine.
Conversely, an exact local ODE solution cut off in an interval of minimal
damping mass produces a quasimode in the resonant block.  The first engine
turns mass into resolvent bounds, and the second recovers mass from sparse
resolvent information.

Throughout this section $0\le a\in L^\infty(\T)$ with $\int_\T a>0$; since $\Theta_a(\pi)=(2\pi)^{-1}\int_\T a>0$, a scale $r_0\in(0,\pi]$ with $\Theta_a(r_0)>0$ exists, and one is fixed; set $m_0=\int_\T a\ge 2r_0\Theta_a(r_0)>0$.

\begin{lemma}[Weighted interval Poincar\'e inequality]\label{lem:wp}
Let $I\subset\R$ be a bounded interval with $0<|I|<\infty$, and let $0\le w\in L^1(I)$ satisfy $\int_Iw>0$.  Then for $v\in H^1(I)$,
\begin{equation}\label{eq:wp}
    \int_I|v|^2\le \frac{2|I|}{\int_Iw}\int_Iw|v|^2+2|I|^2\int_I|v'|^2 .
\end{equation}
Consequently, there is a universal constant $C_P$ such that for every $0<r\le\pi$ with $\Theta_a(r)>0$ and every $v\in H^1(\T)$,
\begin{equation}\label{eq:periodic-wp}
    \|v\|_2^2\le C_P\Bigl(\Theta_a(r)^{-1}\int_\T a|v|^2+r^2\|v'\|_2^2\Bigr) .
\end{equation}
\end{lemma}

\begin{proof}
Let $v_w=(\int_Iw)^{-1}\int_Iwv$.  Then $|v_w|^2\le(\int_Iw)^{-1}\int_Iw|v|^2$ and, for $x\in I$,
\[
    |v(x)-v_w|^2\le\frac{1}{\int_Iw}\int_Iw(y)|v(x)-v(y)|^2\,dy\le|I|\int_I|v'|^2 .
\]
Integrating $|v|^2\le2|v_w|^2+2|v-v_w|^2$ over $I$ gives \eqref{eq:wp}; covering $\T$ by the arcs $I_k$ of radius $r$ centered at $kr$, $0\le k<\lceil2\pi/r\rceil$ --- every point of $\T$ lies in at least one and at most three of them, wrapped arcs included --- summing \eqref{eq:wp} over $k$, and using $\int_{I_k}a\ge2r\Theta_a(r)$ gives \eqref{eq:periodic-wp} with a universal $C_P$.
\end{proof}

\begin{lemma}[Propagating-frequency observability]\label{lem:observability}
There is $C_1=C_1(\|a\|_\infty,m_0)$ such that for all $s\ge1$, all $1\le\mu\le s^2$, and all $v\in H^2(\T)$,
\begin{equation}\label{eq:observability}
    \|v\|_2\le C_1\,\frac{s}{\mu}\,\|Q_{s,\mu}v\|_2 .
\end{equation}
\end{lemma}

\begin{proof}
Write $f=Q_{s,\mu}v$, $N=\|v\|_2$, $F=\|f\|_2$.  Pairing with $v$ and taking imaginary parts,
\begin{equation}\label{eq:im-pair}
    s\int_\T a|v|^2=\Im\langle f,v\rangle\le FN .
\end{equation}
Let $\Pi$ be the projection onto $V_\mu=\operatorname{span}\{e^{inx}:\ \bigl||n|-\sqrt\mu\bigr|\le2\}$, a space of dimension $\le10$, and split $v=v_c+v_\perp$ with $v_c=\Pi v$.  For frequencies outside the cluster, $|n^2-\mu|=\bigl||n|-\sqrt\mu\bigr|\,(|n|+\sqrt\mu)\ge2\sqrt\mu$, so from $(-\partial_x^2-\mu)v=f-isav$ and \eqref{eq:im-pair},
\begin{equation}\label{eq:gap-bound}
    2\sqrt\mu\,\|v_\perp\|\le\|f\|+s\|a\|_\infty^{1/2}\Bigl(\int a|v|^2\Bigr)^{1/2}
    \le F+\|a\|_\infty^{1/2}(sFN)^{1/2}.
\end{equation}
Let $\delta=m_0/(8\pi)$ and $E=\{a\ge\delta\}$; then $m_0\le2\pi\delta+\|a\|_\infty|E|$, so $|E|\ge m_0/(2\|a\|_\infty)$.  By Lemma~\ref{lem:nazarov-input}, applied on a length-$2\pi$ representative of $\T$ and using that $v_c$ has at most $10$ Fourier frequencies,
\[
    \sup_\T|v_c|\le\Bigl(\frac{C}{|E'|}\Bigr)^{9}\sup_{E'}|v_c|
\]
for every measurable $E'\subset\T$ with $|E'|>0$; the constant is independent of the frequency locations, and every set to which this is applied below has positive measure.
If $v_c\equiv0$, skip this step; if $v_c\not\equiv0$ but $\|v_c\|_{L^2(E)}=0$, then $v_c$, a finite exponential polynomial and hence continuous, vanishes identically on the subset $E':=\{x\in E:v_c(x)=0\}$ with $|E'|=|E|>0$, so Nazarov's inequality with this $E'$ gives $\sup_\T|v_c|=0$, i.e.\ $v_c\equiv0$, a contradiction --- hence $\|v_c\|_{L^2(E)}>0$.  Now apply this with the Chebyshev sublevel set $E'=\{x\in E:\ |v_c(x)|^2\le2\|v_c\|_{L^2(E)}^2/|E|\}$, which satisfies $|E'|\ge|E|/2$; the Nazarov constant is uniform over all measurable sets of the prescribed measure, so $E'$ may depend on $v_c$.  On $E'$, $\sup_{E'}|v_c|\le(2/|E|)^{1/2}\|v_c\|_{L^2(E)}$, and $\|v_c\|_{L^2(\T)}\le(2\pi)^{1/2}\sup_\T|v_c|$.  Hence $\|v_c\|_{L^2(\T)}\le C_{m_0,\|a\|_\infty}\|v_c\|_{L^2(E)}$.  The same displayed estimate is trivial when $v_c=0$.  Since $a\ge\delta$ on $E$ and by \eqref{eq:im-pair},
\[
    \|v_c\|_{L^2(E)}
      \le\|v\|_{L^2(E)}+\|v_\perp\|_{L^2(E)}
      \le\Bigl(\frac{FN}{s\delta}\Bigr)^{1/2}
         +\|v_\perp\|_{L^2(\T)} .
\]
Combining with \eqref{eq:gap-bound} and Young's inequality ($(FN/(s\delta))^{1/2}\le\frac14N+CF/s$ and $\mu^{-1/2}(sFN)^{1/2}\le\frac14N+CsF/\mu$),
\[
    N\le\|v_c\|+\|v_\perp\|\le\frac12N+C\Bigl(\frac Fs+\frac F{\sqrt\mu}+\frac{sF}\mu\Bigr)
    \le\frac12N+C\,\frac s\mu\,F,
\]
using $\mu\le s^2$ in the last step.  This proves \eqref{eq:observability}.
\end{proof}

The next proposition has a stronger hypothesis than the standing assumption of this section: it requires positivity of the lower mass at every radius up to the working scale.  In all applications below this is supplied either by a uniform endpoint lower bound or by an explicit critical/power lower-mass hypothesis.

\begin{proposition}[Upper local-mass engine]\label{prop:upper-engine}
Assume $\Theta_a(r)>0$ for $0<r\le r_0$.  For $s\ge s_0$, $\mu\le s^2$, and $v\in H^2(\T)$,
\begin{equation}\label{eq:upper-engine}
    \begin{aligned}
    \|v\|_2&\lesssim \cU_a(s)\,\|Q_{s,\mu}v\|_2,\\
    \cU_a(s)&=\sup_{c_*/(2s)\le t\le r_0}\min\Bigl\{st^2,\ \inf_{0<r\le t}\bigl(r^2+s^{-1}\Theta_a(r)^{-1}\bigr)\Bigr\},
    \end{aligned}
\end{equation}
where $\Theta_a(r)^{-1}:=+\infty$ if $\Theta_a(r)=0$ and $c_*:=\tfrac12(1+C_P)^{-1/2}$, with $C_P$ the constant in \eqref{eq:periodic-wp}; the implicit constants depend only on $r_0$, $\|a\|_\infty$, and $m_0$, not on $\mu$ or on $Y$; the threshold $s_0$ depends only on $r_0$ and on the value $\Theta_a(r_0')$ at the explicit smaller scale $r_0':=\min\bigl\{r_0,\ c_*(1+\max(1,r_0^{-2}))^{-1/2}\bigr\}$ fixed in the proof --- not on $\Theta_a(r_0)$ alone, which does not control masses at smaller scales.  Consequently $P_X(is)$ is invertible for $s\ge s_0$, with
\[
    \|P_X(is)^{-1}\|_{L^2\to L^2}\lesssim\cU_a(s).
\]
\end{proposition}

\begin{proof}
Write $\mathrm{abs}(t)=\inf_{0<r\le t}(r^2+s^{-1}\Theta_a(r)^{-1})$, a nonincreasing function of $t$, $f=Q_{s,\mu}v$, $N=\|v\|$, $F=\|f\|$.  Pairing with $v$,
\begin{equation}\label{eq:energy-pair}
    s\int_\T a|v|^2\le FN,\qquad \|v'\|_2^2\le\mu_+N^2+FN,\qquad \mu_+:=\max(\mu,0) .
\end{equation}

\emph{Absorption branch (all $\mu\le s^2$).}  Set $t_\mu=\min(r_0,\,c_*(1+\mu_+)^{-1/2})$.  For any $0<r\le t_\mu$ with $\Theta_a(r)>0$ --- radii with $\Theta_a(r)=0$ contribute $+\infty$ to $\mathrm{abs}$ and may be ignored --- insert \eqref{eq:energy-pair} into \eqref{eq:periodic-wp}:
\[
    N^2\le C_P\Theta_a(r)^{-1}\frac{FN}{s}+C_Pr^2(\mu_+N^2+FN)
    \le \frac14N^2+C\bigl(r^2+s^{-1}\Theta_a(r)^{-1}\bigr)FN,
\]
since $C_Pr^2\mu_+\le C_Pc_*^2\le\frac14$; hence $N\le C\,\mathrm{abs}(t_\mu)\,F$.

\emph{Propagating branch ($1\le\mu\le s^2$).}  Lemma~\ref{lem:observability} gives $N\le C_1(s/\mu)F$.

It remains to check that for every $\mu\le s^2$ the better of the two bounds is $\lesssim\cU_a(s)$.  If $\mu\ge\max(1,r_0^{-2})$, then $t_\mu\asymp\mu^{-1/2}$, so $s/\mu\asymp st_\mu^2$ and both branches together give $N\le C\min\{st_\mu^2,\mathrm{abs}(t_\mu)\}F\le C\,\cU_a(s)F$ (note $t_\mu\ge c_*(1+s^2)^{-1/2}\ge c_*/(2s)$ since $\mu\le s^2$).  If $\mu\le\max(1,r_0^{-2})$, we use the absorption branch alone at the fixed radius $r_0':=\min\bigl\{r_0,\ c_*(1+\max(1,r_0^{-2}))^{-1/2}\bigr\}\le t_\mu$, a constant depending only on $r_0$, admissible in the supremum once $c_*/(2s)\le r_0'$, i.e.\ $s\ge c_*/(2r_0')$: $N\le C\,\mathrm{abs}(r_0')F$; and since $st^2$ is increasing while $\mathrm{abs}$ is nonincreasing, once $s(r_0')^2\ge\mathrm{abs}(r_0')$ the definition of $\cU_a$ at $t=r_0'$ gives $\cU_a(s)\ge\mathrm{abs}(r_0')$, as required.  Since $\mathrm{abs}(r_0')\le(r_0')^2+s^{-1}\Theta_a(r_0')^{-1}$ with $\Theta_a(r_0')>0$ by hypothesis, that condition holds as soon as $s\ge2$ and $s^2\ge2(r_0')^{-2}\Theta_a(r_0')^{-1}$: it is through the value $\Theta_a(r_0')$, and only through it, that $s_0$ depends on the damping.  We also take $s_0\ge c_*/(2r_0')$, so the interval in the supremum defining $\cU_a(s)$ is nonempty.  The hypothesis $\Theta_a>0$ on all of $(0,r_0]$ strengthens the standing assumption of this section; every application below (Theorems~\ref{thm:A}(i), \ref{thm:B}(i), \ref{thm:C}(b)) supplies it through its mass lower bound.  (For $\mu\le-1$, coercivity gives the stronger bound $N\le|\mu|^{-1}F$, but the absorption branch already covers all $\mu\le1$.)  For the resolvent statement: the estimate just proved holds for every block $Q_{s,\mu_j}$, $\mu_j=s^2-\lambda_j\le s^2$, uniformly in $j$, giving injectivity and closed range; since $Q_{s,\mu}^*v=\overline{Q_{s,\mu}\bar v}$, applying the estimate to $\bar v$ gives $\|v\|=\|\bar v\|\le C\,\cU_a(s)\|Q_{s,\mu}\bar v\|=C\,\cU_a(s)\|Q_{s,\mu}^*v\|$, so $Q_{s,\mu}^*$ is injective, the closed range is dense, and each block is bijective with $\|Q_{s,\mu_j}^{-1}\|\le C\,\cU_a(s)$ uniformly in $j$.  The blockwise inverses assemble: for $f=\sum_jf_j\otimes\phi_j\in L^2(X)$, the vector $u:=\sum_jQ_{s,\mu_j}^{-1}f_j\otimes\phi_j$ satisfies $\|u\|\le C\,\cU_a(s)\|f\|$ and lies in $H^2(X)$.  Indeed all three $H^2$ sums are controlled with constants $C(s)$ uniform in $j$: transverse regularity from $\lambda_j\|u_j\|\le C(s)\|f_j\|$ (for $\mu_j\le-1$, coercivity gives $\|u_j\|\le|\mu_j|^{-1}\|f_j\|$ with $\lambda_j\le(s^2+1)|\mu_j|$; for $|\mu_j|\le s^2$, $\lambda_j\le2s^2$ and $\|u_j\|\le C\,\cU_a(s)\|f_j\|$, $s$ being fixed); tangential regularity from the block equation $u_j''=-(f_j+\mu_ju_j-isau_j)$, which gives $\|u_j''\|\le\|f_j\|+(|\mu_j|+s\|a\|_\infty)\|u_j\|\le C(s)\|f_j\|$ in both regimes; and the mixed term from $\|u_j'\|^2=\Re\langle-u_j'',u_j\rangle\le\|u_j''\|\,\|u_j\|$, so $\lambda_j\|u_j'\|^2\le(\lambda_j\|u_j\|)\|u_j''\|\le C(s)^2\|f_j\|^2$.  Hence $\sum_j(\|u_j''\|^2+\lambda_j\|u_j'\|^2+\lambda_j^2\|u_j\|^2)\le C(s)^2\|f\|^2$.  The missing unweighted terms in the $(1+\lambda_j)$ criterion are controlled separately: $\sum_j\|u_j\|^2\le C(s)^2\|f\|^2$, and, by the periodic integration identity and Cauchy--Schwarz, $\sum_j\|u_j'\|^2\le\bigl(\sum_j\|u_j''\|^2\bigr)^{1/2}\bigl(\sum_j\|u_j\|^2\bigr)^{1/2}\le C(s)^2\|f\|^2$.  Together these estimates give the full right-hand side of the product characterization \eqref{eq:product-H2-block}, so $u\in H^2(\T\times Y)$ and $u$ solves $P_X(is)u=f$.  Hence $P_X(is)$ is invertible, and the direct-sum identity \eqref{eq:P-block} gives the displayed estimate for $P_X(is)^{-1}$.
\end{proof}

\begin{lemma}[Critical collapse]\label{lem:critical-collapse}
If $L$ is critical and $\Theta_a(r)\ge cL(1/r)^{-1}$ for $0<r\le r_0$, then
\begin{equation}\label{eq:collapse}
    \cU_a(s)\lesssim s^{-1}L(s),\qquad s\ \text{large}.
\end{equation}
\end{lemma}

\begin{proof}
Let $\rho_s=s^{-1}L(s)^{1/2}$.  Since critical gauges are subpolynomial, $\rho_s\to0$ and $1/\rho_s=sL(s)^{-1/2}\to\infty$.  Increase the threshold in $s$ so that $\rho_s\le r_0$ and $1/\rho_s\ge\max(S_0,S_1)$, where $S_1$ is the scale-stability threshold for the rescaling $S\mapsto SL(S)^{-1/2}$; also $\rho_s\ge c/s$ because $L\ge1$.  If $t\le\rho_s$ then $st^2\le s^{-1}L(s)$.  If $t\ge\rho_s$, take $r=\rho_s$ in the absorption branch: by Lemma~\ref{lem:scale-stability}, $L(1/\rho_s)=L(sL(s)^{-1/2})\asymp L(s)$, so
\[
    \rho_s^2+s^{-1}\Theta_a(\rho_s)^{-1}\le s^{-2}L(s)+Cs^{-1}L(s)\lesssim s^{-1}L(s).\qedhere
\]
\end{proof}

\begin{proposition}[Absorbing ODE quasimode]\label{prop:absorbing}
There are $c_0,c,C>0$, depending only on $\|a\|_\infty$, with the following property.  If $s>0$ and $0<r\le\tfrac12$ satisfy the admissibility condition
\begin{equation}\label{eq:admissible}
    s\,r^2\,\Theta_a(r)\le c_0,
\end{equation}
then there is $v\in H^2(\T)\setminus\{0\}$, supported in an interval of radius $r$, with
\begin{equation}\label{eq:absorbing-quasimode}
    \bigl\|(-\partial_x^2+isa)v\bigr\|_2\le C\bigl(r^{-2}+s\,\Theta_a(r)\bigr)\|v\|_2 .
\end{equation}
Consequently, whenever $Q_{s,0}$ is invertible,
\begin{equation}\label{eq:absorbing-scalar-lower}
    \|Q_{s,0}^{-1}\|\ \ge\
       c\bigl(r^{-2}+s\Theta_a(r)\bigr)^{-1}.
\end{equation}
If, in addition, $s=s_j$ is a transverse eigenfrequency, then either
$is_j\in\spec(\cA)$ or
\begin{equation}\label{eq:absorbing-lower}
    \|P_X(is_j)^{-1}\|\ \ge\ c\bigl(r^{-2}+s_j\Theta_a(r)\bigr)^{-1}.
\end{equation}
\end{proposition}

\begin{proof}
The infimum in \eqref{eq:Theta} is attained: $x\mapsto\int_{x-r}^{x+r}a$ is continuous by $L^1$-continuity of translations, and $\T$ is compact.  So choose $I=(x_0-r,x_0+r)$ with $\int_Ia=2r\Theta_a(r)$.  Let $w$ solve $-w''+isaw=0$ on $I$ with $w(x_0)=1$, $w'(x_0)=0$.  The Volterra representation
\[
    w(x)=1+is\int_{x_0}^x(x-y)a(y)w(y)\,dy
\]
gives (the arc $I$ being identified with an interval of $\R$, possible since $2r<2\pi$), by Gr\"onwall applied on each side of $x_0$ separately --- for $x\in[x_0,x_0+r]$ one has $|x-y|\le r$ on the domain of integration, so $|w(x)|\le1+rs\int_{x_0}^{x}a(y)|w(y)|\,dy$ and Gr\"onwall yields $\sup_{[x_0,x_0+r]}|w|\le e^{rs\int_Ia}=e^{2sr^2\Theta_a(r)}$; likewise on $[x_0-r,x_0]$ --- $\sup_I|w|\le e^{2sr^2\Theta_a(r)}\le e^{2c_0}$, then $|w-1|\le 2sr^2\Theta_a(r)e^{2c_0}\le\frac12$ for $c_0$ small, and $|w'|\le s\int_Ia\,\sup|w|\le Csr\Theta_a(r)$ on $I$.  Since $w\in L^\infty(I)$ and $a\in L^\infty$, the equation gives $w''=isaw\in L^\infty(I)$, hence $w\in W^{2,\infty}(I)$.  Let $\chi\in C_c^\infty(I)$ equal $1$ on the middle half with $|\chi'|\lesssim r^{-1}$, $|\chi''|\lesssim r^{-2}$, and set $v=\chi w\in H^2(\T)$.  Then $(-\partial_x^2+isa)v=-\chi''w-2\chi'w'$, so
\[
    \|(-\partial_x^2+isa)v\|_2\lesssim\bigl(r^{-2}+s\Theta_a(r)\bigr)r^{1/2},
    \qquad \|v\|_2\gtrsim r^{1/2},
\]
which is \eqref{eq:absorbing-quasimode}.  If $Q_{s,0}$ is invertible,
applying its inverse to this quasimode gives
\eqref{eq:absorbing-scalar-lower}.  Multiplying $v$ by a transverse
eigenfunction $\phi_j$ with $\lambda_j=s_j^2$ gives
\[
    P_X(is_j)(v\phi_j)=\bigl((-\partial_x^2+is_ja)v\bigr)\phi_j .
\]
If $P_X(is_j)$ is not invertible, Lemma~\ref{lem:blocks}\textup{(iii)} gives $is_j\notin\rho(\cA)$, and since $\cA$ has compact resolvent this means $is_j\in\spec(\cA)$.  If $P_X(is_j)$ is invertible, applying $P_X(is_j)^{-1}$ to the displayed quasimode gives
\[
    \|P_X(is_j)^{-1}\|\ge c\bigl(r^{-2}+s_j\Theta_a(r)\bigr)^{-1},
\]
which is \eqref{eq:absorbing-lower}.
\end{proof}

\section{The endpoint and the critical dictionary}\label{sec:dictionary}

Throughout this section we take $s>0$; since $\cA$ has real coefficients and $P_X(-is)v=\overline{P_X(is)\bar v}$, all axis estimates are symmetric in $s$.

\begin{remark}[A priori bounds imply invertibility]\label{rem:invertibility}
If $\|v\|_2\le M\|Q_{s,\mu}v\|_2$ for all $v\in H^2(\T)$, then $Q_{s,\mu}$ is invertible with $\|Q_{s,\mu}^{-1}\|\le M$: the estimate gives injectivity and closed range, and since $Q_{s,\mu}^*v=\overline{Q_{s,\mu}\overline v}$, the adjoint is injective as well, so the range is dense.  As the blocks of $P_X(is)$ are then invertible with uniformly bounded inverses, the conclusion of Proposition~\ref{prop:upper-engine} reads: \emph{$P_X(is)$ is invertible and $\|P_X(is)^{-1}\|\lesssim\cU_a(s)$ for $s\ge s_0$}; this stronger form is the one used later.  No invertibility hypothesis is imposed.
\end{remark}

\subsection{Proof of Theorem \ref{thm:A}}

If $a=0$ a.e., then $\cA$ is the standard skew-adjoint realization of the undamped wave generator on $\cH$: $\Re\langle\cA U,U\rangle_\cH=0$, and the Lax--Milgram surjectivity argument of Lemma~\ref{lem:blocks}\textup{(i)}, with $a=0$ and with the signs changed, gives both $\operatorname{Ran}(I-\cA)=\cH$ and $\operatorname{Ran}(I+\cA)=\cH$.  Thus $\cA$ and $-\cA$ are maximally dissipative, so $e^{t\cA}$ is a unitary group.  For every $j\ge1$, $Y$ being connected gives $\phi_j\perp1$, so $1\otimes\phi_j$ has zero mean on $X$ and represents a nonzero element of $\dot H^1(X)$; consequently $U_j=(1\otimes\phi_j,\ is_j\,1\otimes\phi_j)$ is an eigenvector of $\cA$ at $is_j$.  Thus (i)--(iv) all fail, and the equivalence holds trivially.  Assume henceforth $\int_\T a>0$ and set $r_0:=\pi$; then $\Theta_a(r_0)=(2\pi)^{-1}\int_\T a>0$, so this explicitly supplies the working scale required by Proposition~\ref{prop:upper-engine}.

\emph{(i)$\Rightarrow$(iii).}  The hypothesis gives $c_\flat>0$ and $r'\in(0,\pi]$ with $\Theta_a(r)\ge c_\flat$ for all $0<r\le r'$; after decreasing $r_0$ to $\min(r_0,r')$ --- harmless, every constant below being allowed to depend on it --- assume $\Theta_a(r)\ge c_\flat>0$ for $0<r\le r_0$.  For any admissible $t$ in \eqref{eq:upper-engine}, choosing $r=\min(t,s^{-1/2})$ in the absorption branch gives
\[
    \inf_{0<r\le t}\bigl(r^2+s^{-1}\Theta_a(r)^{-1}\bigr)\le (1+c_\flat^{-1})\,s^{-1},
\]
so $\cU_a(s)\le(1+c_\flat^{-1})s^{-1}$; the hypothesis gives $\Theta_a(r)\ge c_\flat>0$ for all $0<r\le r_0$, in particular at the auxiliary scale $r_0'$, so the threshold data of Proposition~\ref{prop:upper-engine} are controlled, and by that proposition and Remark~\ref{rem:invertibility}, $\|P_X(is)^{-1}\|\le Cs^{-1}$ for $s\ge s_0$; by Lemma~\ref{lem:comparison}, $is\in\rho(\cA)$ and $\|(is-\cA)^{-1}\|\le C$ for $|s|\ge s_0$.  For bounded frequencies it suffices, by compactness of the resolvent (Lemma~\ref{lem:blocks}) --- so that every spectral point is an eigenvalue of finite algebraic multiplicity --- to exclude eigenvalues on $i[-s_0,s_0]$.  At $z=0$: $\cA U=0$ gives $[v]=0$, so $v$ is a constant $c$; the second component reads $\Delta_Xu=ca$, and integrating over $X$ gives $c\int_Xa=0$, so $c=0$ (as $\int_\T a\ge2r_0c_\flat>0$ under (i)), $v=0$, $\Delta_Xu=0$, and $[u]=0$ in $\dot H^1$.  At $z=is$, $0<|s|\le s_0$: by Lemma~\ref{lem:blocks} an eigenfunction produces, for some $j$, a nontrivial $v\in H^2(\T)$ with $-v''+(\lambda_j-s^2)v+isav=0$; pairing with $v$ and taking imaginary parts gives $s\int_\T a|v|^2=0$, so $av=0$ a.e.  Thus $v$ satisfies $-v''+(\lambda_j-s^2)v=0$ distributionally on $\T$.  Solutions of this constant-coefficient equation are real-analytic, and $v$ vanishes a.e.\ on $\{a>0\}$, a set of positive measure since $\int_\T a\ge 2r_0c_\flat>0$; therefore $v\equiv0$.  Thus $i\R\subset\rho(\cA)$, and continuity of the resolvent on the compact segment $i[-s_0,s_0]$ gives (iii).

\emph{(iii)$\Rightarrow$(iv).}  This is the Gearhart--Pr\"uss--Huang theorem \cite{Gearhart,Pruss,Huang}.

\emph{(iv)$\Rightarrow$(ii).}  By the standard resolvent characterization of exponential stability for bounded semigroups on Hilbert space, exponential stability gives $i\R\subset\rho(\cA)$ and $\sup_{s\in\R}\|(is-\cA)^{-1}\|<\infty$; then \eqref{eq:stat-gen-reverse} yields (ii).

\emph{(ii)$\Rightarrow$(i).}  We prove the contrapositive.  If $\liminf_{r\downarrow0}\Theta_a(r)=0$, choose $r_k\downarrow0$ with $\theta_k=\Theta_a(r_k)\to0$.  Fix $K\ge1$.  Since the transverse eigenfrequencies are dyadically syndetic (\S\ref{subsec:notation}) and every window $[T,4T)$ contains a full dyadic block $[2^m,2^{m+1})$ (take $2^m\in[T,2T)$), for large $k$ we may choose $s^{(k)}\in\{s_j\}\cap[Kr_k^{-2},4Kr_k^{-2})$.  Then $s^{(k)}r_k^2\theta_k\le 4K\theta_k\to0$, so \eqref{eq:admissible} holds for $k$ large; passing to larger $k$, also $r_k\le\tfrac12$, so the radius hypothesis of Proposition~\ref{prop:absorbing} is satisfied.  Hence Proposition~\ref{prop:absorbing} applies: either $is^{(k)}\in\spec(\cA)$, in which case $P_X(is^{(k)})$ is not invertible and (ii) fails, or
\[
    \|P_X(is^{(k)})^{-1}\|\ \ge\ c\bigl(r_k^{-2}+s^{(k)}\theta_k\bigr)^{-1}
    \ \ge\ \frac{c}{2}\,r_k^{2}
    \ \ge\ \frac{cK}{2\,s^{(k)}},
\]
using $s^{(k)}\theta_k\le 4K\theta_kr_k^{-2}\le r_k^{-2}$ for large $k$ and, from $s^{(k)}\ge Kr_k^{-2}$, the inequality $r_k^2\ge K/s^{(k)}$.  If (ii) held with constant $C$, then $cK/2\le C$, which fails once $K>2C/c$.  This proves the equivalence.

\emph{Characteristic case.}  If $|\T\setminus\Omega|=0$ then for every $x$ and $r$, $|(x-r,x+r)\setminus\Omega|=0$, so $\vartheta_\Omega(r)=1$ and (i) holds.  If $|\T\setminus\Omega|>0$, let $x_*$ be a Lebesgue density point of $\T\setminus\Omega$; then $(2r)^{-1}|\Omega\cap(x_*-r,x_*+r)|\to0$, so $\vartheta_\Omega(r)\to0$ and (i) fails. \qed

\subsection{Proof of Theorem \ref{thm:B}}

\emph{(i)$\Rightarrow$(ii).}  For this implication set $r_0:=r_{\mathrm{mass}}\in(0,\pi]$.  Condition \textup{(i)} gives $\Theta_a(r)\ge cL(1/r)^{-1}>0$ for every $0<r\le r_0$, in particular at the auxiliary scale $r_0'$, supplying the full-range positivity required by Proposition~\ref{prop:upper-engine}; Lemma~\ref{lem:critical-collapse}, that proposition, and Remark~\ref{rem:invertibility} give $\|P_X(is)^{-1}\|\lesssim s^{-1}L(s)$ for $s\ge s_0$.

\emph{(ii)$\Rightarrow$(iv).}  Lemma~\ref{lem:comparison} gives $is\in\rho(\cA)$ and $\|(is-\cA)^{-1}\|\le C(1+|s|\cdot C|s|^{-1}L(|s|))\lesssim L(|s|)$.

\emph{(iv)$\Rightarrow$(ii).}  Estimate \eqref{eq:stat-gen-reverse}.

\emph{(ii)$\Rightarrow$(iii).}  For every $s\in\Sigma_Y$ above the
threshold, $Q_{s,0}$ is a direct summand of $P_X(is)$.  Hence
\eqref{eq:P-block} gives
\[
    \|Q_{s,0}^{-1}\|\le\|P_X(is)^{-1}\|
       \le Cs^{-1}L(s).
\]
Take $\cS=\Sigma_Y\cap[s_0,\infty)$, which is dyadically syndetic by
\S\ref{subsec:notation}.

\emph{(iii)$\Rightarrow$(i).}  We prove the contrapositive.  If $a=0$
a.e., then $Q_{s,0}1=0$ for every $s$, so (iii) fails, while (i) also
fails.  Thus $\int_\T a>0$ and the damping is nontrivial.  No mass scale
from condition \textup{(i)} is available in the contrapositive argument
below; the only radius restriction in the absorbing quasimode is the
harmless smallness condition $r\le\tfrac12$.  If \eqref{eq:crit-mass}
fails for every $c,r_{\mathrm{mass}}$, then
$\liminf_{r\downarrow0}\Theta_a(r)L(1/r)=0$: choose $r_k\downarrow0$ with
\[
    \eps_k:=\Theta_a(r_k)\,L(R_k)\to0,\qquad R_k=1/r_k,\ \theta_k=\Theta_a(r_k).
\]
Let $C$ be the constant in (iii), let $c,c_0$ be the constants of Proposition~\ref{prop:absorbing}, and let $C_1$ be the constant of Lemma~\ref{lem:scale-stability} for the rescaling $S\mapsto 4KS^2L(S)$, which for $R_k\ge S_1(K)$ may be taken \emph{independent of $K$}.  Fix $K\ge1$, to be chosen.  Every window $[T,4T)$ contains a full dyadic block $[2^m,2^{m+1})$ with $2^m\in[T,2T)$, so by dyadic syndeticity of $\cS$ we may choose, for $k$ large,
\[
    s^{(k)}\in\cS\cap\bigl[KR_k^2L(R_k),\,4KR_k^2L(R_k)\bigr).
\]
Then $s^{(k)}r_k^2\theta_k\le 4K\,\eps_k\to0$, so
\eqref{eq:admissible} holds for large $k$; moreover $r_k\le\tfrac12$ for
large $k$, so the radius hypothesis of Proposition~\ref{prop:absorbing}
is also satisfied.  Also
$s^{(k)}\theta_k\le 4K\eps_kR_k^2\le R_k^2$ once
$4K\eps_k\le1$.  Condition (iii) says that $Q_{s^{(k)},0}$ is invertible,
so \eqref{eq:absorbing-scalar-lower} gives
\[
    \|Q_{s^{(k)},0}^{-1}\|\ \ge\ c\bigl(R_k^2+s^{(k)}\theta_k\bigr)^{-1}\ \ge\ \frac{c}{2R_k^2}.
\]
On the other hand, (iii) and Lemma~\ref{lem:scale-stability} --- applied with $\alpha=2$, $\gamma=1$, $C_0=4K$, whose comparison constant is independent of $K$ --- give
\[
    \|Q_{s^{(k)},0}^{-1}\|\ \le\ C\,\frac{L(s^{(k)})}{s^{(k)}}
    \ \le\ C\,\frac{C_1L(R_k)}{K R_k^2L(R_k)}
    \ =\ \frac{CC_1}{KR_k^2}.
\]
Choosing $K>2CC_1/c$ (legitimate: $C_1$ does not depend on $K$) yields a contradiction for large $k$.  Hence \eqref{eq:crit-mass} holds.

\emph{Saturation sharpness.}  Assume (i)--(iv) and \eqref{eq:crit-upper-saturation}, and fix $\eps\in(0,1]$ small.  For large $j$ set $r_j=\eps\,s_j^{-1/2}L(s_j)^{1/2}$.  Since every critical gauge is subpolynomial, $r_j\to0$; increasing the threshold in $j$, we may and do assume $r_j\le\min(r_{\mathrm{mass}},r_1,\tfrac12)$.  By Lemma~\ref{lem:scale-stability}, $L(1/r_j)\asymp L(s_j)$, so \eqref{eq:crit-upper-saturation} gives $\Theta_a(r_j)\le C'L(s_j)^{-1}$, and
\[
    s_jr_j^2\Theta_a(r_j)\le C'\eps^2\le c_0
\]
for $\eps$ small: \eqref{eq:admissible} holds.  By \textup{(iv)}, $is_j\in\rho(\cA)$, so Proposition~\ref{prop:absorbing} yields
\[
    \|P_X(is_j)^{-1}\|\ \ge\ c\bigl(\eps^{-2}s_jL(s_j)^{-1}+C's_jL(s_j)^{-1}\bigr)^{-1}
    \ \ge\ c_\eps\,\frac{L(s_j)}{s_j},
\]
Here $c_\eps$ may be chosen with $c_\eps\asymp\eps^2$ as
$\eps\downarrow0$.  This and (ii) prove the first half of
\eqref{eq:sharp-transverse}.  For the generator, the upper bound is (iv)
and the lower bound follows from \eqref{eq:stat-gen-reverse}:
$\|(is_j-\cA)^{-1}\|\ge s_j\|P_X(is_j)^{-1}\|\ge c_\eps L(s_j)$.
\qed

\subsection{Proof of Corollary \ref{cor:log}}
$L(S)=(\log(eS))^A$ is critical with $\ell(u)=u^A$, and $(\log|s|)^A\asymp L(|s|)$ for $|s|\ge2$; restricting $r_0$ to be at most $1$ is harmless because Theorem~\ref{thm:B} uses only small radii in the inverse direction.  The dyadic-block formulation of the sparse stationary hypothesis is exactly dyadic syndeticity, so Theorem~\ref{thm:B} applies directly to the stationary bound.  If instead the generator bound is assumed only on a dyadically syndetic subset $\cS_0$ of the distinct transverse frequencies, with $is\in\rho(\cA)$ for $s\in\cS_0$, then the reverse comparison estimate \eqref{eq:stat-gen-reverse} gives, on the same subset,
\[
    \|P_X(is)^{-1}\|\le |s|^{-1}\|(is-\cA)^{-1}\|
    \lesssim |s|^{-1} L(|s|).
\]
By \eqref{eq:P-block}, this full stationary estimate restricts to the
resonant scalar block.  Thus
Theorem~\ref{thm:B}\textup{(iii)}$\Rightarrow$\textup{(i)} again gives
the logarithmic lower-mass condition. \qed

\section{Exact rough profiles and the power breakdown}\label{sec:profiles}

Theorem~\ref{thm:C} complements the rough dictionary in two directions.
Part (a) shows that every critical profile occurs for a
genuinely rough characteristic damping, so the equivalence of
Theorem~\ref{thm:B} is not tied to regular model wells.  Part (b) shows
that the analogous reverse implication fails for every sublinear power
gauge.  Together they establish both the abundance of the critical theory
and its precise contrast with power scales.

\subsection{Proof of Theorem \ref{thm:C}(a)}

Write $\theta(r)=L(1/r)^{-1}$ and $\theta_n=\theta(2^{-n})$; by Definition~\ref{def:dyadic-regular}, $\theta$ is nondecreasing, $\theta_n\downarrow0$, and $\theta_{n+1}\ge c_\theta\theta_n$.  Consequently, for every fixed $C>0$ there are constants $0<c_C\le C_C<\infty$ such that $c_C\theta(r)\le\theta(Cr)\le C_C\theta(r)$ for all sufficiently small $r$ for which both radii are in $(0,1)$; one moves only finitely many dyadic levels.  All intervals below are arcs in $\T$ of the stated lengths; harmless factors of $2\pi$ are absorbed into these constants.

\emph{Selection of generations.}  Let $N_1$ be large --- large enough that the dyadic-regularity estimates below hold for all $n\ge N_1$, that $2^{-N_1}$ and every fixed-factor radius used below lie in the domain on which $\theta$ is being compared, that $\theta_{N_1}\le1$, and that all small-radius comparisons below occur beyond this threshold --- and define inductively $N_{k+1}$ as the least integer with $\theta_{N_{k+1}}\le\theta_{N_k}/2$.  Then $\theta_{N_k}\le2^{-(k-1)}\theta_{N_1}$, so
\[
    \sum_k\theta_{N_k}\le2\theta_{N_1}<\infty,
\]
and by minimality $\theta_{N_{k+1}}\ge c_\theta\,\theta_{N_{k+1}-1}>c_\theta\theta_{N_k}/2$, so
\begin{equation}\label{eq:generation-comparability}
    \frac{c_\theta}{2}\,\theta_{N_k}\ \le\ \theta_{N_{k+1}}\ \le\ \frac12\,\theta_{N_k}.
\end{equation}

\emph{Construction.}  For $\delta\in(0,\tfrac14]$ let $\Omega=\Omega_\delta$ be the union, over all $k$ and over all dyadic cells $C$ of generation $N_k$ (arcs of length $h_{N_k}=2\pi\,2^{-N_k}$ with the nested dyadic endpoints $2\pi m2^{-N_k}$, $0\le m<2^{N_k}$; in particular $0$ is an endpoint of every generation, as the upper-bound witness below uses), of the open arc of length $\delta h_{N_k}\theta_{N_k}$ centered at the center of $C$.  Since $\theta_{N_k}\le1$ and $\delta\le1/4$, every added arc is strictly shorter than its cell.  Then $\Omega$ is open and
\[
    |\Omega|\le\sum_k2^{N_k}\cdot\delta h_{N_k}\theta_{N_k}=2\pi\delta\sum_k\theta_{N_k}<\eta
\]
for $\delta$ small.  Since $h_{N_k}\to0$, every open arc of $\T$ contains a full dyadic cell of generation $N_k$ for all sufficiently large selected $k$, hence one of the added central arcs; so $\Omega$ is dense, and $\T\setminus\Omega$ is closed and nowhere dense, of measure $\ge2\pi-\min(\eta,\pi)\ge\pi>0$ once $\delta$ is chosen so small that $|\Omega|<\min(\eta,\pi)$.

\emph{Lower bound.}  Let $J=(x-r,x+r)$ with $r$ small, and let $k$ be minimal with $h_{N_k}\le r/2$ (for $r$ small, $k\ge2$).  Then $J$ contains at least $\lfloor 2r/h_{N_k}\rfloor-1\ge r/h_{N_k}$ full cells of generation $N_k$, each contributing its central arc in full, so
\[
    |\Omega\cap J|\ \ge\ \frac{r}{h_{N_k}}\cdot\delta h_{N_k}\theta_{N_k}=\delta\,r\,\theta_{N_k}.
\]
After decreasing the final small-radius threshold so that all dyadic indices used are beyond the dyadic-regularity threshold, by minimality, $h_{N_{k-1}}>r/2$, i.e.\ $2^{-N_{k-1}}>r/(4\pi)$, so $\theta_{N_{k-1}}\ge\theta(r/(4\pi))\ge\theta(r/16)\ge c_\theta^5\,\theta(r)=:c'\theta(r)$ --- bracketing $2^{-n-1}<r\le2^{-n}$, monotonicity and five applications of $\theta_{m+1}\ge c_\theta\theta_m$ give $\theta(r/16)\ge\theta(2^{-n-5})\ge c_\theta^5\theta(2^{-n})\ge c_\theta^5\theta(r)$; with \eqref{eq:generation-comparability}, $\theta_{N_k}\ge(c_\theta/2)\theta_{N_{k-1}}\ge c''\theta(r)$.  Hence, for every center $x$,
\[
    \frac{|\Omega\cap (x-r,x+r)|}{2r}\ \ge\ \frac{\delta c''}{2}\,\theta(r).
\]
Taking the infimum over $x\in\T$ gives
\[
    \vartheta_\Omega(r)\ge \frac{\delta c''}{2}\,\theta(r).
\]

\emph{Upper bound.}  Take the witness $x_*=0$, a dyadic endpoint of every generation, and $J=(-r,r)$.  If $h_{N_m}>4r$, the two cells of generation $N_m$ adjacent to $0$ have centers at distance $h_{N_m}/2>2r$ from $0$, while their central arcs have half-length $\le\delta h_{N_m}/2\le h_{N_m}/8$; since $h_{N_m}/2-h_{N_m}/8>r$, no central arc of generation $N_m$ meets $J$.  If $h_{N_m}\le4r$, at most $2r/h_{N_m}+2\le10r/h_{N_m}$ cells of generation $N_m$ meet $J$, contributing at most $10\delta r\theta_{N_m}$.  Using subadditivity and ignoring overlaps between generations, let $k$ be minimal with $h_{N_k}\le4r$; then $2^{-N_k}=h_{N_k}/(2\pi)\le 2r/\pi<r$, so $\theta_{N_k}\le\theta(r)$, and by \eqref{eq:generation-comparability},
\[
    |\Omega\cap J|\ \le\ 10\delta r\sum_{m\ge k}\theta_{N_m}\ \le\ 20\delta r\,\theta_{N_k}\ \le\ 20\delta r\,\theta(r),
\]
whence $\vartheta_\Omega(r)\le10\delta\,\theta(r)$.  Together, $\vartheta_\Omega(r)\asymp\theta(r)=L(1/r)^{-1}$, which is \eqref{eq:exact-profile-theorem}.

\emph{Resolvent statements.}  Two-sided saturation holds, so Theorem~\ref{thm:B} gives \eqref{eq:crit-stat}, \eqref{eq:crit-gen}, and the attainment \eqref{eq:sharp-transverse} along the transverse eigenfrequencies.  Since $\vartheta_\Omega(r)\to0$, Theorem~\ref{thm:A} shows exponential stabilization fails. \qed

\begin{lemma}[Centered arcs minimize radial averages]\label{lem:centered-arc-min}
Let $0<r<\pi/2$ and let $F:[0,\pi]\to[0,\infty)$ be nondecreasing.  Put $f(x)=F(\operatorname{dist}_{\T}(x,0))$.  Among all arcs $I\subset\T$ of length $2r$, the integral $\int_I f$ is minimized by the centered arc $(-r,r)$.
\end{lemma}

\begin{proof}
For $\alpha\ge0$, since $F$ is nondecreasing, $\{t\in[0,\pi]:F(t)\le\alpha\}$ is an interval with left endpoint $0$; with $\rho_\alpha$ its supremum ($\rho_\alpha:=0$ if the set is empty), the superlevel set $\{f>\alpha\}$ is, up to a null set, the complement of the centered arc $[-\rho_\alpha,\rho_\alpha]$.  For every arc $I$ of length $2r$ and every $t\in[0,\pi]$,
\[
    |I\cap[-t,t]|\le \min(2r,2t)=|(-r,r)\cap[-t,t]|,
\]
because a centered arc is the translate of length $2r$ with maximal overlap with a centered arc of length $2t$ on the circle; hence $|I\cap\{f>\alpha\}|\ge2r-\min(2r,2\rho_\alpha)=|(-r,r)\cap\{f>\alpha\}|$ up to null sets.  By the layer-cake formula,
\[
    \int_If=\int_0^\infty\bigl|I\cap\{f>\alpha\}\bigr|\,d\alpha\ \ge\ \int_0^\infty\bigl|(-r,r)\cap\{f>\alpha\}\bigr|\,d\alpha=\int_{-r}^rf .\qedhere
\]
\end{proof}

\subsection{Proof of Theorem \ref{thm:C}(b)}

\emph{Power collapse.}  We first record: if $\Theta_a(r)\ge c\,r^\gamma$ for $0<r\le r_0$ with $0<\gamma<2$, then
\begin{equation}\label{eq:power-collapse}
    \cU_a(s)\ \lesssim\ s^{-\frac{2-\gamma}{2+\gamma}} .
\end{equation}
Indeed, set $t_*=s^{-2/(2+\gamma)}$ and $r_*=s^{-1/(2+\gamma)}\ge t_*$.  Since $r_*\to0$, for $s$ large we have $r_*\le r_0$, so every radius used below lies in the range of the mass hypothesis.  For $t\le t_*$, $st^2\le s^{-(2-\gamma)/(2+\gamma)}$.  For $t_*\le t\le r_*$, take $r=t$: $t^2+s^{-1}c^{-1}t^{-\gamma}\le r_*^2+Cs^{-1}t_*^{-\gamma}\lesssim s^{-(2-\gamma)/(2+\gamma)}$, using $r_*^2=s^{-2/(2+\gamma)}\le s^{-(2-\gamma)/(2+\gamma)}$.  For $t\ge r_*$, take $r=r_*$: $r_*^2+Cs^{-1}r_*^{-\gamma}\lesssim s^{-2/(2+\gamma)}$.  This proves \eqref{eq:power-collapse}.

\emph{Failure at the power gauge $L_\beta(S):=S^\beta$.}  Fix $\beta<\beta'<2\beta/(2-\beta)$ (a nonempty range for $0<\beta<1$, with $\beta'<2$) and let $a(x)=\dist_\T(x,0)^{\beta'}$, a continuous damping.  For $0<r<\pi/2$, Lemma~\ref{lem:centered-arc-min} shows that the radial nondecreasing function $x\mapsto\dist_\T(x,0)^{\beta'}$ has, among arcs of length $2r$, minimal average on the arc centered at $0$.  Hence
\[
    \Theta_a(r)=(2r)^{-1}\int_{-r}^r|y|^{\beta'}\,dy\asymp r^{\beta'},
\]
so $\Theta_a(r)r^{-\beta}\to0$: the local-mass bound at gauge $L_\beta$ fails.  On the other hand \eqref{eq:power-collapse} with $\gamma=\beta'$, Proposition~\ref{prop:upper-engine}, and Remark~\ref{rem:invertibility} give
\[
    \|P_X(is)^{-1}\|\ \lesssim\ s^{-\frac{2-\beta'}{2+\beta'}}\ \le\ s^{-(1-\beta)}=s^{-1}L_\beta(s),
\]
since $(2-\beta')/(2+\beta')\ge1-\beta$ exactly when $\beta'\le2\beta/(2-\beta)$.  Thus the asserted bound holds for all sufficiently large positive $s$.  If $C$ denotes pointwise conjugation, then
$P_X(-is)=CP_X(is)C$; invertibility and the same norm bound therefore hold
for every sufficiently large $|s|$.  So the implication fails.

\emph{The window.}  Let $0\le b\in L^\infty(\T)$ satisfy
$\Theta_b(r)\asymp r^\gamma$ as $r\downarrow0$, where $0<\gamma<2$.
Choose $c,C>0$ and $r_0\in(0,\pi]$ so that
$c\,r^\gamma\le\Theta_b(r)\le C\,r^\gamma$ for all $0<r\le r_0$; in
particular $\Theta_b>0$ on this range, supplying the hypotheses of
Proposition~\ref{prop:upper-engine} with the damping replaced by $b$.
By \eqref{eq:power-collapse}, Proposition~\ref{prop:upper-engine}, and
Remark~\ref{rem:invertibility}, all with that replacement,
$P_{X,b}(is)$ is invertible and
\[
 \|P_{X,b}(is)^{-1}\|\lesssim s^{-(2-\gamma)/(2+\gamma)},
 \qquad s\ge s_0.
\]
Pointwise conjugation gives the same conclusion with $s$ replaced by
$|s|$ for every $|s|\ge s_0$.  Lemma~\ref{lem:comparison}, with $a$
replaced by $b$, then gives $is\in\rho(\cA_b)$ there and
\[
 \|(is-\cA_b)^{-1}\|
   \lesssim1+|s|\,|s|^{-(2-\gamma)/(2+\gamma)}
   \asymp |s|^{2\gamma/(2+\gamma)}.
\]
This is the upper bound in \eqref{eq:power-window} along the positive
transverse frequencies.  At $z=0$, $\cA_b U=0$ forces $v$ constant,
$\Delta_Xu=vb$, and $\int_Xb>0$ gives $U=0$; with compact resolvent,
$0\in\rho(\cA_b)$.  For $0<|s|\le s_0$, the inequality
$\int_\T b\ge2r_0\Theta_b(r_0)>0$ and
Lemma~\ref{lem:axis-cleanness} give $is\in\rho(\cA_b)$ on the bounded part
of the axis as well.  For the lower bound, set
$r_j=\kappa s_j^{-1/(2+\gamma)}$ with $\kappa$ small.  Since $r_j\to0$,
for all large $j$ we have $r_j\le\min(r_0,\tfrac12)$, so
Proposition~\ref{prop:absorbing}, applied with $b$ in place of $a$, is
applicable; moreover
$s_jr_j^2\Theta_b(r_j)\asymp\kappa^{2+\gamma}\le c_0$, and that
proposition gives
\[
    \|P_{X,b}(is_j)^{-1}\|\ \ge\ c\bigl(\kappa^{-2}s_j^{2/(2+\gamma)}+C\kappa^\gamma s_j^{2/(2+\gamma)}\bigr)^{-1}
    \ \ge\ c_\kappa\,s_j^{-2/(2+\gamma)} ,
\]
whence by \eqref{eq:stat-gen-reverse}, applied to $b$,
$\|(is_j-\cA_b)^{-1}\|\ge s_j\|P_{X,b}(is_j)^{-1}\|
\ge c_\kappa s_j^{\gamma/(2+\gamma)}$. \qed

\begin{remark}
The window \eqref{eq:power-window} is what local mass alone yields at
power scales, and part (b) shows that the loss of an inverse theorem there
is genuine.  We do not claim that either endpoint is attained for every
profile: the sharp exponent inside the window can depend on the
arrangement of the damping, consistently with the structured power-well
theory \cite{Kleinhenz2019,DK2020,DKP,AK}.  At critical gauges no such
finer datum is needed for the resolvent bounds of Theorem~\ref{thm:B}.
This is the precise critical-versus-power boundary established here.
\end{remark}

This completes the rough inverse theorem.  At the endpoint and throughout
the critical logarithmic/doubling class, $\Theta_a$ is the exact datum for
the resolvent estimates considered here.  Theorem~\ref{thm:C} realizes
every critical profile by genuinely rough characteristic damping and shows
that the reverse implication already fails for every sublinear power
gauge.  We now prove that the canonical logarithmic scale is carried by
true eigenvalues.

\part{The logarithmic spectral ladder}\label{part:model}

\section{From Schr\"odinger blocks to damped-wave spectrum}\label{sec:pencil}

The rough inverse theorem predicts the size of the logarithmic obstruction but cannot
decide whether it is spectral.  We first isolate the exact bridge needed
to answer that question.  A low, tame, algebraically simple eigenvalue of
$T_s=-\partial_x^2+isa$ must be transferred through the quadratic
damped-wave pencil without losing its algebraic count.  The results of
this section provide both the isolated-rung and finite-window versions of
that transfer; Sections~\ref{sec:oscillator} and~\ref{sec:log-resonance}
then construct the required $T_s$ ladder.

\subsection{The exact pencil reduction}
Fix a transverse eigenfrequency $s=s_j$, so $\lambda_j=s^2$ in \eqref{eq:pencil-block}.  Writing $z=is+w$,
\[
    s^2+za+z^2=isa+wa+2isw+w^2,
\]
so $z$ is a root of the block pencil if and only if
\begin{equation}\label{eq:pencil-Ts}
    \bigl(T_s+wa+w^2\bigr)v=-2isw\,v
\end{equation}
for some nontrivial $v\in H^2(\T)$, with $T_s=-\partial_x^2+isa$ as in
\eqref{eq:Ts}.  Set $\zeta=-2isw$.  Then
\begin{equation}\label{eq:depth-map}
    w=\frac{i\zeta}{2s},\qquad -\Re z=-\Re w=\frac{\Im\zeta}{2s},\qquad \Im z-s=\Im w=\frac{\Re\zeta}{2s}.
\end{equation}
After substituting $w=i\zeta/(2s)$, the exact characteristic equation is
\begin{equation}\label{eq:pencil-zeta}
    \left(T_s+\frac{i\zeta}{2s}a-\frac{\zeta^2}{4s^2}\right)v=\zeta v .
\end{equation}
Thus, for a characteristic value $\zeta$ of this self-consistent pencil, the relation
\[
    -\Re z=\frac{\Im\zeta}{2s}
\]
is exact.  To produce a damped-wave root at the critical depth $-\Re z\asymp L(s)^{-1}$ predicted by Theorem~\ref{thm:B}, the unperturbed spectral target is therefore a genuine low eigenvalue $\lambda_s\in\spec(T_s)$ whose imaginary part lies at the height
\begin{equation}\label{eq:lambda-imag-needed}
    \Im\lambda_s\ \asymp\ \frac{s}{L(s)} .
\end{equation}
The existence of such an eigenvalue is not by itself a pencil-root statement: one must also control the self-consistent perturbation in \eqref{eq:pencil-zeta}.  Accordingly, this section isolates the two obstructions to carrying out that conversion: \emph{existence} of a low cluster at the height \eqref{eq:lambda-imag-needed}, and \emph{phase rigidity} of that cluster.  Both are then verified unconditionally for the logarithmic cusp in Sections~\ref{sec:oscillator}--\ref{sec:log-resonance}.

\subsection{Spectral calibration by local mass}

The reduction above converts the imaginary height of a Schr\"odinger
eigenvalue into the depth of a wave eigenvalue.  The next proposition
calibrates both quantities directly by the damping mass.  Its first part
shows that the rough inverse theorem forces the critical spectral scale; its second part
identifies the depth of every genuine high-frequency wave eigenvalue
exactly.

\begin{proposition}[Spectral calibration]\label{prop:spectral-calibration}
Let $0\le a\in L^\infty(\T)$.
\begin{enumerate}[label=\textup{(\alph*)},leftmargin=*]
\item Let $s>0$.  If $T_sv=\lambda v$ for some
$0\ne v\in H^2(\T)$, then
\begin{equation}\label{eq:scalar-eigen-identities}
 \Re\lambda=\frac{\|v'\|_2^2}{\|v\|_2^2},
 \qquad
 \Im\lambda=s\frac{\int_\T a|v|^2}{\|v\|_2^2}.
\end{equation}
If $0<r\le\pi$, $\Theta_a(r)>0$, and
$2C_Pr^2\Re\lambda\le1$, where $C_P$ is the universal constant in
\eqref{eq:periodic-wp}, then
\begin{equation}\label{eq:scalar-lower-wall}
 \Im\lambda\ge \frac{s\Theta_a(r)}{2C_P}.
\end{equation}
Consequently, if $L$ is critical and the lower-mass bound
\eqref{eq:crit-mass} holds, there are $c_0,s_0>0$ such that
\begin{equation}\label{eq:critical-spectral-wall}
 \lambda\in\spec(T_s),\quad \Re\lambda\le s,\quad s\ge s_0
 \qquad\Longrightarrow\qquad
 \Im\lambda\ge c_0\frac{s}{L(s)}.
\end{equation}

\item Let $s_j>\frac12\|a\|_\infty$, and suppose that
$0\ne u\in H^2(\T)$ and $z\in\C$ solve the exact block equation
\[
 -u''+(s_j^2+za+z^2)u=0.
\]
Then
\begin{equation}\label{eq:wave-depth-identity}
 \Re z=-\frac12\frac{\int_\T a|u|^2}{\|u\|_2^2},
 \qquad
 (\Im z)^2=s_j^2+\frac{\|u'\|_2^2}{\|u\|_2^2}
 -\frac14\left(\frac{\int_\T a|u|^2}{\|u\|_2^2}\right)^{\!2}.
\end{equation}
In particular $-\frac12\|a\|_\infty\le\Re z\le0$, and, for every
$c>0$,
\begin{equation}\label{eq:wave-depth-localization}
 \frac{\int_{\{a\ge c\}}|u|^2}{\|u\|_2^2}
 \le \frac{2(-\Re z)}{c}.
\end{equation}
\end{enumerate}
\end{proposition}

\begin{proof}
Pairing $T_sv=\lambda v$ with $v$ and taking real and imaginary parts
gives \eqref{eq:scalar-eigen-identities}.  Applying
\eqref{eq:periodic-wp} to $v$ and dividing by $\|v\|_2^2$ yields
\[
 1\le C_P\Theta_a(r)^{-1}\frac{\Im\lambda}{s}
       +C_Pr^2\Re\lambda.
\]
The assumed bound on the second term gives
\eqref{eq:scalar-lower-wall}.  Under \eqref{eq:crit-mass}, take
$r_s=(2C_Ps)^{-1/2}$.  For all sufficiently large $s$ one has
$r_s\le r_{\mathrm{mass}}$, and if $\Re\lambda\le s$, then
$2C_Pr_s^2\Re\lambda\le1$.  Hence
\[
 \Im\lambda\ge
 \frac{c s}{2C_P L((2C_Ps)^{1/2})}
 \ge c_0\frac{s}{L(s)},
\]
where the last inequality follows from Lemma~\ref{lem:scale-stability}.

For the block pencil, pairing with $u$ and dividing by $\|u\|_2^2$
gives
\[
 z^2+\beta z+K=0,
 \qquad
 \beta=\frac{\int_\T a|u|^2}{\|u\|_2^2},
 \qquad
 K=s_j^2+\frac{\|u'\|_2^2}{\|u\|_2^2}.
\]
Here $0\le\beta\le\|a\|_\infty$ and
$4K\ge4s_j^2>\|a\|_\infty^2\ge\beta^2$.  Thus the scalar quadratic
has negative discriminant, and its two roots are
$-\beta/2\pm i(K-\beta^2/4)^{1/2}$.  This proves
\eqref{eq:wave-depth-identity}.  Finally,
\[
 2(-\Re z)=\beta
 \ge c\frac{\int_{\{a\ge c\}}|u|^2}{\|u\|_2^2},
\]
which is \eqref{eq:wave-depth-localization}.
\end{proof}

\begin{remark}[Floquet uniformity]\label{rem:spectral-calibration-floquet}
The proposition remains valid, with the same constants, on every
Floquet domain
\[
 \begin{aligned}
 H^2_\theta:=\{v\in H^2(0,2\pi):{}&
 v(2\pi)=e^{i\theta}v(0),\\[-2pt]
 &v'(2\pi)=e^{i\theta}v'(0)\},
 \qquad \theta\in\R/2\pi\Z.
 \end{aligned}
\]
This is the $k=2$ instance of the notation used in
Section~\ref{sec:mapping-torus} before Theorem~\ref{thm:F}.
Indeed, the sesquilinear boundary term cancels, and the weighted
Poincar\'e inequality applies to the quasi-periodic extension, whose
modulus is periodic.
\end{remark}

The wall \eqref{eq:critical-spectral-wall} is a forced scale, not an
existence theorem.  Through \eqref{eq:depth-map}, height $s/L(s)$
corresponds to wave depth of order $L(s)^{-1}$.  For the logarithmic
cusp, Sections~\ref{sec:oscillator}--\ref{sec:log-resonance} construct the
fixed-window ladder at that scale; the exact identity
\eqref{eq:wave-depth-identity} also explains the structural factor $1/2$
in the leading depth of Theorem~\ref{thm:E}.

\subsection{Complex symmetry and non-isotropy}
\begin{lemma}[Complex symmetry]\label{lem:complex-symmetry}
Let $\mathbb S=\R/\ell\Z$ be a circle of arbitrary circumference $\ell>0$, let $V\in L^\infty(\mathbb S)$ be real-valued, and let $\cT=-\partial_y^2+iV$ on $\mathcal D(\cT)=H^2(\mathbb S)$ --- a closed operator with compact resolvent, being a bounded perturbation of the self-adjoint $-\partial_y^2$.  Let $C$ denote pointwise conjugation, $Cu=\bar u$.  In this section $\cT=T_s$ on $\mathbb S=\T$, i.e.\ $V=sa$.  The rescaled-circle operators introduced later in Section~\ref{sec:log-resonance} are instances of this same template after their notation is defined there.  Let $\mathfrak b(u,v)=\int uv$ \textup{(}bilinear\textup{)} and, for $v\ne0$,
\begin{equation}\label{eq:delta}
    \delta(v)=\frac{\bigl|\int v^2\bigr|}{\int|v|^2}\in[0,1],
\end{equation}
the integrals over the underlying one-dimensional domain --- the same two definitions are used verbatim on $\R$ in Sections~\ref{sec:oscillator}--\ref{sec:log-resonance}.  The quantity $\delta(0)$ is not defined and is never used.  Recall from \S\ref{subsec:notation} that $\langle\cdot,\cdot\rangle$ is linear in the first argument, so $\langle u,\bar v\rangle=\mathfrak b(u,v)$.
\begin{enumerate}[label=\textup{(\roman*)}]
\item One has $\mathcal D(\cT^*)=H^2(\mathbb S)$, $C\mathcal D(\cT^*)=\mathcal D(\cT)$, and $\cT=C\cT^*C$.  Thus $\cT$ is complex symmetric in the closed-operator sense above.  In particular,
\[
    \mathfrak b(\cT u,v)=\mathfrak b(u,\cT v),
    \qquad u,v\in H^2(\mathbb S).
\]
\item If $\lambda_0$ is an isolated eigenvalue of $\cT$ of algebraic multiplicity one, with eigenfunction $v_0$ and Riesz projection $\Pi_0$, then $\mathfrak b(v_0,v_0)\neq0$,
\begin{equation}\label{eq:projection-formula}
    \Pi_0u=\frac{\mathfrak b(u,v_0)}{\mathfrak b(v_0,v_0)}\,v_0,
    \qquad
    \|\Pi_0\|=\delta(v_0)^{-1}.
\end{equation}
\item If $(\cT-\lambda_0)w=v_0$ has a solution $w\in H^2(\mathbb S)$, then $\mathfrak b(v_0,v_0)=0$.  Hence a geometrically simple eigenvalue whose eigenfunction satisfies $\mathfrak b(v_0,v_0)\neq0$ is algebraically simple.
\end{enumerate}
\end{lemma}

\begin{proof}
(i) Since $V$ is real and bounded, $\cT^*=-\partial_y^2-iV$ on $H^2(\mathbb S)$.  Pointwise conjugation preserves $H^2(\mathbb S)$ and gives $C\cT^*Cu=-u''+iVu=\cT u$ for every $u\in H^2(\mathbb S)$, proving the asserted domain identity and closed-operator complex symmetry.  The bilinear identity follows either from this operator identity or directly by integrating by parts twice on $\mathbb S$; there are no boundary terms, and multiplication by $V$ is bilinearly symmetric.  (ii) Since $V$ is real, $\cT^*u=\overline{\cT\bar u}$, so $\cT^*\bar v_0=\bar\lambda_0\bar v_0$.  The Riesz projection of $\cT$ at $\lambda_0$ has rank one, and Lemma~\ref{lem:riesz-tools}\textup{(ii)} identifies its adjoint with the Riesz projection of $\cT^*$ at $\bar\lambda_0$; hence $\bar\lambda_0$ is algebraically and geometrically simple for $\cT^*$, so $\bar v_0$ spans $\ker(\cT^*-\bar\lambda_0)$.  A rank-one Riesz projection has the form $\Pi_0u=\langle u,w\rangle v_0$ with $w\in\ker(\cT^*-\bar\lambda_0)$ and the normalization is fixed by $\Pi_0v_0=v_0$, i.e. $\langle v_0,w\rangle=1$; writing $w=c\,\bar v_0$ gives $\bar c\,\langle v_0,\bar v_0\rangle=1$ (the inner product is conjugate-linear in its second argument), so $\langle v_0,\bar v_0\rangle=\int v_0^2=\mathfrak b(v_0,v_0)\neq0$ and $\Pi_0=\langle\cdot,\bar v_0\rangle v_0/\langle v_0,\bar v_0\rangle$.  This is \eqref{eq:projection-formula}, and $\|\Pi_0\|=\|v_0\|\,\|\bar v_0\|/|\mathfrak b(v_0,v_0)|=\delta(v_0)^{-1}$.  (iii) $\mathfrak b(v_0,v_0)=\mathfrak b((\cT-\lambda_0)w,v_0)=\mathfrak b(w,(\cT-\lambda_0)v_0)=0$ by (i).
\end{proof}

\subsection{Pencil transfer}\label{subsec:pencil-transfer}

The next statements are formulated for an arbitrary real bounded damping
$d\in L^\infty(\T)$, since they are used both for the cusp and for its fixed
perturbations.  Put
\[
    T_s^d=-\partial_x^2+isd,\qquad
    P_d(z):=P_{d,j}(z)=-\partial_x^2+\lambda_j+zd+z^2,
\]
and let $\cA_{d,j}$ be the corresponding transverse block generator.
Constants depend on $d$ only through $\|d\|_\infty$.
\begin{lemma}[Companion linearization]\label{lem:companion}
Fix $j\ge1$ and a real bounded damping $d\in L^\infty(\T)$.  A point $z_0\in\C$ is a characteristic value of the pencil $P_d(z):H^2(\T)\to L^2(\T)$ if and only if $z_0\in\spec(\cA_{d,j})$, and the Gohberg--Sigal multiplicity of $z_0$ for $P_d$ equals the algebraic multiplicity of $z_0$ as an eigenvalue of $\cA_{d,j}$.
\end{lemma}

\begin{proof}
The pencil $P_d(z):H^2(\T)\to L^2(\T)$ is a holomorphic Fredholm family of index zero on the fixed domain $H^2(\T)$, since it is a bounded lower-order perturbation of $-\partial_x^2+1$ after subtracting a scalar.  On the block, $\cA_{d,j}(u,v)=(v,\,(\partial_x^2-\lambda_j)u-dv)$ with domain $H^2(\T)\times H^1(\T)$, so $\cA_{d,j}$ is the companion operator of $P_d$.  Let $(u_k,v_k)_{0\le k\le m-1}$ be a Jordan chain of $\cA_{d,j}$ at $z_0$, i.e.\ $(\cA_{d,j}-z_0)(u_k,v_k)=(u_{k-1},v_{k-1})$ with $(u_{-1},v_{-1})=(0,0)$; we also set $u_{-2}:=0$, so that the displayed relation below is meaningful for $k=0,1$.  The first components give $v_k=z_0u_k+u_{k-1}$; substituting into the second components gives, for $0\le k\le m-1$,
\[
    P_d(z_0)u_k+P_d'(z_0)u_{k-1}+\tfrac12P_d''(z_0)u_{k-2}=0,
    \qquad P_d'(z)=d+2z,\quad P_d''(z)=2,
\]
which is exactly the Keldysh chain condition for $P_d$ at $z_0$.  Conversely, a Keldysh chain $(u_k)_{0\le k\le m-1}$ of $P_d$ defines a Jordan chain of $\cA_{d,j}$ of the same length via $v_k:=z_0u_k+u_{k-1}$; and $u_0\neq0$ in either direction (if $u_0=0$ then $v_0=z_0u_0=0$).  The correspondence is a length-preserving bijection between Jordan chains of $\cA_{d,j}$ and Keldysh chains of $P_d$ at $z_0$ --- all chain vectors lying in the fixed domains $\mathcal D(\cA_{d,j})$ and $H^2(\T)$ respectively --- so the algebraic multiplicity of $z_0$ for $\cA_{d,j}$ equals the sum of the lengths of a canonical system of Keldysh chains, which is the Gohberg--Sigal multiplicity \cite{GohbergSigal}.
\end{proof}

\begin{proposition}[Pencil transfer]\label{prop:pencil-transfer}
Fix a real bounded damping $d\in L^\infty(\T)$ and a transverse eigenfrequency $s=s_j\ge1$.  Let $\lambda_s$ be an eigenvalue of $T_s^d$ of algebraic multiplicity one, with $|\lambda_s|\le\Lambda_s$, and let $\rho_s>0$ and $C_t>0$ be such that the closed disc $\overline D(\lambda_s,\rho_s)$ meets $\spec(T_s^d)$ only at $\lambda_s$ and the \emph{tameness bound}
\begin{equation}\label{eq:tameness}
    \sup_{|\zeta-\lambda_s|=\rho_s}\bigl\|(T_s^d-\zeta)^{-1}\bigr\|\ \le\ \frac{C_t}{\rho_s}
\end{equation}
holds.  Assume $(\Lambda_s+\rho_s)\le2s$ and
\begin{equation}\label{eq:transfer-smallness}
    \eps_s:=C_t(\|d\|_\infty+1)\,\frac{\Lambda_s+\rho_s}{2s\rho_s}\le\frac14,
    \qquad
    4C_t(\Lambda_s+\rho_s)\,\eps_s\le\frac{\rho_s}{2}.
\end{equation}
Then the block pencil $P_d(z)$ at $\lambda_j=s^2$ has exactly one root $z$ satisfying $-2is(z-is)\in D(\lambda_s,\rho_s)$, counted with Gohberg--Sigal multiplicity one, and it satisfies
\begin{equation}\label{eq:resonance-inclusion}
    z=is+\frac{i\lambda_s}{2s}+O\Bigl(\frac{C_t(\Lambda_s+\rho_s)\,\eps_s}{s}\Bigr).
\end{equation}
The count is the Gohberg--Sigal count for the holomorphic Fredholm family $G(\zeta):H^2(\T)\to L^2(\T)$ on the fixed graph-norm domain of $T_s^d$, after the affine change $z=is+i\zeta/(2s)$; multiplicities are preserved by this biholomorphic reparametrization.  By Lemma~\ref{lem:companion}, this root is an algebraically simple eigenvalue of the block generator $\cA_{d,j}$.
\end{proposition}

\begin{proof}
For $\zeta\in\overline D(\lambda_s,\rho_s)$ set $w(\zeta)=i\zeta/(2s)$ and $\mathcal E_s^d(\zeta)=w(\zeta)d+w(\zeta)^2$, a holomorphic family with
\[
    \|\mathcal E_s^d(\zeta)\|\le|w(\zeta)|\bigl(\|d\|_\infty+|w(\zeta)|\bigr)
    \le\frac{(\Lambda_s+\rho_s)(\|d\|_\infty+1)}{2s}=:e_s ,
\]
using $|w(\zeta)|\le(\Lambda_s+\rho_s)/(2s)\le1$.  By \eqref{eq:pencil-Ts}, with $a$ replaced by $d$, $z=is+w(\zeta)$ is a root of $P_d$ if and only if $G(\zeta):=T_s^d-\zeta+\mathcal E_s^d(\zeta)$ is singular.  On the contour $\Gamma_s=\{|\zeta-\lambda_s|=\rho_s\}$, \eqref{eq:tameness} gives $\|(T_s^d-\zeta)^{-1}\mathcal E_s^d(\zeta)\|\le C_te_s/\rho_s=\eps_s\le\frac14$, so $G_\tau(\zeta)=T_s^d-\zeta+\tau \mathcal E_s^d(\zeta)$ is invertible on $\Gamma_s$ for all $\tau\in[0,1]$.  We verify the hypotheses of Lemma~\ref{lem:GS-input} for this homotopy.  The families are closed operators with fixed domain $H^2(\T)$ and compact resolvent.  More explicitly, put $\mathcal X_s^d:=H^2(\T)$ with graph norm $\|u\|_{\mathcal X_s^d}:=\|u\|_2+\|T_s^du\|_2$; because $d\in L^\infty$ and the principal part is $-\partial_x^2$, this norm is equivalent to the usual $H^2$ norm, with constants depending on $s$ but fixed throughout the present proposition --- in particular independent of $\zeta$ and $\tau$, so $\mathcal X_s^d$ is precisely the fixed Banach domain required by Lemma~\ref{lem:GS-input}.  For each fixed $s$ this supplies the fixed-domain Fredholm family; no uniform equivalence of these graph norms as $s\to\infty$ is used in the Gohberg--Sigal count, whose later uniformity comes only from the explicit contour estimates.  To see that $-1\in\rho(T_s^d)$, note that $T_s^d+1:\mathcal X_s^d\to L^2(\T)$ is Fredholm of index zero: relative to $-\partial_x^2+1$, the perturbation is $is\,d(-\partial_x^2+1)^{-1}$, compact on $L^2$ because $(-\partial_x^2+1)^{-1}:L^2\to H^2\Subset L^2$ and multiplication by $d\in L^\infty$ is bounded.  It is injective by the real-part identity $\Re\langle(T_s^d+1)u,u\rangle=\|u'\|_2^2+\|u\|_2^2$; hence it is surjective.  Thus $J:=(T_s^d+1)^{-1}$ is a bounded isomorphism from $L^2(\T)$ onto $\mathcal X_s^d$, and right-multiplication gives the bounded holomorphic Fredholm family
\[
    G_\tau(\zeta)J=I+(\tau \mathcal E_s^d(\zeta)-\zeta-1)(T_s^d+1)^{-1},
\]
of the form identity plus compact.  Characteristic values and multiplicities are unchanged because $J$ is constant in $\zeta$ and bijective.  Explicitly, a root function $u(\zeta)\in\mathcal X_s^d$ for $G_\tau$ corresponds to the root function $w(\zeta)=(T_s^d+1)u(\zeta)\in L^2(\T)$ for $G_\tau J$, and conversely $w(\zeta)$ corresponds to $Jw(\zeta)$; multiplication by this fixed isomorphism preserves the order of vanishing of $G_\tau(\zeta)u(\zeta)$ and therefore the Keldysh-chain lengths.  Since this family is invertible on the contour for every $\tau$, the total multiplicity of characteristic values enclosed is constant in $\tau$ --- the number of characteristic values of $G$ inside $\Gamma_s$, counted with multiplicity, equals that of $T_s^d-\zeta$, namely one.

To localize it, fix $\zeta\in\overline D(\lambda_s,\rho_s)$.  The same Neumann series shows that $T_s^d+\mathcal E_s^d(\zeta)$ has exactly one eigenvalue $\lambda(\zeta)$ in $\overline D(\lambda_s,\rho_s)$, algebraically simple, with rank-one Riesz projection $P_\zeta=\frac1{2\pi i}\oint_{\Gamma_s}(\zeta'-T_s^d-\mathcal E_s^d(\zeta))^{-1}d\zeta'$; the series keeps $\Gamma_s$ in the resolvent set of $T_s^d+\mathcal E_s^d(\zeta)$ uniformly for $\zeta\in\overline D(\lambda_s,\rho_s)$, and $\mathcal E_s^d(\zeta)$ is polynomial in $\zeta$, so $P_\zeta$ --- hence, by the trace formula below, $\lambda(\zeta)$ --- depends holomorphically on $\zeta$; let $P$ be the corresponding projection of $T_s^d$ at $\lambda_s$.  Both $(T_s^d+\mathcal E_s^d(\zeta))P_\zeta$ and $T_s^dP$ are bounded of rank one: from $A(\zeta'-A)^{-1}=\zeta'(\zeta'-A)^{-1}-I$ and $\oint_{\Gamma_s}I\,d\zeta'=0$,
\[
    (T_s^d+\mathcal E_s^d(\zeta))P_\zeta=\frac1{2\pi i}\oint_{\Gamma_s}\zeta'\,(\zeta'-T_s^d-\mathcal E_s^d(\zeta))^{-1}\,d\zeta' ,
\]
a bounded operator whose range lies in the one-dimensional, $(T_s^d+\mathcal E_s^d(\zeta))$-invariant subspace $\operatorname{Ran}P_\zeta\subset\mathcal D(T_s^d)$; likewise for $T_s^dP$.  Since $P_\zeta$ has rank one and its range is invariant under $T_s^d+\mathcal E_s^d(\zeta)$, the finite-rank operator $(T_s^d+\mathcal E_s^d(\zeta))P_\zeta$ has trace equal to the unique enclosed eigenvalue $\lambda(\zeta)$; similarly, $P$ has rank one and $\operatorname{tr}(T_s^dP)=\lambda_s$.  Thus $\lambda(\zeta)=\operatorname{tr}\bigl[(T_s^d+\mathcal E_s^d(\zeta))P_\zeta\bigr]$ and $\lambda_s=\operatorname{tr}[T_s^dP]$, while the resolvent identity gives
\[
    \begin{aligned}
    (T_s^d+\mathcal E_s^d(\zeta))P_\zeta-T_s^dP
    &=\frac1{2\pi i}\oint_{\Gamma_s}\zeta'
       \Bigl[(\zeta'-T_s^d-\mathcal E_s^d(\zeta))^{-1}
             -(\zeta'-T_s^d)^{-1}\Bigr]\,d\zeta'\\
    &=\frac1{2\pi i}\oint_{\Gamma_s}\zeta'
       (\zeta'-T_s^d-\mathcal E_s^d(\zeta))^{-1}
       \mathcal E_s^d(\zeta)(\zeta'-T_s^d)^{-1}\,d\zeta' .
    \end{aligned}
\]
This difference has rank at most two, since its range lies in $\Ran P_\zeta+\Ran P$ (spectral invariance: $(T_s^d+\mathcal E_s^d(\zeta))P_\zeta=P_\zeta(T_s^d+\mathcal E_s^d(\zeta))P_\zeta$, and likewise for $T_s^dP$), and its operator norm is at most
\[
    \rho_s\cdot(\Lambda_s+\rho_s)\cdot\frac{C_t}{\rho_s}(1-\eps_s)^{-1}\cdot e_s\cdot\frac{C_t}{\rho_s}
    \ \le\ \tfrac43\,C_t(\Lambda_s+\rho_s)\,\eps_s ,
\]
using $|\zeta'|\le\Lambda_s+\rho_s$, the tameness \eqref{eq:tameness}, and the Neumann series bound $\|(\zeta'-T_s^d-\mathcal E_s^d(\zeta))^{-1}\|\le\frac{C_t}{\rho_s}(1-\eps_s)^{-1}$ with $\eps_s\le\frac14$.  Since $|\operatorname{tr}D|\le\operatorname{rank}(D)\,\|D\|$ for finite-rank $D$,
\[
    |\lambda(\zeta)-\lambda_s|\ \le\ \tfrac83\,C_t(\Lambda_s+\rho_s)\,\eps_s\ \le\ 4C_t(\Lambda_s+\rho_s)\,\eps_s\ \le\ \frac{\rho_s}{2} .
\]
For each fixed $\zeta$, $G(\zeta)=T_s^d-\zeta+\mathcal E_s^d(\zeta)$ is singular precisely when $\zeta\in\spec(T_s^d+\mathcal E_s^d(\zeta))$.  The Neumann-series argument above keeps the whole contour $\Gamma_s$ in the resolvent set of $T_s^d+\mathcal E_s^d(\zeta)$ for every $\zeta\in\overline D(\lambda_s,\rho_s)$, and the enclosed Riesz projection has rank one; hence the only spectrum of $T_s^d+\mathcal E_s^d(\zeta)$ in $D(\lambda_s,\rho_s)$ is the simple eigenvalue $\lambda(\zeta)$.  Thus, for $\zeta\in D(\lambda_s,\rho_s)$, singularity of $G(\zeta)$ is equivalent to $\zeta=\lambda(\zeta)$; for $\zeta\in\Gamma_s$, invertibility of $G(\zeta)$ has already been proved by the Neumann-series bound.  The scalar Rouch\'e theorem applied to $f(\zeta)=\zeta-\lambda(\zeta)$ against $g(\zeta)=\zeta-\lambda_s$ on $\Gamma_s$ (where $|f-g|\le\rho_s/2<\rho_s=|g|$) produces exactly one solution $\zeta_*$ in $D(\lambda_s,\rho_s)$, with $|\zeta_*-\lambda_s|=|\lambda(\zeta_*)-\lambda_s|\le4C_t(\Lambda_s+\rho_s)\eps_s$.  The Gohberg--Sigal multiplicity of $G$ at $\zeta_*$ is one by the homotopy count alone: at $\tau=0$ the family is the linear pencil $T_s^d-\zeta$, whose only characteristic value inside $\Gamma_s$ is the algebraically simple $\lambda_s$, of Gohberg--Sigal multiplicity one, so the homotopy invariance established above gives total multiplicity one for $G$ inside $\Gamma_s$; since $\zeta_*$ is a characteristic value, it is the only one and has multiplicity one.  (No analysis of the moving projection $P_\zeta$ is needed; the scalar function $\zeta-\lambda(\zeta)$ served to locate $\zeta_*$ and to prove uniqueness of the fixed point.)  Then $z=is+i\zeta_*/(2s)$ gives \eqref{eq:resonance-inclusion}.  Finally, Lemma~\ref{lem:companion} identifies the Gohberg--Sigal multiplicity of this root with its algebraic multiplicity as an eigenvalue of $\cA_{d,j}$; multiplicity one gives algebraic simplicity.
\end{proof}

The simple-mode proposition gives localization and algebraic simplicity.
The companion result below preserves the total algebraic count on a bounded
window; in Theorem~\ref{thm:E} it is the step that excludes unlisted wave
eigenvalues between the transferred rungs.
\begin{proposition}[Finite-window pencil transfer]\label{prop:cluster-transfer}
Let $d\in L^\infty(\T)$ be real-valued, let $s=s_j\ge1$, and let $\Omega\subset\{|\zeta|\le\Lambda\}$ be a bounded domain with piecewise $C^1$ boundary $\Gamma\subset\rho(T_s^d)$.  Assume the boundary tameness
\begin{equation}\label{eq:cluster-tame}
    \sup_{\zeta\in\Gamma}\bigl\|(T_s^d-\zeta)^{-1}\bigr\|\cdot\frac{\Lambda}{2s}\Bigl(\|d\|_\infty+\frac{\Lambda}{2s}\Bigr)\ <\ 1 .
\end{equation}
Then the total Gohberg--Sigal multiplicity of the characteristic values of the holomorphic family $\zeta\mapsto P_d\bigl(is+\tfrac{i\zeta}{2s}\bigr)$ in $\Omega$ equals the total algebraic multiplicity of $\spec(T_s^d)\cap\Omega$.  In particular, if $T_s^d$ has an eigenvalue in $\Omega$, then $\cA_{d,j}$ has an eigenvalue $z=is+\tfrac{i\zeta}{2s}$ with $\zeta\in\Omega$.
\end{proposition}

\begin{proof}
With $z=is+w$ and $\zeta:=-2isw$, i.e.\ $w=\tfrac{i\zeta}{2s}$, expanding gives the exact identity
\begin{gather*}
    P_d\bigl(is+\tfrac{i\zeta}{2s}\bigr)\ =\ T_s^d-\zeta+\mathcal E_s^d(\zeta),
    \qquad
    \mathcal E_s^d(\zeta)=\frac{i\zeta}{2s}\,d-\frac{\zeta^2}{4s^2},\\
    \|\mathcal E_s^d(\zeta)\|\le\frac{\Lambda}{2s}\Bigl(\|d\|_\infty+\frac{\Lambda}{2s}\Bigr)
\end{gather*}
on $\overline\Omega$.  By \eqref{eq:cluster-tame} and a Neumann series, the homotopy $F_\tau(\zeta):=T_s^d-\zeta+\tau \mathcal E_s^d(\zeta)$, $\tau\in[0,1]$, is invertible for every $\zeta\in\Gamma$.  We use the same fixed-domain reduction as in Proposition~\ref{prop:pencil-transfer}: let $\mathcal X_s^d:=H^2(\T)$ with the graph norm of the fixed operator $T_s^d$, and let $J:=(T_s^d+1)^{-1}:L^2(\T)\to\mathcal X_s^d$ be the bounded isomorphism constructed there.  The bounded perturbation $\mathcal E_s^d(\zeta)$ is polynomial in $\zeta$ as a map $\mathcal X_s^d\to L^2(\T)$, so $F_\tau(\zeta):\mathcal X_s^d\to L^2(\T)$ is a fixed-domain holomorphic Fredholm family of index zero; equivalently,
\[
    F_\tau(\zeta)J=I+\bigl(\tau \mathcal E_s^d(\zeta)-\zeta-1\bigr)(T_s^d+1)^{-1}
\]
is a bounded holomorphic Fredholm family on the fixed space $L^2(\T)$.  Since $J$ is independent of $\zeta$ and bijective, characteristic values, Keldysh chains, and Gohberg--Sigal multiplicities are unchanged by this reduction.  Lemma~\ref{lem:GS-input} therefore applies, and the total multiplicity of the characteristic values of $F_\tau$ in $\Omega$ is independent of $\tau$.  At $\tau=0$ the characteristic values are the eigenvalues of $T_s^d$ in $\Omega$, with Gohberg--Sigal multiplicity equal to algebraic multiplicity; at $\tau=1$ they are the points $\zeta$ at which $P_d(is+\tfrac{i\zeta}{2s})$ is singular, and Lemma~\ref{lem:companion} identifies the resulting point $z=is+\tfrac{i\zeta}{2s}$ of the block generator $\cA_{d,j}$ with the same algebraic multiplicity.
\end{proof}

\subsection{From resolvent and eigenvalues to the decay envelope}
The pencil transfer produces genuine eigenvalues, whereas
Theorem~\ref{thm:B} supplies the matching global resolvent upper bound.
The following elementary envelope calculus is where those two halves
meet: its upper estimate is semigroup-theoretic, and its lower estimate
uses one true eigenvalue in each dyadic octave.
\begin{lemma}[Envelope calculus]\label{lem:envelope}
Let $L$ be critical and set
\begin{equation}\label{eq:envelope-def}
    \mathfrak E(t)=\mathfrak E_{S_1}(t)
      :=\inf_{S\ge S_1}\bigl(\log S+t/L(S)\bigr),
\end{equation}
where $S_1\ge S_0$ is any fixed threshold beyond which the hypotheses below are in force.  Any two admissible choices of $S_1$ give equivalent decay envelopes in the sense that $\exp(-\mathfrak E_{S_1}(t))\asymp_{\exp}\exp(-\mathfrak E_{S_1'}(t))$; in fact every $O(1)$-near-minimizing sequence leaves every compact set as $t\to\infty$, since $L$ is unbounded.  If $\cS$ is dyadically syndetic, then for every large $S$ there is $\sigma\in\cS$ with $S\le \sigma\le4S$; hence, by monotonicity of $L$,
\begin{equation}\label{eq:dyadic-envelope-compare}
    \inf_{\substack{\sigma\in\cS\\ \sigma\ge S_1}}
       \bigl(\log \sigma+t/L(\sigma)\bigr)
    \le \mathfrak E(t)+O(1),
    \qquad t\ge1 .
\end{equation}
This is the only dyadic restriction used below: the point in $\cS$ is chosen in the dyadic block \emph{above} the competitor $S$, so the damping term decreases rather than requiring any additive regularity of $L$.

Throughout, $e^{t\cA}$ is a contraction semigroup (Lemma~\ref{lem:blocks}).  In part \textup{(a)}, increase $S_1$ if necessary so that \eqref{eq:crit-gen} holds for $|s|\ge S_1$, and let $C_{\mathrm{hf}}$ be a corresponding high-frequency constant.  Since $i\R\subset\rho(\cA)$, continuity of the resolvent gives the finite compact-frequency bound
\[
    K_0:=\sup_{|\sigma|\le S_1}\|(i\sigma-\cA)^{-1}\|<\infty.
\]
With the extension of $L$ fixed in Definition~\ref{def:critical}, define for every $S\ge0$
\begin{equation}\label{eq:global-majorant}
    C_*:=\max\{C_{\mathrm{hf}},1\},\qquad
    M(S):=C_*\max\{L(S),K_0,1\}.
\end{equation}
Then $M$ is positive and nondecreasing and $\|(is-\cA)^{-1}\|\le M(|s|)$ for every $s\in\R$.  Part \textup{(a)} applies Lemma~\ref{lem:BD-RST-input} to this global majorant.  Part \textup{(b)} uses $0\in\rho(\cA)$ only through $\cA^{-1}$.  The phrase \emph{regularized semigroup envelope} refers to this operator norm: for $U_0\in\mathcal D(\cA)$, $\|e^{t\cA}U_0\|\le\|e^{t\cA}\cA^{-1}\|\,\|\cA U_0\|$.
\begin{enumerate}[label=\textup{(\alph*)}]
\item If $i\R\subset\rho(\cA)$ and \eqref{eq:crit-gen} holds, then, for all sufficiently large $t$, $\|e^{t\cA}\cA^{-1}\|\le C\exp(-c\,\mathfrak E(ct))$.
\item Assume $0\in\rho(\cA)$.  If there is a dyadically syndetic set of frequencies $s_j$ and eigenvalues $z_j\in\spec(\cA)$ with $|z_j|\le Cs_j$ and $-\Re z_j\le CL(s_j)^{-1}$, then $\|e^{t\cA}\cA^{-1}\|\ge c\exp(-C\,\mathfrak E(Ct))$.
\item For $L(S)=(\log(eS))^A$ one has $\mathfrak E(t)\asymp t^{1/(A+1)}$; under \textup{(a)} and \textup{(b)},
\begin{equation}\label{eq:sharp-envelope}
    \|e^{t\cA}\cA^{-1}\|\ \asymp_{\exp}\ \exp\bigl(-t^{1/(A+1)}\bigr).
\end{equation}
\end{enumerate}
\end{lemma}

\begin{proof}
For the preliminary dyadic comparison, take any large $S\ge S_1$ and choose $\sigma\in\cS\cap[2^k,2^{k+1})$ in the dyadic block immediately above the one containing $S$; then $S\le\sigma\le4S$, so $\log\sigma\le\log S+\log4$ and $L(\sigma)\ge L(S)$, proving \eqref{eq:dyadic-envelope-compare} after taking the infimum over $S$.  The choice of the dyadic block above the competitor is essential: it makes the damping term monotone in the required direction, so no additive estimate for $t/L(S)-t/L(4S)$ is used.

For the threshold-independence assertion, fix a compact interval $[S_1,S_*]$.  On this interval the objective in \eqref{eq:envelope-def} is at least $t/L(S_*)+O(1)$.  Choose $\alpha<1/(4L(S_*))$ and use the competitor $S=e^{\alpha t}$; since $L(e^{\alpha t})\to\infty$, for large $t$ one has $t/L(e^{\alpha t})<t/(4L(S_*))$, and therefore $\log S+t/L(S)<t/(2L(S_*))$.  Thus no $O(1)$-near-minimizing sequence can remain in any fixed compact set.  Changing the threshold $S_1$ only removes or adds a fixed compact interval, so the corresponding exponentials $e^{-\mathfrak E_{S_1}(t)}$ are equivalent in the stated $\asymp_{\exp}$ sense.

\textup{(a)} By Lemma~\ref{lem:BD-RST-input}, with the global majorant \eqref{eq:global-majorant}, for all sufficiently large $t$,
\[
    \|e^{t\cA}\cA^{-1}\|\le\frac{C}{M_{\log}^{-1}(ct)},
    \qquad
    M_{\log}(S)=M(S)\bigl(\log(1+M(S))+\log(1+S)\bigr).
\]
For $S$ large, the fixed compact-frequency bound built into $M$ is dominated by the unbounded function $L(S)$, so $M(S)\asymp L(S)$ and, because a doubling $\ell$ has polynomial upper growth, $\log(1+M(S))\lesssim\log\log(eS)\ll\log S$; hence
\begin{equation}\label{eq:Mlog-asymp}
    M_{\log}(S)\asymp L(S)\log S,
    \qquad S\ge S_2.
\end{equation}
Let $N(t)=M_{\log}^{-1}(ct)$ in the generalized right-inverse sense of Lemma~\ref{lem:BD-RST-input}.  Since $M_{\log}(S)\to\infty$, $N(t)\to\infty$.  Fix $\eps\in(0,1)$.  For large $t$, $N(t)\ge S_2$ and $(1+\eps)N(t)>N(t)$.  By the definition of $N(t)=\sup\{S:M_{\log}(S)\le ct\}$, every $S>N(t)$ satisfies $M_{\log}(S)>ct$; otherwise $S$ would belong to the defining set and the supremum would be at least $S$.  Hence
\[
    ct<M_{\log}((1+\eps)N(t))\lesssim L((1+\eps)N(t))\log((1+\eps)N(t))\lesssim L(N(t))\log N(t),
\]
where the middle comparison uses \eqref{eq:Mlog-asymp} and the last one uses doubling/scale stability of the critical gauge at the fixed factor $1+\eps$.  Thus
\begin{equation}\label{eq:inverse-bookkeeping}
    t\lesssim L(N(t))\log N(t) .
\end{equation}
Thus, for a sufficiently small absolute $c_0>0$,
\[
    \mathfrak E(c_0t)\le\log N(t)+\frac{c_0t}{L(N(t))}\le C\log N(t).
\]
Consequently $N(t)^{-1}\le\exp(-c\mathfrak E(c_0t))$, and the asserted upper bound follows, with constants absorbed into the notation of part \textup{(a)}.  This argument uses only right-inverse inequalities and is unaffected by jumps or flat portions of $M_{\log}$.  This proves the asserted estimate for all sufficiently large $t$.  Since $0\in\rho(\cA)$, the function $t\mapsto\|e^{t\cA}\cA^{-1}\|$ is bounded on compact time intervals, while $\exp(-c\mathfrak E(c_0t))$ has a positive minimum on each such interval.  Increasing the exterior constant therefore extends the estimate to every $t\ge0$.

\textup{(b)} Let $U_j$ be a unit eigenvector for $z_j$.  Then
\[
    e^{t\cA}\cA^{-1}U_j=e^{tz_j}z_j^{-1}U_j,
\]
so
\[
    \|e^{t\cA}\cA^{-1}\|\ge\sup_j\frac{e^{t\Re z_j}}{|z_j|}
    \ge c\exp\Bigl(-\inf_j\bigl(\log s_j+CtL(s_j)^{-1}\bigr)\Bigr).
\]
The frequencies under consideration are dyadically syndetic; applying \eqref{eq:dyadic-envelope-compare} with $Ct$ in place of $t$ bounds the discrete infimum by $\mathfrak E(Ct)+O(1)$.  This yields part \textup{(b)}.

\textup{(c)} With $u=\log(eS)$ and $u_1=\log(eS_1)$,
\[
 \mathfrak E(t)=\inf_{u\ge u_1}(u-1+tu^{-A})
   =\inf_{u\ge u_1}(u+tu^{-A})+O(1).
\]
The last objective is minimized at $u_*\asymp t^{1/(A+1)}$ with value
$\asymp t^{1/(A+1)}$.
\end{proof}

\section{The imaginary logarithmic oscillator: proof of Theorem \ref{thm:D}}\label{sec:oscillator}

Theorem~\ref{thm:D} supplies the spectral model required by the bridge
above.  The proof has four stages: construct the self-adjoint and
nonnormal half-line realizations; use logarithmic Liouville--Green control
to justify boundary rotation; identify algebraic multiplicity through the
rotated left--right pairing; and merge the Neumann and Dirichlet spectra
by parity and rank-one interlacing.  The resulting nonzero bilinear
overlap is retained explicitly because it later prevents rank and Jordan
defects in the cusp limit.

\subsection{Realizations and branches}
Throughout, $\Log$ is the principal logarithm on
$\C\setminus(-\infty,0]$.  For $\sharp\in\{N,D\}$, the self-adjoint
operator $H_\gamma^\sharp$ is associated with the closed lower-bounded form
\[
 \int_0^\infty |u'|^2\,dy
 +\gamma\int_0^\infty(\log y)_+|u|^2\,dy
 -\gamma\int_0^\infty(\log y)_-|u|^2\,dy
\]
on $H^1(\R_+)\cap L^2((\log y)_+dy)$, with the additional condition
$u(0)=0$ in the Dirichlet case.  The nonselfadjoint half-line realization
is the maximal separated operator
\[
\begin{aligned}
\mathcal D(\cL_\gamma^\sharp)=\{u\in L^2(\R_+):{}&u,u'\in
 AC_{\mathrm{loc}}(0,\infty),\ -u''+i\gamma\log y\,u\in L^2,\\
&u'(0)=0\ \text{if }\sharp=N,\quad u(0)=0\ \text{if }\sharp=D\},
\end{aligned}
\]
where the endpoint traces exist by local regularity.  The full-line
operator is the maximal distributional realization
\[
 \mathcal D(\cL_\gamma)=\{u\in L^2(\R):
 -u''+i\gamma\log|y|\,u\in L^2(\R)\text{ in }\mathcal D'(\R)\}.
\]
Lemma~\ref{lem:log-form-realization} and Lemma~\ref{lem:parity} verify the
equivalent maximal-domain descriptions and the parity splitting.
Rotated values use the sectorial continuation constructed below, with
$\Log(e^{i\pi/4}y)=\log y+i\pi/4$.

\subsection{Auxiliary ODE facts}

\begin{lemma}[Sector Liouville--Green]\label{lem:sector-LG}
Let $\gamma>0$, $E\in\C$, $\alpha\in\{0,\pi/2\}$, and $q(z)=e^{i\alpha}\gamma\Log z-E$.  Let $[\theta_-,\theta_+]\subset[-\pi/4,\pi/4]$ satisfy $\cos(\theta+\alpha/2)\ge c_0>0$ for all $\theta\in[\theta_-,\theta_+]$, and let
\[
    \Sigma=\{re^{i\theta}:\ r\ge y_1,\ \theta\in[\theta_-,\theta_+]\}.
\]
Then for $y_1=y_1(\gamma,E,c_0)$ large the following hold; more generally, if $E$ ranges over a fixed compact subset of $\C$, the same $y_1$ and the constants below may be chosen uniformly over that compact set.  Put
\[
    \Sigma_*:=\{|z|>y_1/2,\ |\arg z|<\pi/3\},
\]
an open truncated sector containing $\Sigma$, and set
\[
    \xi(z)=\int_{y_1}^{z}q(t)^{1/2}\,dt ,
\]
where $q^{1/2}$ is the principal branch and $q^{1/4}$ is the unique compatible holomorphic branch on $\Sigma_*$ selected by
\[
    q^{1/4}(y_1)
       =\exp\!\left(\tfrac12\Log(q^{1/2}(y_1))\right).
\]
This branch is well defined by holomorphic continuation on the simply connected domain $\Sigma_*$: indeed, $q(z)=e^{i\alpha}\gamma\log|z|+O(1)$ uniformly there, so for $y_1$ large the image $q(\Sigma_*)$ lies in the sector $\{|\arg\zeta-\alpha|\le\pi/4\}$ and avoids the cut $(-\infty,0]$.  The phase $\xi$ is path-independent, and is well defined even when $y_1\notin\Sigma$, because $q$ is holomorphic and zero-free on the open, simply connected $\Sigma_*$, which contains both $\Sigma$ and the base point $y_1$.  The equation $u''=q(z)u$ has a solution
\[
    u_+(z)=q(z)^{-1/4}e^{-\xi(z)}\bigl(1+\eta_+(z)\bigr),
\]
holomorphic on an open sector containing $\Sigma$, with $\sup_\Sigma|\eta_+|\le\tfrac12$ and $\eta_+(z)\to0$ as $|z|\to\infty$ uniformly on $\Sigma$; and, along each ray $\{re^{i\theta}:r\ge y_1\}$ of $\Sigma$, a second solution
\[
    u_-(z)=u_-^{(\theta)}(z)=q(z)^{-1/4}e^{+\xi(z)}\bigl(1+\eta_-(z)\bigr),
    \qquad \sup_{r\ge y_1}\bigl|\eta_-(re^{i\theta})\bigr|\le\tfrac12 .
\]
The same asymptotics hold for the derivatives, with an extra factor $\mp q^{1/2}$, with relative error again at most $\tfrac12$, and $o(1)$ as $|z|\to\infty$ in the recessive case.  Consequently, any local holomorphic solution of $u''=q(z)u$ on a connected neighborhood in $\Sigma_*$, continued along paths in that simply connected domain --- equivalently, any ray solution with finite Cauchy data at some point of the ray and then analytically continued by the ODE in $\Sigma_*$ --- that is square-integrable near infinity along some ray of $\Sigma$, or that tends to $0$ along some ray, is a multiple of $u_+$, and therefore obeys the recessive bounds on all of $\Sigma$.  Along every ray in $\Sigma$,
\[
    \frac{d}{dr}\,\Re\xi(re^{i\theta})=\Re\bigl(e^{i\theta}q(re^{i\theta})^{1/2}\bigr)\ \ge\ \frac{c_0}{2}\,(\gamma\log r)^{1/2},
\]
so $u_+$ is recessive on all of $\Sigma$:
\[
    |u_+(re^{i\theta})|\le C\exp\bigl(-c\,r(\log r)^{1/2}\bigr)\quad\text{uniformly on }\Sigma,
\]
and every solution that tends to $0$ along one ray of $\Sigma$ is a constant multiple of $u_+$, hence is recessive on all of the tail $\Sigma$.  The lemma makes no assertion about the compact part $0<r<y_1$; in later sectorial-continuation arguments that part is handled separately by the local Volterra continuation and ordinary compactness.
\end{lemma}

\begin{proof}
On $\Sigma$, $q(z)^{1/2}=e^{i\alpha/2}(\gamma\log|z|)^{1/2}(1+O(1/\log|z|))$ uniformly, which gives the stated lower bound for $\frac{d}{dr}\Re\xi$; in particular every ray is a progressive regular $C^2$ path and $\Sigma$ contains no turning points for $y_1$ large.  Along these rays, the majorizing error-control variation from Lemma~\ref{lem:olver-input} obeys
\[
    \begin{aligned}
    \mathcal V&:=\sup_{\theta\in[\theta_-,\theta_+]}\int_{y_1}^{\infty}\omega_q\bigl(re^{i\theta}\bigr)\,dr\\
    &\lesssim\ \int_{y_1}^\infty\frac{dr}{r^2(\log r)^{3/2}}\ \longrightarrow\ 0
    \quad(y_1\to\infty),
    \end{aligned}
\]
so the logarithmic endpoint-at-infinity conclusions proved in Lemma~\ref{lem:olver-input} apply.  Take $D=\Sigma_*$, the open truncated sector of the statement, on which $q$ is holomorphic and zero-free.  For $u_+$ take the reference endpoint $z_{\rm ref}=\infty$: every point of $\Sigma$ is joined to $\infty$ along its own ray, on which $\Re\xi$ is monotone by the display in the statement, and no other paths are needed; moreover, by continuity of $\cos(\theta+\alpha/2)$ the same display holds with $c_0/2$ for $\theta$ in a slightly larger interval and for $|z|>y_1'$ with some $y_1'\in(y_1/2,y_1)$, so $H(\infty)$ contains a connected open sector neighborhood of $\Sigma$ inside $\Sigma_*$, with the same variation bound (the error-control majorant is estimated on all of $\Sigma_*$), and one application of Lemma~\ref{lem:olver-input} on this single progressive-path domain --- no raywise patching --- yields $u_+$, holomorphic on an open sector containing $\Sigma$ as stated, with $\sup_\Sigma|\eta_+|\le C(e^{C\mathcal V}-1)\le\tfrac12$ for $y_1$ large and, the tail of the variation vanishing at infinity, $\eta_+(z)\to0$ as $|z|\to\infty$ uniformly.  No patching of raywise solutions is needed: $u_+$ is produced by the single endpoint-at-infinity Volterra construction on the progressive-path domain $H(\infty)$, and the raywise uniqueness of recessive solutions is the basis argument at the end of this proof.  For $u_-^{(\theta)}$ no sector-wide dominant branch is used.  Fix the ray $re^{i\theta}$ and use the raywise Volterra construction \eqref{eq:olver-ray-dominant} from Lemma~\ref{lem:olver-input}, with the progressive direction reversed for the exponential $e^{+\xi}$ and with initial point $y_1e^{i\theta}$.  The raywise variation is bounded by the same integral, so $\sup_{r\ge y_1}|\eta_-|\le C(e^{C\mathcal V}-1)\le\tfrac12$; after increasing $y_1$ this quantity, together with the derivative-error bound from Lemma~\ref{lem:olver-input}, is in fact as small as desired.  The derivative asymptotics are part of the same construction.  The obtained ray solution has Cauchy data at any point of the ray and hence admits a local holomorphic continuation by the ODE, but no global dominant sectorial branch is asserted.  Finally, along any fixed ray, $(u_+,u_-^{(\theta)})$ is a solution basis.  Indeed, increase $y_1$ so that the four relative errors in the two function asymptotics and the two differentiated asymptotics are all $<\varepsilon$ on the ray, with $\varepsilon<1/10$.  Evaluating the constant Wronskian far out on that ray gives $W(u_+,u_-^{(\theta)})=2+O(\varepsilon)$, hence it is nonzero.  If $u=c_+u_++c_-u_-^{(\theta)}$ with $c_-\neq0$, then the dominance of $e^{+\xi}$ over $e^{-\xi}$ must be made explicit.  Since $\Re\xi(re^{i\theta})\to+\infty$ and $u_+$ has the recessive bound while $|1+\eta_-|\ge1/2$, there is $R$ on that ray such that, for all $r\ge R$,
\[
\begin{aligned}
    |c_+u_+(re^{i\theta})|
    &\le\frac{|c_-|}{4}|q(re^{i\theta})|^{-1/4}e^{\Re\xi(re^{i\theta})},\\
    |c_-u_-^{(\theta)}(re^{i\theta})|
    &\ge\frac{|c_-|}{2}|q(re^{i\theta})|^{-1/4}e^{\Re\xi(re^{i\theta})}.
\end{aligned}
\]
Thus $|u(re^{i\theta})|\ge\frac{|c_-|}{4}|q|^{-1/4}e^{\Re\xi}\to\infty$ for $r\ge R$, which is neither square-integrable near infinity nor null there; the finite initial segment is irrelevant.  Hence any solution square-integrable or vanishing at infinity along some ray has $c_-=0$, is a multiple of $u_+$, and obeys the recessive bounds on all of $\Sigma$.
\end{proof}

\begin{lemma}[Volterra continuation at the logarithmic endpoint]\label{lem:log-volterra-endpoint}
Let $\alpha\in\{0,\pi/2\}$, $\gamma>0$, $\lambda\in\C$, and let $-\pi<\theta_-<\theta_+<\pi$.  On the punctured sector
\[
    \Sigma_0=\{0<|z|\le\delta,\ \theta_-\le\arg z\le\theta_+\}
\]
we use the branch of $\Log$ determined by this angular interval and set $c(z)=e^{i\alpha}\gamma\Log z-\lambda$.  For every $\varepsilon>0$, the truncated sector $\Sigma_0\cap\{|z|\ge\varepsilon\}$ is compactly contained in the branch domain; the endpoint $0$ is adjoined only for the continuous extension and for the Volterra traces.  There exists $\delta_0=\delta_0(\alpha,\gamma,\lambda,\theta_-,\theta_+)>0$ such that, for every $0<\delta\le\delta_0$ and every pair $(a_0,b_0)\in\C^2$, there is a unique function $u$ continuous on $\Sigma_0\cup\{0\}$ and holomorphic in the punctured interior of $\Sigma_0$, whose complex derivative in the punctured interior extends continuously to the closed sector and to $0$.  Along a boundary ray $z=re^{i\theta}$ the radial derivative is not the complex derivative but
\[
    \partial_r u(re^{i\theta})=e^{i\theta}u'(re^{i\theta}),
    \qquad \partial_r u(0;\theta)=e^{i\theta}b_0.
\]
It satisfies
\[
    u(z)=a_0+b_0z+
    \int_0^z (z-t)c(t)u(t)\,dt
\]
where the integral is taken over the radial segment from $0$ to $z$.  The complex Cauchy values $u(0)=a_0$ and $u'(0)=b_0$ are independent of the ray; here $u'(0)$ means the limit at $0$ of the holomorphic derivative on the punctured sector interior, while boundary derivatives on the sector edges are the radial derivatives displayed above and are compatible with that trace.  Any sectorial holomorphic solution of $u''=c(z)u$ with finite Cauchy data at the endpoint satisfies this Volterra formula.  Therefore the Dirichlet condition is preserved under rotation, and the zero Neumann condition is preserved under rotation because the radial endpoint derivative is multiplied only by the nonzero phase $e^{i\theta}$.
\end{lemma}

\begin{proof}
Choose $\delta_0>0$ so small that the linear Volterra operator has norm at most $1/2$ on every closed sector $0<|z|\le\delta\le\delta_0$ with the fixed angular interval; this is possible because $\sup_{\theta\in[\theta_-,\theta_+] }\int_0^\delta |c(re^{i\theta})|\,dr\to0$ as $\delta\downarrow0$.  Thus the Picard map defined by the displayed Volterra equation is a contraction on continuous functions on the closed angular sector with the endpoint adjoined for all $0<\delta\le\delta_0$, uniformly in the affine data $(a_0,b_0)$.  Rewriting the integral as
\[
    z^2\int_0^1(1-\tau)c(\tau z)u(\tau z)\,d\tau
\]
shows the analytic dependence on $z$: the Picard iterates are holomorphic on every truncated punctured sector and converge locally uniformly there, hence the fixed point is holomorphic in the punctured interior by Weierstrass' theorem.  Differentiating the holomorphic Volterra expression with respect to the complex variable gives
\[
    u'(z)=b_0+\int_0^z c(t)u(t)\,dt
       =b_0+z\int_0^1c(\tau z)u(\tau z)\,d\tau .
\]
Equivalently, differentiating along a ray gives $\partial_r u(re^{i\theta})=e^{i\theta}u'(re^{i\theta})$.  The same Picard/local-uniform argument gives holomorphy of $u'$ off $0$, while the displayed formula gives continuity up to $0$ and the ray-independent complex-derivative limit $b_0$.  For $z_0\ne0$, choose a small disc in the punctured sector about $z_0$.  There $c(t)u(t)$ is holomorphic, the local antiderivative in $u'(z)=b_0+\int_0^z c(t)u(t)\,dt$ is path-independent up to a constant, and the complex derivative of $u'$ at $z_0$ is therefore $c(z_0)u(z_0)$; equivalently, this follows by differentiating the locally uniformly convergent holomorphic Picard iterates on that disc.  Hence $u''=cu$ throughout the punctured sector.  Conversely, integrate $u''=cu$ twice along a radial segment and use the finite endpoint traces.  Uniqueness follows either from the contraction argument on a sufficiently small sectorial neighborhood of $0$ and continuation, or directly from the homogeneous Volterra equation and Gronwall along each ray.

\end{proof}

\begin{lemma}[Logarithmic Sturm--Liouville form realization]\label{lem:log-form-realization}
Let $\gamma>0$ and $\sharp\in\{N,D\}$.  Let $h^\sharp_\gamma$ be the closed semibounded quadratic form
\[
    h^\sharp_\gamma[u]=\int_0^\infty |u'|^2+\gamma\int_0^\infty (\log y)_+|u|^2
       -\gamma\int_0^\infty (\log y)_-|u|^2
\]
on $H^1(\R_+)\cap L^2((\log y)_+dy)$, with the additional trace condition $u(0)=0$ in the Dirichlet case and no trace condition in the Neumann case; write $\mathcal Q(h^\sharp_\gamma)$ for this form domain.  The negative part is finite on this domain and infinitesimally form-bounded on $(0,1)$, so the displayed plus-minus expression is equivalent to the signed potential notation $\int(\log y)|u|^2$ only after this fact is used.  The embedding of $\mathcal Q(h^\sharp_\gamma)$ into $L^2(\R_+)$ is compact, hence the associated self-adjoint operator has compact resolvent; this operator is exactly the maximal separated differential realization
\[
\begin{aligned}
\mathcal D(H^\sharp_\gamma)=\{u\in L^2(\R_+):
&\ u,u'\in AC_{\rm loc}(0,\infty),\ u,u'\ \text{have finite traces at }0,\\
&\ -u''+\gamma\log y\,u\in L^2(\R_+),\quad\\
&\ u'(0)=0\ \text{if }\sharp=N,
\quad u(0)=0\ \text{if }\sharp=D\}.
\end{aligned}
\]
The endpoint traces exist for every element of this maximal domain; the separated condition in the displayed set is therefore a genuine trace condition, as verified in the proof.
\end{lemma}

\begin{proof}
The negative part of $\log y$ lies in $L^1(0,1)$ and is infinitesimally form-bounded with respect to the $H^1(0,1)$ norm, while $(\log y)_+\to\infty$ at infinity.  Thus the forms above are closed and semibounded.  The embedding of the form domain into $L^2(\R_+)$ is compact: on $[0,R]$ this is Rellich compactness for $H^1(0,R)$, while for $R>e$ the tail satisfies $\int_R^\infty |u|^2\le(\log R)^{-1}\int_R^\infty (\log y)_+|u|^2$, uniformly on form-bounded sets.  Consequently the associated form operators have compact resolvent.  The endpoint $+\infty$ is limit point for the differential expression $-d^2/dy^2+\gamma\log y$: for a fixed real spectral parameter $E$, choose the lower radius in Lemma~\ref{lem:sector-LG} with $\alpha=0$ on the positive ray larger than the possible turning point $e^{E/\gamma}$ when $E>0$ (and larger than $e$).  On that tail the Liouville--Green basis has one recessive and one dominant member, the latter growing superexponentially.  Endpoint classification is unchanged by removing the finite interval before this radius, so not every solution is square-integrable near $+\infty$, and Weyl's alternative gives the limit-point case.

Let $A^\sharp$ denote the associated form operator and suppose $A^\sharp u=f$.  Testing the form identity against $C_c^\infty(0,\infty)$ gives
\[
    -u''+\gamma\log y\,u=f
\]
in distributions.  Near $0$, since $\log y\in L^2(0,1)$ and $u\in L^2(0,\delta)$, Cauchy--Schwarz gives $(\log y)u\in L^1(0,\delta)$; also $f\in L^2(0,\delta)\subset L^1(0,\delta)$.  Hence $u''=\gamma\log y\,u-f\in L^1(0,\delta)$, so $u,u'$ extend absolutely continuously to $[0,\delta]$ and the endpoint traces exist.  Testing against form-domain functions with arbitrary trace at $0$ gives the Neumann boundary condition in the Neumann case, while the Dirichlet trace condition is already part of the Dirichlet form domain.  The limit-point property proved in the first paragraph of this lemma supplies the absence of a boundary condition at $+\infty$.  Therefore $\mathcal D(A^\sharp)$ is contained in the displayed maximal separated domain.

Conversely, let $u$ belong to the displayed maximal separated domain and write $f=-u''+\gamma\log y\,u\in L^2$.  The same Cauchy--Schwarz argument near $0$ gives $(\log y)u\in L^1(0,\delta)$ and $u''\in L^1(0,\delta)$, hence $u,u'\in AC[0,\delta]$ and the separated boundary condition is meaningful.  On every finite interval the regular-endpoint Green identity is valid: integrate first on $[\varepsilon,2R]$, pass $\varepsilon\downarrow0$ using $u,u'\in AC[0,\delta]$ and the separated trace condition at $0$, and then use the outer cutoff to remove the endpoint at $2R$.  Choose cutoffs $\chi_R$ equal to one on $[0,R]$, supported in $[0,2R]$, with $|\chi_R'|\le C/R$.  Integration by parts gives
\[
    \int |(\chi_Ru)'|^2+\gamma\int (\log y)|\chi_Ru|^2
    =\Re\int f\,\overline{\chi_R^2u}+\int |\chi_R'|^2|u|^2,
\]
with the endpoint term at $0$ killed by the separated condition and the outer endpoint killed by the cutoff.  Move the negative part of the logarithm to the right.  Since $(\log y)_-$ is infinitesimally form-bounded on $(0,1)$, for every small $\varepsilon>0$,
\[
    \gamma\int_0^1(\log y)_-|\chi_Ru|^2
    \le \varepsilon\int_0^1 |(\chi_Ru)'|^2+C_\varepsilon\int_0^1 |u|^2.
\]
The source term is bounded by $\|f\|_2\|u\|_2$, and the cutoff term is bounded by $C\|u\|_2^2$.  Taking $\varepsilon$ small and absorbing, we obtain a uniform bound
\[
    \sup_R\left(\int |(\chi_Ru)'|^2+\gamma\int (\log y)_+|\chi_Ru|^2\right)<\infty .
\]
Lower semicontinuity and Fatou's lemma imply $u\in H^1(\R_+)\cap L^2((\log y)_+dy)$, with the same separated trace condition.  Finally, for every smooth form test function $v$ with compact support in $[0,\infty)$ --- with no endpoint restriction in the Neumann form case and with $v(0)=0$ in the Dirichlet form case --- the Green identity gives $h^\sharp_\gamma[u,v]=\langle f,v\rangle$.  This class of test functions is form-dense in the corresponding form domain: in the Neumann case the approximants may be nonzero at $0$, while in the Dirichlet case the trace-zero condition is preserved.  To see density, truncate at infinity by cutoffs equal to one near the regular endpoint, then mollify on compact half-line charts; in the Neumann case use a standard extension preserving the boundary value, and in the Dirichlet case use an odd extension to preserve trace zero.  Passing by this density yields
\[
    h^\sharp_\gamma[u,v]=\langle f,v\rangle,
    \qquad v\in\mathcal Q(h^\sharp_\gamma),
\]
so $u\in\mathcal D(A^\sharp)$ and $A^\sharp u=f$.  This proves equality with the maximal separated realization.
\end{proof}

The next two lemmas package, for citation, the two nonstandard quantitative devices used in the proof of Theorem~\ref{thm:D}: a Schur-type compactness criterion for Green kernels whose off-diagonal decay is measured by a phase $\Xi$ of Liouville--Green type, and the inhomogeneous recessive estimate that drives the exclusion of Jordan chains.

\begin{lemma}[Schur tail estimate]\label{lem:schur-tail}
Let $\gamma>0$, $q_0(y)=1+\gamma\log_+y$ with $\log_+y:=\max(\log y,0)$, and let $\Xi:(0,\infty)\to\R$ be nondecreasing, constant on $(0,y_1]$ for some $y_1\ge2$, and $C^1$ on $[y_1,\infty)$ with $\Xi'(y)\asymp q_0(y)^{1/2}$ there.  Since only phase differences enter the kernel hypothesis, we normalize $\Xi(y_1)=0$; equivalently one may replace an arbitrary such phase by $\Xi-\Xi(y_1)$, leaving \eqref{eq:schur-kernel-bound} unchanged.  Let $G$ be a measurable kernel on $(0,\infty)^2$, bounded on $(0,2]^2$, square-integrable on $(0,R)^2$ for every $R<\infty$, and satisfying
\begin{equation}\label{eq:schur-kernel-bound}
    |G(x,y)|\ \le\ C_*\,q_0(x\wedge y)^{-1/4}\,q_0(x\vee y)^{-1/4}\,e^{-(\Xi(x\vee y)-\Xi(x\wedge y))}
    \qquad\text{for }x\vee y\ge2 .
\end{equation}
Then the Schur norms of $G$ are finite, the truncation remainders $G_R:=G\,\one_{\{x\vee y\ge R\}}$ obey
\begin{equation}\label{eq:schur-tail-bound}
\begin{aligned}
    &\sup_x\int_0^\infty |G_R(x,y)|\,dy
      +\sup_y\int_0^\infty |G_R(x,y)|\,dx \\
    &\hspace{2.5cm}\le C\,q_0(R)^{-1/4}\longrightarrow0,
      \qquad R\to\infty
\end{aligned}
\end{equation}
The exponent $1/4$ is deliberately non-sharp: the proof gives the
stronger global tail bound $O(q_0(R)^{-3/4})$, while only the displayed
weaker rate is needed below.  The integral
operator associated with $G$ is bounded and compact on $L^2(0,\infty)$.
All constants depend only on $C_*$, $\gamma$, $y_1$, the comparability
constants of $\Xi'$, and the bound on $(0,2]^2$.
\end{lemma}

\begin{proof}
Since $q_0$ and $\Xi$ are nondecreasing, the right-hand side of
\eqref{eq:schur-kernel-bound} is
\[
 C_*M(x,y),\qquad
 M(x,y):=q_0(x)^{-1/4}q_0(y)^{-1/4}
 e^{-|\Xi(x)-\Xi(y)|},
\]
and $M$ is symmetric.  After enlarging the constant, the assumed bound
on $(0,2]^2$ therefore gives a symmetric pointwise majorant for $|G|$
on all of $(0,\infty)^2$.

We first verify the global Schur bounds.  Since $\Xi=0$ on $(0,y_1]$,
\[
 \int_0^{y_1}M(x,y)\,dy
 \le Cq_0(x)^{-1/4}e^{-\Xi(x)}\le C.
\]
On $[y_1,\infty)$ the hypothesis on $\Xi'$ makes $\Xi$ strictly
increasing and unbounded.  With $u=\Xi(y)$ and
$dy\le Cq_0(y)^{-1/2}\,du$, we obtain
\[
\begin{aligned}
 \int_{y_1}^{\infty}M(x,y)\,dy
 &\le Cq_0(x)^{-1/4}
   \int_0^\infty e^{-|u-\Xi(x)|}
        q_0(y(u))^{-3/4}\,du  \\
 &\le Cq_0(x)^{-1/4}
   \int_0^\infty e^{-|u-\Xi(x)|}\,du
 \le C .
\end{aligned}
\]
Thus $\sup_x\int_0^\infty |G(x,y)|\,dy<\infty$; the other Schur norm
is finite by symmetry of the majorant.

It remains to prove the tail estimate.  The bounded range of $R$ is
absorbed into the constant by the global Schur bounds just proved.  Put
$Q_R:=q_0(R)$ and henceforth take $R\ge2y_1$ sufficiently large that
$q_0(R/2)\asymp Q_R$.  If $x<R$,
the support condition in $G_R$ forces $y\ge R$, and the same change of
variables gives
\[
\begin{aligned}
 \int_0^\infty |G_R(x,y)|\,dy
 &=\int_R^\infty |G(x,y)|\,dy \\
 &\le Cq_0(x)^{-1/4}Q_R^{-3/4}
       \int_{\Xi(R)}^\infty e^{-(u-\Xi(x))}\,du \\
 &\le CQ_R^{-3/4}.
\end{aligned}
\]
Here $\Xi(x)\le\Xi(R)$ and $q_0(x)^{-1/4}\le1$.

Now let $x\ge R$.  On the large--large region, substitution over
$y\ge R$ yields
\[
\begin{aligned}
 \int_R^\infty |G(x,y)|\,dy
 &\le Cq_0(x)^{-1/4}Q_R^{-3/4}
    \int_{\Xi(R)}^\infty e^{-|u-\Xi(x)|}\,du \\
 &\le CQ_R^{-1}.
\end{aligned}
\]
For the mixed region $R/2\le y\le R$, the comparability
$q_0(y)\asymp Q_R$ gives
\[
\begin{aligned}
 \int_{R/2}^{R}|G(x,y)|\,dy
 &\le Cq_0(x)^{-1/4}Q_R^{-3/4}
    \int_{\Xi(R/2)}^{\Xi(R)}
       e^{-(\Xi(x)-u)}\,du \\
 &\le CQ_R^{-1}.
\end{aligned}
\]
On the remaining part $0<y\le R/2$, monotonicity and the lower bound on
$\Xi'$ give
\[
 \Xi(x)-\Xi(y)
 \ge \Xi(R)-\Xi(R/2)
 \ge cR\,q_0(R/2)^{1/2}
 \ge c'R Q_R^{1/2}.
\]
Consequently,
\[
 \int_0^{R/2}|G(x,y)|\,dy
 \le CRq_0(x)^{-1/4}e^{-c'RQ_R^{1/2}}
 \le CQ_R^{-1}
\]
for all sufficiently large $R$.  Hence
\[
 \sup_x\int_0^\infty |G_R(x,y)|\,dy
 \le CQ_R^{-3/4}.
\]
Because $|G_R|$ is bounded by the symmetric kernel
$C_*M\,\one_{\{x\vee y\ge R\}}$, the identical calculation with the
variables interchanged gives
\[
 \sup_y\int_0^\infty |G_R(x,y)|\,dx
 \le CQ_R^{-3/4}.
\]
This proves \eqref{eq:schur-tail-bound}, indeed with the stronger exponent
$3/4$.

Finally, for fixed $R$ the kernel
$G\,\one_{\{x\vee y<R\}}$ is Hilbert--Schmidt by local square
integrability.  Schur's test and the preceding estimates show that the
operator with kernel $G_R$ has norm tending to zero.  The operator
associated with $G$ is therefore the operator-norm limit of
Hilbert--Schmidt operators, and hence is bounded and compact.
\end{proof}

\begin{lemma}[Inhomogeneous recessive lemma]\label{lem:inhom-recessive}
Let $\gamma>0$ and $\lambda\in\C$, and set $q_\lambda(y)=i\gamma\Log y-\lambda$, $\xi_\lambda(y)=\int_{y_1}^y q_\lambda(t)^{1/2}\,dt$, and $\Xi_\lambda=\Re\xi_\lambda$, with $y_1=y_1(\gamma,\lambda)$ as in Lemma~\ref{lem:sector-LG}, applied with $\alpha=\pi/2$, $E=\lambda$, on the ray $\theta=0$; the weight $q_0(y)=1+\gamma\log_+y$ is comparable to $|q_\lambda(y)|$ for large $y$, with constants depending on $\lambda$.  Let $g$ be continuous on $[y_1,\infty)$ with
\[
    |g(y)|\le C_g\,(1+y)\,e^{-\Xi_\lambda(y)},\qquad y\ge y_1,
\]
and let $w\in L^2(y_1,\infty)$, with $w,w'\in AC_{\mathrm{loc}}[y_1,\infty)$, solve
\[
    w''-q_\lambda w=-g\qquad\text{on }[y_1,\infty).
\]
Then $w$ and $w'$ decay superexponentially:
\begin{equation}\label{eq:inhom-recessive-bound}
    |w(y)|+q_0(y)^{-1/2}\,|w'(y)|\ \le\ C\,(1+y)^2\,e^{-\Xi_\lambda(y)},
    \qquad y\ge y_1 .
\end{equation}
\end{lemma}

\begin{proof}
Write $q=q_\lambda$.  On $[y_1,\infty)$ take the Liouville--Green basis $u_R:=u_+$ (recessive), $u_D:=u_-^{(0)}$ (the dominant solution on the positive real ray), supplied by Lemma~\ref{lem:sector-LG} and normalized on that ray so that $W(u_R,u_D)=u_Ru_D'-u_R'u_D=1$, and set
\[
    w_p(y):=u_D(y)\int_y^\infty u_R\,g\ +\ u_R(y)\int_{y_1}^{y}u_D\,g ;
\]
the boundary terms in $w_p'$ cancel, $w_p'=u_D'\int_y^\infty u_Rg+u_R'\int_{y_1}^yu_Dg$, and differentiating once more gives $w_p''-qw_p=(u_R'u_D-u_Ru_D')g=-g$: a particular solution.  Since $\frac{d}{dt}\Xi_\lambda\ge c(\log t)^{1/2}\ge c>0$, one has $\int_y^\infty(1+t)e^{-2\Xi_\lambda(t)}dt\le C(1+y)e^{-2\Xi_\lambda(y)}$; so $|u_Rg|\le C(1+t)e^{-2\Xi_\lambda(t)}$ bounds the first term by $C(1+y)e^{-\Xi_\lambda(y)}$, and $|u_Dg|\le C|q|^{-1/4}e^{\Xi_\lambda(t)}\cdot(1+t)e^{-\Xi_\lambda(t)}\le C(1+t)$ bounds the second by $C(1+y)^2e^{-\Xi_\lambda(y)}$.  The same estimates apply to $w_p'$, each differentiation of the basis costing a factor $|q|^{1/2}$.  So $w_p,w_p'$ decay superexponentially, the polynomial factors being absorbed by $e^{-\Xi_\lambda(y)}=e^{-cy(\log y)^{1/2}(1+o(1))}$.  Every solution $w$ as in the statement admits on $[y_1,\infty)$ the representation $w=c_Ru_R+c_Du_D+w_p$, with explicit coefficients.  Indeed $f:=w-w_p$ solves the homogeneous equation $f''-qf=0$ on $[y_1,\infty)$, hence
\[
    c_R=W(f,u_D),\qquad c_D=W(u_R,f),\qquad W(p,r)=pr'-p'r,
\]
are independent of $y$ because $W(u_R,u_D)=1$, and $f=c_Ru_R+c_Du_D$.  The estimates above show $w_p,w_p'\in L^2$ and are recessive.  If $c_D\ne0$, consider the unit $\Xi_\lambda$-slabs $S_k=\{y\ge y_1:\ \Xi_\lambda(y)\in[k,k+1]\}$ for large $k$.  Since $c\le\Xi_\lambda'(y)\le Cq_0(y)^{1/2}$, on $S_k$ one has $y\le y_1+(k+1)/c$, hence $q_0\le C\log(k+2)$ and $|S_k|\ge c\,(\log(k+2))^{-1/2}$; the Liouville--Green asymptotics then give $\int_{S_k}|u_D|^2\ge c\,e^{2k}/\log(k+2)$, while $|c_Ru_R|+|w_p|\le C(1+y)^2e^{-\Xi_\lambda}\le C\,k^2e^{-k}$ on $S_k$ by the estimates above.  Hence, by $|f+w_p|^2\ge\tfrac12|c_Du_D|^2-\bigl(|c_Ru_R|+|w_p|\bigr)^2$,

\[
\begin{aligned}
    \int_{S_k}|w|^2
    &\ge \frac{|c_D|^2}{2}\int_{S_k}|u_D|^2-C\,k^4e^{-2k} \\
    &\ge c\,|c_D|^2\,\frac{e^{2k}}{\log(k+2)}-C\,k^4e^{-2k}
      \longrightarrow \infty .
\end{aligned}
\]
whereas the slab integrals of the $L^2$ function $w$ are summable, hence tend to $0$: a contradiction.  Necessarily $c_D=0$.  Thus $w=c_Ru_R+w_p$ on $[y_1,\infty)$, and the bounds just displayed for $w_p$ and $w_p'$, together with the recessive bounds for $u_R$ and $u_R'$ from Lemma~\ref{lem:sector-LG} --- an extra factor $|q_\lambda|^{1/2}$ for the derivative --- give \eqref{eq:inhom-recessive-bound}.
\end{proof}

\subsection{Proof of Theorem~\ref{thm:D} and the interlaced ladder}

\begin{proof}[Proof of Theorem \ref{thm:D}: half-line part and reduction to parity/interlacing]
The proof is organized in two stages.  Steps 0--7 below prove the half-line realizations, spectral-line identity, nonzero overlaps, and half-line simplicity used later.  Lemma~\ref{lem:parity} and Proposition~\ref{prop:ladder} then perform the full-line parity splitting and interlacing, after which the short completion of Theorem~\ref{thm:D} records the resulting full-line statement.  Thus the later parity and ladder results use only the half-line conclusions already established in these steps.

We organize the proof in steps.  Fix $\gamma>0$ and
$\sharp\in\{N,D\}$.  Within Steps~0--7 we abbreviate the selected mode
$\phi^\sharp_{\gamma,n}$ to $\phi_n^\sharp$.

\emph{Step 0: the self-adjoint family.}  The Neumann and Dirichlet realizations of $H_\gamma^\sharp$ are obtained from the closed lower-bounded quadratic form
\[
    q_\gamma^\sharp[u]=\int_0^\infty |u'|^2+\gamma\int_0^\infty (\log y)_+|u|^2
       -\gamma\int_0^\infty (\log y)_-|u|^2,
\]
with form domain $H^1(\R_+)\cap L^2((\log y)_+dy)$ in the Neumann case and its trace-zero subspace in the Dirichlet case.  The negative part of $\log y$ lies in $L^1(0,1)$ and is infinitesimally form-bounded with respect to $\int_0^1(|u'|^2+|u|^2)$ --- indeed $\int_0^1(\log y)_-|u|^2\le\|(\log y)_-\|_{L^1(0,1)}\,\|u\|_{L^\infty(0,1)}^2$ and $\|u\|_{L^\infty(0,1)}^2\le\eps\|u'\|_{L^2(0,1)}^2+(1+\eps^{-1})\|u\|_{L^2(0,1)}^2$ for every $\eps\in(0,1)$ ---; hence the forms are closed and bounded below, and the first representation theorem gives the associated self-adjoint operators \cite[Ch.~VI, \S1.6--7 and \S2.1, pp.~319--325]{Kato}.  Since $(\log y)_+\to\infty$, the embedding of the form domain into $L^2(\R_+)$ is compact: truncate to $[0,R]$ and use Rellich there, while the tail is bounded by $(\log R)^{-1}\int_R^\infty(\log y)|u|^2$.  Thus $H_\gamma^\sharp$ has compact resolvent.  The endpoint $0$ is regular because $\log y\in L^1(0,1)$; $+\infty$ is limit point: at each fixed real spectral parameter choose the Liouville--Green tail in Lemma~\ref{lem:sector-LG} beyond all zeros of $\gamma\log y-E$ on the positive ray; on that tail the dominant solution grows superexponentially, so not every solution is square-integrable near $+\infty$.  The discarded finite interval is irrelevant to the Weyl endpoint alternative, and the endpoint is limit point \cite[Secs.~5--6, pp.~72--103]{Weidmann}.  Standard separated-endpoint Sturm--Liouville theory gives self-adjointness and simplicity of the eigenvalues $E^\sharp_{\gamma,0}<E^\sharp_{\gamma,1}<\cdots$, with real eigenfunctions $\phi_n^\sharp$.  The ground state in each separated self-adjoint problem is strictly positive on $(0,\infty)$ by Lemma~\ref{lem:standard-tools}\textup{(iv)}; the hypotheses of that result have just been verified.  The equality between the form operator and the displayed maximal separated differential realization is the content of Lemma~\ref{lem:log-form-realization}; in particular maximal-domain functions have finite traces at $0$, the separated condition eliminates the regular-endpoint boundary term, and the limit-point condition at $+\infty$ removes any second boundary condition.  By Lemma~\ref{lem:sector-LG} with $\alpha=0$, $\theta=0$, each $\phi_n^\sharp$ and its derivative have the stated superexponential Liouville--Green decay on a tail; local endpoint regularity supplies the boundedness needed on compact intervals.

\emph{Step 1: rigorous realization of $\cL^\sharp_\gamma$ and compactness.}  We first construct an operator with resolvent $\mathcal G$ and then verify that its domain is exactly the announced maximal separated domain.  Apply Lemma~\ref{lem:sector-LG} with $\alpha=\pi/2$, $E=-1$, on the ray $\theta=0$ (allowed: $\cos(\pi/4)>0$): the equation $-u''+(i\gamma\log y+1)u=0$ has a recessive solution $u_+\in L^2$ near $\infty$, unique up to constants, decaying superexponentially together with $u_+'$.  Extend $u_+$ to all of $(0,\infty)$ by the ODE; near the regular endpoint, writing the equation as a first-order system for $(u_+,u_+')$ with coefficient matrix in $L^1(0,1)$, Gr\"onwall gives $u_+,u_+'$ bounded near $0$, hence $u_+''\in L^1(0,\delta)$ and $u_+,u_+'\in AC[0,\delta]$ with finite traces at $0$.  Since the potential is integrable at $0$, fix the nonzero separated solution $u_\sharp$ by the normalization $u_N(0)=1$, $u_N'(0)=0$ in the Neumann case and $u_D(0)=0$, $u_D'(0)=1$ in the Dirichlet case; in either case it is $C^1$ up to $0$ and satisfies the $\sharp$-condition.  Their Wronskian $W:=u_+u_\sharp'-u_+'u_\sharp$ is constant on $(0,\infty)$.  If $W=0$, then $u_+$ satisfies the $\sharp$ boundary condition at $0$.  To see that this is impossible, integrate on $[\varepsilon,R]$:
\[
\begin{aligned}
0&=\Re\int_\varepsilon^R\bigl(-u_+''+(i\gamma\log y+1)u_+\bigr)\overline{u_+}\,dy\\
&=\int_\varepsilon^R(|u_+'|^2+|u_+|^2)\,dy+\Re[-u_+'\overline{u_+}]_\varepsilon^R .
\end{aligned}
\]
The boundary term at $R$ tends to zero by the recessive asymptotics of $u_+$ and $u_+'$; the term at $0$ tends to zero because the traces are finite and $u_+$ satisfies the separated boundary condition.  Letting $\varepsilon\downarrow0$ and $R\to\infty$ gives $u_+=0$, a contradiction.  So $W\neq0$ and
\[
    (\mathcal G f)(x)=\frac1W\Bigl(u_+(x)\int_0^xu_\sharp f+u_\sharp(x)\int_x^\infty u_+f\Bigr)
\]
is well defined on $L^2$.  With $\Xi(x)=\Re\xi(x)$ for $x\ge y_1$, extended to $(0,y_1]$ by the constant value $\Xi(y_1)$ --- a bounded extension, harmless in the exponential bounds below --- and $q_0(x):=1+\gamma\log_+x$, where $\log_+x:=\max(\log x,0)$ (so $q_0$ is comparable to $|q|$ for $x\ge2$): writing $u_\sharp=\alpha u_++\beta u_-^{(0)}$ on the positive real ray $[y_1,\infty)$, with $\beta\neq0$ precisely because $W\neq0$, gives $|u_\sharp(x)|\le Cq_0(x)^{-1/4}e^{\Xi(x)}$ for $x\ge2$; with $|u_+(y)|\le Cq_0(y)^{-1/4}e^{-\Xi(y)}$, the kernel obeys $|G(x,y)|\le Cq_0(x\wedge y)^{-1/4}q_0(x\vee y)^{-1/4}e^{-(\Xi(x\vee y)-\Xi(x\wedge y))}$ on the tail $x\vee y\ge y_1$.  After enlarging the constant, the same bound holds on the compact annulus $2\le x\vee y\le y_1$, because $q_0$, $\Xi$, $u_+$, and $u_\sharp$ are bounded there and $\Xi$ is extended constantly on $(0,y_1]$.  The kernel is bounded on $(0,2]^2$, both solutions being bounded there with finite traces at the regular endpoint.  The kernel bound just displayed, after replacing $\Xi$ by $\Xi-\Xi(y_1)$ for the normalized statement of Lemma~\ref{lem:schur-tail}, the extension of $\Xi$ --- nondecreasing, constant on $(0,y_1]$, and $C^1$ on $[y_1,\infty)$ with $\Xi'=\Re q^{1/2}\asymp q_0^{1/2}$ there, by the ray display of Lemma~\ref{lem:sector-LG} --- and the local square-integrability of $G$ --- on each $(0,R)^2$ the kernel is continuous off the diagonal and continuous across it, only the derivative jumping, hence bounded on compact rectangles and in $L^2((0,R)^2)$ for every finite $R$ --- are exactly the hypotheses \eqref{eq:schur-kernel-bound} of Lemma~\ref{lem:schur-tail}.  Hence the Schur norms of $G$ are finite, the truncation remainders obey the tail bound \eqref{eq:schur-tail-bound}, and $\mathcal G$ is bounded and \emph{compact} on $L^2$.  The range is dense already: if $\varphi\in C_c^\infty(0,\infty)$, then $\varphi$ satisfies either separated boundary condition vacuously and $f=-\varphi''+(i\gamma\log y+1)\varphi\in C_c^\infty(0,\infty)$; the same homogeneous-uniqueness argument as above gives $\mathcal Gf=\varphi$.  Thus $C_c^\infty(0,\infty)\subset\operatorname{Ran}\mathcal G$.  For $f\in C_c(0,\infty)$, differentiating the kernel --- the orientation of $W$ makes the derivative jump of $G(\cdot,y)$ across $x=y$ equal to $-1$ --- shows that $u=\mathcal Gf$ solves $-u''+(i\gamma\log y+1)u=f$ classically, with the $\sharp$-condition inherited from $u_\sharp$: near $0$, $\mathcal Gf=W^{-1}\bigl[u_+\int_0^{\,\cdot}u_\sharp f+u_\sharp\int_{\,\cdot}^\infty u_+f\bigr]$, and the first-order jump terms cancel upon differentiation.  The endpoint trace formula extends directly to all $f\in L^2$: since $u_+\in L^2(0,\infty)$ and $u_\sharp,u_\sharp'$ are bounded near $0$,
\[
    \int_0^x u_\sharp(y)f(y)\,dy=O(x^{1/2}\|f\|_2),\qquad
    \int_x^\infty u_+(y)f(y)\,dy\to\int_0^\infty u_+(y)f(y)\,dy
\]
as $x\downarrow0$, and the same estimate with $u_\sharp'$ gives the derivative trace after differentiating the displayed Green formula.  Hence
\[
    (\mathcal Gf)^{(k)}(0)=W^{-1}u_\sharp^{(k)}(0)\int_0^\infty u_+(y)f(y)\,dy,
    \qquad k=0,1,
\]
for every $f\in L^2$, so the boundary condition holds without relying on trace convergence.  For general $f\in L^2$, take $f_n\in C_c(0,\infty)$ with $f_n\to f$ in $L^2$: then $\mathcal Gf_n\to\mathcal Gf$ in $L^2$ by boundedness, and the identity passes to the limit in $\mathcal D'(0,\infty)$, so $-(\mathcal Gf)''+(i\gamma\log y+1)\mathcal Gf=f$ distributionally for every $f\in L^2$.  In particular $\mathcal Gf=0$ forces $f=0$: $\mathcal G$ is injective.  Define
\[
    \cL^\sharp_\gamma:=\mathcal G^{-1}-1,\qquad \mathcal D(\cL^\sharp_\gamma)=\Ran\mathcal G .
\]
By construction $\cL^\sharp_\gamma+1:\Ran\mathcal G\to L^2(0,\infty)$ is bijective with bounded inverse $\mathcal G$; hence $-1\in\rho(\cL^\sharp_\gamma)$, and this resolvent is compact because $\mathcal G$ is compact.  The operator is closed: if $u_k=\mathcal Gf_k\to u$ and $(\cL^\sharp_\gamma+1)u_k=f_k\to g$ in $L^2$, then $u=\lim\mathcal Gf_k=\mathcal Gg$ by boundedness of $\mathcal G$, so $u\in\Ran\mathcal G$ and $(\cL^\sharp_\gamma+1)u=g$.  This realization is the maximal one asserted in the theorem.  Let $\mathcal D_{\max}$ consist of all $u\in L^2(0,\infty)$ with $u,u'\in AC_{\mathrm{loc}}$, $-u''+i\gamma\log y\,u\in L^2$, and the $\sharp$-condition at $0$ (meaningful, since the endpoint is regular and $u,u'$ extend continuously to $0$).  We record the two domain facts explicitly.  First, $\Ran\mathcal G\subset\mathcal D_{\max}$: for $u=\mathcal Gf$ with $f\in L^2$, the distributional identity just proved gives $-u''+i\gamma\log y\,u=f-u\in L^2$; Cauchy--Schwarz with $\log y\in L^2(0,1)$ and $u\in L^2$ gives $(\log y)u\in L^1(0,\delta)$, so $u''\in L^1_{\mathrm{loc}}[0,\infty)$ and $u,u'\in AC_{\mathrm{loc}}[0,\infty)$, while the displayed trace formula gives the $\sharp$-condition for every $f\in L^2$.  Second, the homogeneous $L^2$ problem with the same boundary condition is unique after forcing is fixed: if $u\in\mathcal D_{\max}$, set $g=-u''+(i\gamma\log y+1)u\in L^2$ and $w=\mathcal Gg$; then $u-w$ is an $L^2$ solution of the homogeneous equation satisfying the $\sharp$-condition.  Lemma~\ref{lem:sector-LG} gives $u-w=c\,u_+$ on a positive tail interval; since both sides solve the same second-order equation on the connected interval $(0,\infty)$, ordinary Cauchy uniqueness propagates this equality to all of $(0,\infty)$.  If $c\neq0$, then $u_+$ itself satisfies the same separated $\sharp$-condition at $0$ as the reference solution $u_\sharp$; the two boundary data are therefore linearly dependent, so $W(u_+,u_\sharp)=0$, contradicting the previously established nonzero Wronskian.  Hence $c=0$, and $u=w$.  Therefore $\mathcal D(\cL^\sharp_\gamma)=\Ran\mathcal G=\mathcal D_{\max}$.   Then $\cL^\sharp_\gamma$ is closed with compact resolvent, its spectrum is discrete, and $\lambda$ is an eigenvalue if and only if the ODE $-u''+i\gamma\log y\,u=\lambda u$ has a nontrivial $L^2$ solution with the $\sharp$-condition; every such solution is recessive at $\infty$ (Lemma~\ref{lem:sector-LG}), so eigenfunctions and their derivatives decay superexponentially, all Green identities hold for them, and in particular $\Re\lambda=\|u'\|^2/\|u\|^2\ge0$.

\emph{Step 2: sectorial continuation, for every spectral parameter.}  Let $\alpha\in\{0,\pi/2\}$, $\lambda\in\C$, and let $u$ solve $u''=(e^{i\alpha}\gamma\Log z-\lambda)u$ on $(0,\infty)$.  The coefficient $c(z)=e^{i\alpha}\gamma\Log z-\lambda$ is holomorphic on $\C\setminus(-\infty,0]$; by interior regularity for linear ODEs with analytic coefficients, the $AC_{\mathrm{loc}}$ solution on the real ray is real-analytic on $(0,\infty)$, and its Cauchy data at any base point $y_*>0$ generate the unique holomorphic continuation along paths, so $u$ extends holomorphically to the open sector $\{z\neq0:\,|\arg z|<\pi/3\}$ --- a simply connected domain on which the coefficient is holomorphic --- and in particular is holomorphic on $\{|\arg z|\le\pi/4\}\setminus\{0\}$.  At the origin we use Lemma~\ref{lem:log-volterra-endpoint}.  Equivalently, writing $u''=c(y)u$, $c\in L^1(0,1)$, as a first-order system for $(u,u')$ with $L^1$ coefficient matrix, Gr\"onwall on $[\eps,\delta]$ gives bounds uniform as $\eps\downarrow0$, so $u,u'$ are bounded near $0$; then $cu\in L^1(0,\delta)$, hence $u''\in L^1(0,\delta)$ and $u,u'\in AC[0,\delta]$ with finite Cauchy data; and every solution satisfies the Volterra equation
\[
    u(z)=a_0+b_0z+\int_0^z(z-t)\,c(t)\,u(t)\,dt ,
\]
integrated along the segment $[0,z]$; since $\int_0^\delta\sup_{|\theta|\le\pi/4}|c(te^{i\theta})|\,dt\le\int_0^\delta\bigl(\gamma|\log t|+\gamma\pi/4+|\lambda|\bigr)dt\to0$ as $\delta\downarrow0$, the Picard iteration converges uniformly on $\{|z|\le\delta,\,|\arg z|\le\pi/4\}$ for $\delta$ small.  Consequently $u$ and $u'$ extend continuously to $z=0$ \emph{within the closed sector}, with ray-independent boundary values $u(0)=a_0$, $u'(0)=b_0$; in particular the $\sharp$-condition is preserved along every ray of the sector.  Applied with $\alpha=0$, $\lambda=E^\sharp_{\gamma,n}$, this continues $\phi_n^\sharp$ to the closed sector $\{0\le\arg z\le\pi/4\}$ --- holomorphically off $0$, continuously up to $0$ --- with the boundary values $\phi_n^\sharp(0),\ (\phi_n^\sharp)'(0)$.  By Lemma~\ref{lem:sector-LG} with $\alpha=0$ on $[0,\pi/4]$, the solution $\phi_n^\sharp$, which tends to $0$ along $\theta=0$, is recessive on the sectorial tail: for some $r_0\ge e$,
\[
    |\phi_n^\sharp(re^{i\theta})|\le C\exp\bigl(-cr(\log r)^{1/2}\bigr),\qquad r\ge r_0,\quad 0\le\theta\le\tfrac\pi4 .
\]
On the compact sector $0\le r\le r_0$ the Volterra continuation at the origin and ordinary ODE continuation away from the origin give continuity and boundedness; no Liouville--Green decay estimate is asserted there.

\emph{Step 3: the forward identity.}  Set $u_{\gamma,n}^\sharp(y)=\phi_n^\sharp(e^{i\pi/4}y)$.  Then $(u_{\gamma,n}^\sharp)''(y)=i\,(\phi_n^\sharp)''(e^{i\pi/4}y)$ and $\Log(e^{i\pi/4}y)=\log y+i\pi/4$, so from $(\phi_n^\sharp)''=(\gamma\Log z-E^\sharp_{\gamma,n})\phi_n^\sharp$,
\[
    -(u_{\gamma,n}^\sharp)''+i\gamma\log y\,u_{\gamma,n}^\sharp
    =\Bigl(\frac{\gamma\pi}{4}+iE^\sharp_{\gamma,n}\Bigr)u_{\gamma,n}^\sharp .
\]
By Step 2, $u_{\gamma,n}^\sharp\in L^2$ with superexponential decay, $u_{\gamma,n}^\sharp(0)=\phi_n^\sharp(0)$ and $(u_{\gamma,n}^\sharp)'(0)=e^{i\pi/4}(\phi_n^\sharp)'(0)$, so the $\sharp$-condition is preserved.  Hence $u_{\gamma,n}^\sharp$ is a genuine half-line eigenfunction: writing $g=(\lambda+1)u_{\gamma,n}^\sharp$ with $\lambda=\gamma\pi/4+iE^\sharp_{\gamma,n}$, both $u_{\gamma,n}^\sharp$ and $\mathcal Gg$ are $L^2$ solutions of the same inhomogeneous problem with the $\sharp$-condition, and their difference is an $L^2$ homogeneous solution at spectral value $-1$, hence $0$ by Step 1.  This proves ``$\supseteq$'' in \eqref{eq:log-osc-spec-gamma}.

\emph{Step 4: the converse and the vertical line.}  Let $\lambda$ be an eigenvalue with eigenfunction $u$.  By Step~2 with $\alpha=\pi/2$ (the sector $\{-\pi/4\le\arg z\le0\}$ is treated identically), $u$ extends holomorphically to a neighborhood of that closed sector punctured at $0$, continuously up to $0$, with ray-independent Cauchy data at $0$, so the $\sharp$-condition transports to the rotated ray.  By Lemma~\ref{lem:sector-LG} with $\alpha=\pi/2$ on the sector $[-\pi/4,0]$ --- admissible since $\theta+\pi/4\in[0,\pi/4]$ there, so $\cos(\theta+\pi/4)\ge2^{-1/2}$ on the whole closed interval --- $u$ is recessive on the tail of the closed sector down to $\arg z=-\pi/4$ ($u\in L^2$ along the positive ray, so the square-integrability clause of the lemma applies).  The compact part of the sector is controlled by the endpoint Volterra continuation and local ODE continuation.  Set $\phi(r)=u(e^{-i\pi/4}r)$; then $\phi''=-i\,u''(e^{-i\pi/4}r)$ and $\Log(e^{-i\pi/4}r)=\log r-i\pi/4$, so the eigenvalue equation transforms into
\[
    -\phi''+\gamma\log r\,\phi=-i\Bigl(\lambda-\frac{\gamma\pi}{4}\Bigr)\phi ,
\]
with $\phi\in L^2$ superexponentially decaying.  The boundary traces are preserved explicitly: $\phi(0)=u(0)$ and $\phi'(0)=e^{-i\pi/4}u'(0)$, so the Dirichlet or Neumann condition is transported to the real half-line.  Moreover $\phi\in L^2$ with $\phi,\phi'\in AC_{\mathrm{loc}}$ and $-\phi''+\gamma\log r\,\phi\in L^2$, so $\phi$ lies in the maximal separated domain, which is $\mathcal D(H^\sharp_\gamma)$ by the identification in Step~0.  Hence $E:=-i(\lambda-\gamma\pi/4)$ is an eigenvalue of the self-adjoint $H^\sharp_\gamma$, so $E=E^\sharp_{\gamma,n}$ is real, which forces $\Re\lambda=\gamma\pi/4$ and $\lambda=\gamma\pi/4+iE^\sharp_{\gamma,n}$.  This proves ``$\subseteq$'' in \eqref{eq:log-osc-spec-gamma} and Remark~\ref{rem:vertical}.

\emph{Step 5: the overlap and non-isotropy.}  Let $F(z)=\phi_n^\sharp(z)^2$, holomorphic on the punctured sector $\{0\le\arg z\le\pi/4\}\setminus\{0\}$ --- indeed on an open neighborhood of it, the coefficient being holomorphic off $(-\infty,0]$, while at the origin $\phi_n^\sharp$ is only $C^1$ --- and continuous up to $0$.  By Step~2, for some $r_0\ge e$ one has $|F(re^{i\theta})|\le Ce^{-cr(\log r)^{1/2}}$ uniformly for $r\ge r_0$ and $0\le\theta\le\pi/4$, while $F$ is bounded on the compact sector $0\le r\le r_0$.  Cauchy's theorem on the pie-slice $\{\eps\le r\le R,\ 0\le\theta\le\pi/4\}$ --- applicable since $\phi_n^\sharp$ is holomorphic on the open sector by Step~2 ---, with the outer arc contribution $\le CRe^{-cR(\log R)^{1/2}}\to0$ and the inner arc bounded by $O(\eps)\to0$ from compact continuity at the origin, gives
\begin{equation}\label{eq:contour-rotation}
    \int_0^\infty \bigl(u_{\gamma,n}^\sharp(y)\bigr)^2\,dy
    =e^{-i\pi/4}\int_0^\infty\phi_n^\sharp(r)^2\,dr\ \neq\ 0,
\end{equation}
since $\phi_n^\sharp$ is real and nontrivial.  For the full line,
$u_*(y)=\phi^N_{1,0}(e^{i\pi/4}|y|)$ is even and $C^1$ across $0$
(Neumann condition), solves the equation on $\R$, and
\eqref{eq:contour-rotation} doubles to \eqref{eq:log-osc-overlap}.

\emph{Step 6: simplicity.}  Geometric multiplicity is one: any eigenfunction is recessive at $\infty$, so Lemma~\ref{lem:sector-LG} makes it a scalar multiple of the recessive solution on a tail, and Cauchy uniqueness for the same second-order equation propagates that scalar relation through the whole interval.  If algebraic simplicity failed, a nontrivial Jordan chain of length at least two would have a first generalized vector whose image is a nonzero multiple of the eigenfunction; rescaling that eigenfunction reduces the obstruction to $(\cL^\sharp_\gamma-\lambda)w=u_{\gamma,n}^\sharp$ for some $w\in\mathcal D(\cL^\sharp_\gamma)$.  With $q_\lambda$, $\Xi_\lambda$, $q_0$, and $y_1$ as in Lemma~\ref{lem:inhom-recessive}, the datum $g=u_{\gamma,n}^\sharp$ is recessive in the sense required there, since $|u_{\gamma,n}^\sharp|\le C|q_\lambda|^{-1/4}e^{-\Xi_\lambda(y)}$ by Lemma~\ref{lem:sector-LG}, and the restriction of $w$ to $[y_1,\infty)$ satisfies the remaining hypotheses: $w\in L^2$, $w,w'\in AC_{\mathrm{loc}}$, and the equation reads $w''-q_\lambda w=-u_{\gamma,n}^\sharp$ there.  Hence, by Lemma~\ref{lem:inhom-recessive}, both $w$ and $w'$ decay superexponentially --- quantitatively, \eqref{eq:inhom-recessive-bound} holds --- and $w\in\mathcal D(\cL^\sharp_\gamma)$ carries the same $\sharp$-condition at the regular endpoint as $u_{\gamma,n}^\sharp$, so the boundary terms there vanish.  Hence the Green identity
\[
    \mathfrak b\bigl((\cL^\sharp_\gamma-\lambda)w,\,u_{\gamma,n}^\sharp\bigr)
      =\mathfrak b\bigl(w,\,(\cL^\sharp_\gamma-\lambda)
          u_{\gamma,n}^\sharp\bigr)=0,
    \qquad \mathfrak b(u,v)=\int_0^\infty uv,
\]
holds (integrate by parts twice on $[0,R]$ and let $R\to\infty$: the boundary term $(w'u_{\gamma,n}^\sharp-w\,(u_{\gamma,n}^\sharp)')\big|_0^R$ vanishes at $0$ for both conditions --- Neumann eliminates the derivatives, Dirichlet the values --- and tends to $0$ as $R\to\infty$ by \eqref{eq:inhom-recessive-bound} together with the recessive bounds for $u_{\gamma,n}^\sharp$ and $(u_{\gamma,n}^\sharp)'$ from Lemma~\ref{lem:sector-LG}).  Since $(\cL^\sharp_\gamma-\lambda)w=u_{\gamma,n}^\sharp$ by the normalization of the first generalized vector, the displayed identity gives $\mathfrak b(u_{\gamma,n}^\sharp,u_{\gamma,n}^\sharp)=0$, contradicting \eqref{eq:contour-rotation}.  This excludes every length-two Jordan chain; since geometric multiplicity is one, every longer chain would contain such a length-two subchain.  So every eigenvalue of $\cL^\sharp_\gamma$ is algebraically simple.

\emph{Step 7: the scaling law.}  If $-u''+\gamma\log y\,u=Eu$ with the $\sharp$-condition, set $v(r)=u(\gamma^{-1/2}r)$: then
\[
    -v''+\log r\,v=\Bigl(\gamma^{-1}E+\tfrac12\log\gamma\Bigr)v ,
\]
with the $\sharp$-condition preserved, so $E^\sharp_{1,n}=\gamma^{-1}E^\sharp_{\gamma,n}+\tfrac12\log\gamma$, which is \eqref{eq:log-osc-spec-scaled}.

\emph{Step 8: passage to the full line.}  Steps~1--7 give closedness,
compact resolvent, the exact spectrum, algebraic simplicity, and nonzero
bilinear overlap on each half-line.  Lemma~\ref{lem:parity} identifies the
maximal full-line realization with the Neumann--Dirichlet parity sum.
Proposition~\ref{prop:ladder} then proves strict interlacing, simplicity of
the merged ladder, and the all-mode overlap.  Together these statements
complete the full-line part of Theorem~\ref{thm:D}.
\end{proof}

\begin{remark}
The rotation $y\mapsto e^{i\pi/4}y$ is a bijection between the
eigenfunctions of $H^\sharp_\gamma$ and those of $\cL^\sharp_\gamma$.
The real shift $\gamma\pi/4$ comes from the logarithmic branch term,
whereas the angle $\pi/4$ is forced by the kinetic phase
$e^{2i\theta}=i$.  This rigidity is the analytic source of the universal
coefficients in Theorem~\ref{thm:E}.
\end{remark}

\begin{lemma}[Full-line realization and parity splitting]\label{lem:parity}
Fix $\gamma>0$ and let $\cL_\gamma$ be the maximal distributional realization on $\R$,
\[
    \mathcal D(\cL_\gamma)=\bigl\{u\in L^2(\R):\ -u''+i\gamma\log|y|\,u\in L^2(\R)\ \text{in }\mathcal D'(\R)\bigr\}.
\]
Then every $u\in\mathcal D(\cL_\gamma)$ has $u,u'\in AC_{\mathrm{loc}}(\R)$ --- in particular $u\in C^1(\R)$ and the equation holds across $0$ --- so this domain coincides with the $AC_{\mathrm{loc}}$ realization described before Theorem~\ref{thm:D}: $u\in L^2(\R)$, $u,u'\in AC_{\mathrm{loc}}(\R)$, and $-u''+i\gamma\log|y|\,u\in L^2(\R)$.  Parity commutes with $\cL_\gamma$.  If $u=u_{\mathrm e}+u_{\mathrm o}$ is its even/odd decomposition, then
\[
 \mathcal Uu:=\bigl(\sqrt2\,u_{\mathrm e}|_{\R_+},
                     \sqrt2\,u_{\mathrm o}|_{\R_+}\bigr)
\]
defines a unitary map from $L^2(\R)$ to
$L^2(\R_+)\oplus L^2(\R_+)$ satisfying
\[
 \mathcal U\mathcal D(\cL_\gamma)
   =\mathcal D(\cL^N_\gamma)\oplus\mathcal D(\cL^D_\gamma),
 \qquad
 \mathcal U\cL_\gamma\mathcal U^{-1}
   =\cL^N_\gamma\oplus\cL^D_\gamma.
\]
Consequently $\cL_\gamma$ is closed with compact resolvent, and $\spec(\cL_\gamma)$ is the union of the two families \eqref{eq:log-osc-spec-gamma}, each eigenvalue simple within its parity sector.  The proof uses only the half-line facts recorded before the full-line transition in the proof of Theorem~\ref{thm:D} and no statement of Proposition~\ref{prop:ladder} below.
\end{lemma}

\begin{proof}
For $u\in\mathcal D(\cL_\gamma)$, writing $h:=-u''+i\gamma\log|y|\,u\in L^2(\R)$: since $\log|y|\in L^2_{\mathrm{loc}}$ and $u\in L^2$, Cauchy--Schwarz gives $(\log|y|)u\in L^1_{\mathrm{loc}}$, so the distribution $u''=i\gamma\log|y|\,u-h$ lies in $L^1_{\mathrm{loc}}(\R)$, hence $u,u'\in AC_{\mathrm{loc}}(\R)$, $u\in C^1(\R)$, and the equation holds across $0$; this identifies the domain with the displayed $AC_{\mathrm{loc}}$ realization.  Parity commutes with $\cL_\gamma$; an even element restricts to $(0,\infty)$ with $u'(0)=0$, hence lies in the Neumann maximal domain of Step~1 of the proof of Theorem~\ref{thm:D}, an odd element has $u(0)=0$ and lies in the Dirichlet one, and conversely the even (odd) reflection of a Neumann (Dirichlet) element is $C^1$ across $0$ and belongs to $\mathcal D(\cL_\gamma)$.  Thus the normalized restriction map $\mathcal U$ in the statement is unitary, maps the full domain onto the direct sum of the separated half-line domains, and intertwines the operators as displayed.  Since each half-line realization has compact resolvent by Step~1, the finite direct sum has compact resolvent; hence so does the unitarily equivalent full-line realization.  Therefore $\spec(\cL_\gamma)$ is the union of the two half-line spectra, i.e.\ of the two families \eqref{eq:log-osc-spec-gamma} by Steps~3--4, and within each parity sector every eigenvalue is simple by Step~6.
\end{proof}

\begin{proposition}[The interlaced ladder]\label{prop:ladder}
Fix $\gamma>0$.
\begin{enumerate}[label=\textup{(\roman*)}]
\item The Neumann and Dirichlet spectra of $H^\sharp_\gamma$ are disjoint and strictly interlace:
\[
    E^N_{\gamma,0}<E^D_{\gamma,0}<E^N_{\gamma,1}<E^D_{\gamma,1}<\cdots .
\]
\item Write $\bar E_{\gamma,0}<\bar E_{\gamma,1}<\cdots$ for the merged sequence, so that $\bar E_{\gamma,2m}=E^N_{\gamma,m}$ and $\bar E_{\gamma,2m+1}=E^D_{\gamma,m}$.  Then
\[
    \spec(\cL_\gamma)=\bigl\{\nu_{\gamma,n}:=\tfrac{\gamma\pi}4+i\bar E_{\gamma,n}\ :\ n\ge0\bigr\}
\]
on the full line, and every eigenvalue is algebraically simple.  Use the
mode $u_n:=u_{\gamma,n}$ fixed in Theorem~\ref{thm:D}; it is the parity
extension of the corresponding rotated half-line mode.  As for $\phi_n$
below, the fixed $\gamma$ is suppressed from the notation within this
proposition, and $u_n(-y)=(-1)^nu_n(y)$ on $\R$.  With
$\phi_n:=\phi^N_{\gamma,m}$ for $n=2m$ and
$\phi_n:=\phi^D_{\gamma,m}$ for $n=2m+1$ (the half-line self-adjoint
eigenfunction of the $n$-th rung, with the same scalar normalization),
\begin{equation}\label{eq:ladder-overlap}
    \int_\R u_n(y)^2\,dy=2e^{-i\pi/4}\int_0^\infty\phi_n(\rho)^2\,d\rho\ \neq\ 0,
    \qquad\text{so }\ \delta(u_n)>0 .
\end{equation}
For every $c\in\C\setminus\{0\}$,
$\delta(cu_n)=\delta(u_n)>0$; thus non-isotropy is independent of the
normalization.
\item For $\gamma=1$ write $\bar E_n=\bar E_{1,n}$, $\nu_n=\pi/4+i\bar E_n$, and, for $K\in\N_0$,
\[
    g_K:=\min_{0\le n\le K}\bigl(\bar E_{n+1}-\bar E_n\bigr)>0 .
\]
The closed discs $\overline D(\nu_n,g_K/3)$, $0\le n\le K$, are pairwise disjoint and each meets $\spec(\cL_1)$ exactly in $\{\nu_n\}$; moreover $g_0=\bar E_1-\bar E_0$ is the principal full-line gap above $\nu_*$.
\end{enumerate}
\end{proposition}

\begin{proof}
\emph{(i).}  Both operators are self-adjoint extensions of the minimal operator generated by $-\partial_y^2+\gamma\log y$ on $C_c^\infty(0,\infty)$.  The endpoint $0$ is regular since $\log\in L^1(0,1)$, and $+\infty$ is limit point: for each fixed spectral parameter the Liouville--Green analysis of Lemma~\ref{lem:sector-LG}, restricted to a positive-axis tail chosen beyond any turning point, produces a dominant/recessive solution basis with superexponential growth/decay; endpoint classification is tail-based, so only one solution is square-integrable at $+\infty$.  Hence the minimal operator has deficiency indices $(1,1)$: at the regular endpoint both solutions are square-integrable near $0$ (limit circle), at $+\infty$ exactly one (limit point).

The hypotheses of Lemma~\ref{lem:rank-one-extension-input} have now been verified: Lemma~\ref{lem:log-form-realization} supplies the closed semibounded Neumann and Dirichlet forms, their codimension-one form-domain inclusion, and compact resolvent, while the preceding Liouville--Green argument supplies the limit-point property at $+\infty$.  Fix
\[
    \mu<\min\{\inf\spec H^N_\gamma,\inf\spec H^D_\gamma\}.
\]
Applying Lemma~\ref{lem:rank-one-extension-input} with $V(y)=\gamma\log y$ shows that $(H^N_\gamma-\mu)^{-1}-(H^D_\gamma-\mu)^{-1}$ has rank one and that the two counting functions differ by at most one at every energy.  \emph{Disjointness:} at a common eigenvalue the eigenfunctions of both problems would be multiples of the unique solution square-integrable at $+\infty$, which would then satisfy $\phi(0)=\phi'(0)=0$ and vanish identically.  \emph{Order:} the Dirichlet form domain is contained in the Neumann one, so the min--max principle gives $E^N_{\gamma,m}\le E^D_{\gamma,m}$, strictly by disjointness.  \emph{Interlacing:} if $E^D_{\gamma,m}>E^N_{\gamma,m+1}$ for some $m$, the interval $(-\infty,E^N_{\gamma,m+1}]$ would contain $m+2$ Neumann eigenvalues but at most $m$ Dirichlet ones, contradicting the counting bound; so $E^D_{\gamma,m}<E^N_{\gamma,m+1}$, again strictly by disjointness.

\emph{(ii).}  By Lemma~\ref{lem:parity} --- proved independently of the present proposition from the half-line facts established in Steps~1--6 of the proof of Theorem~\ref{thm:D}: half-line closedness and compact resolvent, the half-line spectral identity, algebraic simplicity, and the contour-rotation overlap --- $\cL_\gamma$ splits under parity into the two half-line problems and its spectrum is the union of the two families, i.e.\ the merged ladder; by (i) the union is disjoint: each full-line eigenvalue lies in exactly one parity sector (even for Neumann rungs, odd for Dirichlet rungs), where its eigenspace is one-dimensional by Step~6.  The full-line Riesz projection is the parity-direct sum of the two sector projections, only one of which is nonzero, of rank one by Step~6; hence every eigenvalue of $\cL_\gamma$ is algebraically simple.  For the overlap, the even extension $u_n(y)=\phi_n(e^{i\pi/4}|y|)$, respectively the odd extension $u_n(y)=\operatorname{sgn}(y)\,\phi_n(e^{i\pi/4}|y|)$, is $C^1$ across $0$ by the Neumann, respectively Dirichlet, condition and solves the eigenvalue equation on $\R$; in both cases $u_n^2$ is even, so \eqref{eq:contour-rotation} doubles to \eqref{eq:ladder-overlap}, and $\delta(u_n)=|\int u_n^2|/\int|u_n|^2>0$.

\emph{(iii).}  Positivity of $g_K$ and the disc statement are immediate from (i) and (ii).  For the last claim, interlacing gives $\bar E_0=E^N_{1,0}$ and $\bar E_1=E^D_{1,0}$, whence
\[
    g_0=\bar E_1-\bar E_0=E^D_{1,0}-E^N_{1,0}.
\]
The next Neumann level satisfies $E^D_{1,0}<E^N_{1,1}$ by interlacing, but it is not part of the definition of the adjacent gap $g_0$.
\end{proof}

\begin{proof}[Completion of the proof of Theorem~\ref{thm:D}]
The half-line assertions are exactly Steps~1--7 above.  Lemma~\ref{lem:parity} unitarily identifies the full-line maximal realization with the parity direct sum $\cL^N_\gamma\oplus\cL^D_\gamma$; since the half-line resolvents are compact, this finite direct sum has compact resolvent, and hence so does the full-line realization.  Proposition~\ref{prop:ladder} gives the strict Neumann--Dirichlet interlacing, disjointness of the two parity families, algebraic simplicity of every full-line rung, and the overlap formula.  In particular, for $\gamma=1$ the principal rung is $\nu_* = \pi/4+iE_0^N$, it is isolated from the rest of the ladder by $g_0=E^D_{1,0}-E^N_{1,0}>0$, and its even eigenfunction satisfies \eqref{eq:log-osc-overlap}.  Combining these statements with the half-line spectral identity, the rotation formula for the eigenfunctions, the half-line algebraic simplicity and overlap identity, and the scaling law proves all assertions of Theorem~\ref{thm:D}.
\end{proof}

\begin{lemma}[Model rung decay and logarithmic moments]\label{lem:model-rung-decay}
Fix $\gamma>0$ and a full-line rung $u_{\gamma,n}$ of $\cL_\gamma$.  Then, for some constants $c,C>0$ depending on $\gamma$ and $n$,
\[
    |u_{\gamma,n}(y)|+|u'_{\gamma,n}(y)|\le C\exp\bigl(-c|y|(\log(e+|y|))^{1/2}\bigr),\qquad |y|\ge2.
\]
In particular, for every $m\ge0$,
\[
    \int_\R (1+|\log|y||)^m|u_{\gamma,n}(y)|^2\,dy<\infty,
    \qquad
    \int_\R (1+|\log|y||)^m|u'_{\gamma,n}(y)|^2\,dy<\infty,
\]
with the local logarithmic singularity understood at $y=0$.
\end{lemma}

\begin{proof}
On each half-line, $u_{\gamma,n}$ is the rotated continuation of a real half-line Sturm--Liouville eigenfunction, and Lemma~\ref{lem:sector-LG} on the corresponding sector gives the stated superexponential bound for the eigenfunction and its derivative.  Lemma~\ref{lem:parity} then reflects the estimate to the full line.  Near $0$, the endpoint is regular in the sense recorded in Lemma~\ref{lem:log-volterra-endpoint} and Lemma~\ref{lem:log-form-realization}; hence $u_{\gamma,n}$ and $u'_{\gamma,n}$ are locally bounded on each side of $0$ and have finite traces.  Since $|\log|y||^m$ is integrable at $0$ for every finite $m$, the displayed moment bounds follow.
\end{proof}

\begin{remark}
Proposition~\ref{prop:ladder} closes the description of the model: the vertical line of Remark~\ref{rem:vertical} is a \emph{simple} ladder, the two boundary conditions supplying alternate rungs.  Every rung, not only the lowest, is phase-rigid in the sense of \eqref{eq:ladder-overlap}; this is what permits the rung-by-rung spectral transfer of Theorem~\ref{thm:E}\textup{(a)}.
\end{remark}

\section{Logarithmic resonances: proof of Theorem \ref{thm:E}}\label{sec:log-resonance}

We now close the paper's two branches.  The local mass of the cusp first
recovers, through Theorem~\ref{thm:B}, the predicted global resolvent
scale.  The blow-up below then transfers the complete model ladder to
$T_s$, and the pencil results of Section~\ref{sec:pencil} turn it into
genuine wave spectrum.  The principal rung finally meets the global upper
bound in Lemma~\ref{lem:envelope}, producing the sharp regularized decay
envelope.

Throughout this section $a$ is as in Theorem~\ref{thm:E}, with the representative fixed there: $0\le a\in L^\infty(\T)$, $a(x)=(\log(e/|x|))^{-A}$ for $0<|x|<x_0$ with $0<x_0<1$.  Fix once and for all a constant $a_->0$ with $a\ge a_-$ a.e.\ on $\{|x|\ge x_0\}$, for instance any number below the essential infimum there.  Note first that $\Theta_a(r)\asymp(\log(e/r))^{-A}$ for $0<r\le x_0/2$, uniformly over the position of the window.  \emph{Lower bound:} for any $x$, at least half of the window $(x-r,x+r)$ lies in $\{|t|\ge r/2\}$, where the original-scale dummy variable $t$ satisfies $a(t)\ge\min\bigl((\log(2e/r))^{-A},a_-\bigr)\ge c(\log(e/r))^{-A}$ --- the first value where the window meets the cusp chart, the second where it does not --- so the average is $\ge\frac c2(\log(e/r))^{-A}$.  \emph{Upper bound:} at the centered window, $(2r)^{-1}\int_{-r}^ra=(\log(e/r))^{-A}(1+O(1/\log(1/r)))$ by integration by parts.  Hence $a$ is two-sidedly saturated at $L(S)=(\log(eS))^A$, and Theorem~\ref{thm:B} applies.

\subsection{The resonance scale}
Let $R_s$ and $\ell_s=\log(e/R_s)$ be defined by \eqref{eq:Rs-main}.  Substituting $R_s=e^{1-\ell_s}$ gives $2\ell_s=\log s-(A+1)\log\ell_s+O(1)$, whence
\begin{equation}\label{eq:ell-asymp}
    \ell_s=\frac12\log s-\frac{A+1}{2}\log\log s+O(1),
    \qquad
    sR_s^2=\frac{\ell_s^{A+1}}{A}\ \longrightarrow\ \infty .
\end{equation}

\subsection{The scaled operator and pointwise potential comparison}
Rescale $x=R_sy$ via the unitary $U_s:L^2(\T)\to L^2(\T_s)$, $(U_sf)(y)=R_s^{1/2}f(R_sy)$; the circle becomes $\T_s=\R/(2\pi R_s^{-1})\Z$ and the cusp chart is $\{|y|<x_0/R_s\}$.  Since $sR_s^2\ell_s^{-A}=\ell_s/A$ by \eqref{eq:Rs-main}, define
\begin{equation}\label{eq:scaled-Ts}
    \cT_s:=R_s^2\,U_sT_sU_s^{-1}-\frac{i\ell_s}{A}=-\partial_y^2+iV_s(y)
    \quad\text{on }L^2(\T_s),
\end{equation}
where $V_s(y)=sR_s^2\bigl(a(R_sy)-\ell_s^{-A}\bigr)$.  In particular $\|(T_s-\lambda)^{-1}\|=R_s^2\|(\cT_s-\zeta)^{-1}\|$ under $\lambda=R_s^{-2}(\zeta+i\ell_s/A)$, and $\delta(U_sv)=\delta(v)$, so tameness and non-isotropy transport exactly.  More generally, for a fixed damping $c\in L^\infty(\T)$ we write $V^c_s(y):=sR_s^2\,c(R_sy)-\ell_s/A$ and $\cT^c_s:=-\partial_y^2+iV^c_s$ on $L^2(\T_s)$.  Thus $\cT_s=\cT^a_s$ and the perturbed operators $\cT^{\tilde a}_s$ below are covered verbatim by Lemma~\ref{lem:complex-symmetry}.  All displayed chart formulas and comparisons for these potentials are statements about $L^\infty$ classes, i.e.\ hold essential-a.e.\ on $0<|y|<x_0/R_s$ for the representative of $a$ fixed at the start of the section; the value at the single point $y=0$ is arbitrary and irrelevant, multiplication operators depending only on the class.  Off-chart lower bounds, such as \eqref{eq:offchart-coercive}, are likewise essential-a.e.\ statements, i.e.\ essential-infimum bounds, which is all the integral estimates below use.  On the chart, $a(R_sy)=(\ell_s-\log|y|)^{-A}$, so with $m=\log|y|$,
\[
    V_s(y)=\frac{\ell_s}{A}\Bigl[\Bigl(1-\frac m{\ell_s}\Bigr)^{-A}-1\Bigr].
\]
Convexity of $x\mapsto(1-x)^{-A}$ on $(-\infty,1)$ gives $(1-x)^{-A}\ge1+Ax$, and $|(1-x)^{-A}-1|\le A|x|$ for $x\le0$ while $(1-x)^{-A}-1\le Ax\,2^{A+1}$ for $0\le x\le\frac12$.  Hence the \emph{two-sided pointwise comparison}
\begin{equation}\label{eq:pointwise-comparison}
    \begin{gathered}
    |V_s(y)|\le\bigl|\log|y|\bigr|\ \ (0<|y|\le1),\quad
    \log|y|\le V_s(y)\ \ (1\le|y|<x_0/R_s),\\
    V_s(y)\le2^{A+1}\log|y|\ \ (1\le|y|\le e^{\ell_s/2}),
    \end{gathered}
\end{equation}
and, from $|(1-x)^{-A}-1-Ax|\le C_Ax^2$ for $|x|\le\frac12$, the local expansion
\begin{equation}\label{eq:local-coeff-expansion}
    \bigl|V_s(y)-\log|y|\bigr|\ \le\ C_A\,\frac{(\log|y|)^2}{\ell_s},
    \qquad e^{-\ell_s/2}\le|y|\le e^{\ell_s/2}.
\end{equation}
Off the chart, $a\ge a_-$.  Since $\ell_s/A=sR_s^2\ell_s^{-A}=o(sR_s^2)$, the negative shift in $V_s=sR_s^2a(R_s\,\cdot)-\ell_s/A$ is negligible there, and for large $s$
\begin{equation}\label{eq:offchart-coercive}
    V_s(y)\ge a_-sR_s^2-\ell_s/A\ge \frac{a_-}{2}\,sR_s^2=\frac{a_-}{2A}\,\ell_s^{A+1}
    \qquad(|y|\ge x_0/R_s).
\end{equation}
Thus the imaginary potential of $\cT_s$ dominates $\log|y|$ everywhere outside the unit window and is uniformly controlled by $|\log|y||$ inside it; on compact subsets of $\R\setminus\{0\}$, $V_s\to\log|y|$ in $L^\infty$, while on compact sets meeting $0$ the convergence is only in the tested local-integral sense proved in Proposition~\ref{prop:log-scaling}, and the limiting operator is the imaginary logarithmic oscillator $\cL_1=-\partial_y^2+i\log|y|$ of Theorem~\ref{thm:D}, whose spectrum is the simple interlaced ladder $\{\nu_n=\pi/4+i\bar E_n\}_{n\ge0}$ with finite-rung gaps $g_K>0$ for each $K\in\N_0$ and non-isotropic eigenfunctions $u_n$ (Proposition~\ref{prop:ladder}); here $\nu_0=\nu_*=\pi/4+iE_0^N$ and $u_0=u_*$.

\subsection{Riesz projection convergence}

We represent $\T_s$ by the fundamental interval $(-\pi/R_s,\pi/R_s]$ and write $\iota_s:L^2(\T_s)\to L^2(\R)$ for the associated isometric zero-extension; every statement of strong, local, or tight convergence for families on $\T_s$ refers to $\iota_su_s$ in the fixed space $L^2(\R)$, and $\iota_s$ is suppressed from the notation.  For any fixed $\varphi\in C_c^\infty(\R)$ and all large $s$, $\operatorname{supp}\varphi\subset(-\pi/R_s,\pi/R_s)$ and lies in the cusp chart, so circle pairings against $\varphi$ coincide with real-line pairings; all identifications of $\T_s$-pairings with $\R$-pairings for functions of unbounded support are invoked only after the exterior tail estimate and Lemma~\ref{lem:tight-pairings}.  Here and below, $|y|$ denotes the distance on $\T_s$ to the well at $0$, and radial cutoffs such as $\psi_M$ are functions of this distance --- well defined once $4M\le\pi R_s^{-1}$, i.e.\ for $s$ large after $M$ is fixed; off the chart, including the wraparound region of the circle, the coercivity \eqref{eq:offchart-coercive} applies.  All sequences considered below will be shown tight, so this notation is harmless.  For $n\ge0$ let $\hat u_n=u_n/\|u_n\|_{L^2(\R)}$; in particular $\hat u_0=u_*/\|u_*\|_{L^2(\R)}$.

\begin{lemma}[Singular-sequence criterion for the rescaled problem]\label{lem:singular-sequence-criterion}
Assume that each $\cT_s$ is closed with compact resolvent.  Suppose that for every compact $\mathcal K\Subset\C$, every sequence $s_k\to\infty$, $\zeta_k\in\mathcal K$, and $u_k\in\mathcal D(\cT_{s_k})$ with $\|u_k\|=1$ and $(\cT_{s_k}-\zeta_k)u_k\to0$, there is a subsequence, still denoted $u_k$, and a point $\zeta_*\in\mathcal K$ with $\zeta_k\to\zeta_*$ along this subsequence, such that $u_k\to u_\infty$ strongly in $L^2(\R)$, $\|u_\infty\|=1$, $u_\infty\in\mathcal D(\cL_1)$, and $(\cL_1-\zeta_*)u_\infty=0$.  Then every compact subset of $\rho(\cL_1)$ is eventually contained in $\rho(\cT_s)$ with a uniform resolvent bound.

Moreover, if $\Gamma$ is a positively oriented simple closed contour contained in $\rho(\cL_1)$ with winding number one about $\lambda$, and $w_s\in\mathcal D(\cT_s)$ is a family with $\|w_s-w\|\to0$, $w\ne0$, and $(\cT_s-\lambda)w_s\to0$, while $\Gamma$ has the uniform resolvent bound just obtained, then
\[
    \|\Pi_\Gamma(\cT_s)w_s-w_s\|\longrightarrow0,
\]
where $\Pi_\Gamma(\cT_s)$ denotes the Riesz projection.  In particular this projection is nonzero for all large $s$.
\end{lemma}

\begin{proof}
If the uniform resolvent conclusion failed on a compact $\mathcal K\subset\rho(\cL_1)$, then, after passing to a sequence, either some $\zeta_k\in\mathcal K$ belongs to $\spec(\cT_{s_k})$, which by compact resolvent is an eigenvalue and gives an exact unit eigenvector, or else $\zeta_k\in\rho(\cT_{s_k})$ and $\|(\cT_{s_k}-\zeta_k)^{-1}\|\to\infty$.  In the latter case choose $g_k$ with $\|g_k\|=1$ and $\|(\cT_{s_k}-\zeta_k)^{-1}g_k\|\ge\tfrac12\|(\cT_{s_k}-\zeta_k)^{-1}\|$, and set
\[
    u_k=\frac{(\cT_{s_k}-\zeta_k)^{-1}g_k}{\|(\cT_{s_k}-\zeta_k)^{-1}g_k\|};
\]
then $\|u_k\|=1$ and $(\cT_{s_k}-\zeta_k)u_k\to0$.  The hypothesis would give a nonzero eigenfunction of $\cL_1$ with eigenvalue in $\mathcal K$, a contradiction.  For the projection statement, use
\[
    \Pi_\Gamma(\cT_s)w_s-w_s=\frac1{2\pi i}\oint_\Gamma
    (\zeta-\cT_s)^{-1}(\cT_s-\lambda)w_s\,(\zeta-\lambda)^{-1}\,d\zeta .
\]
The uniform resolvent bound on $\Gamma$ makes the right-hand side $o(1)$, while $\|w_s\|\to\|w\|>0$; hence $\Pi_\Gamma(\cT_s)w_s\ne0$ for $s$ large.

\end{proof}

\begin{lemma}[Tight convergence of pairings]\label{lem:tight-pairings}
Let $u_k,v_k\in L^2(\T_{s_k})$ be families represented on $(-\pi/R_{s_k},\pi/R_{s_k}]$ and extended by zero to $\R$.  Suppose
\[
    \sup_k(\|u_k\|_2+\|v_k\|_2)<\infty,
    \qquad
    \lim_{M\to\infty}\limsup_{k\to\infty}\int_{|y|>M}(|u_k|^2+|v_k|^2)=0,
\]
and $u_k\to u$, $v_k\to v$ locally in $L^2(\R)$.  Then $u_k\to u$ and $v_k\to v$ strongly in $L^2(\R)$ and
\[
    \int_{\T_{s_k}}u_kv_k\to\int_{\R}uv,
    \qquad
    \langle u_k,v_k\rangle_{L^2(\T_{s_k})}\to\langle u,v\rangle_{L^2(\R)}.
\]
The stronger uniform-tail hypothesis with $\sup_k$ is a special case.  Discarding finitely many initial indices is harmless; this is the form used below, where exterior estimates are uniform only after $s$ exceeds a threshold depending on the fixed exterior radius.
\end{lemma}

\begin{proof}
We first record that the local limits inherit the sequential tail bounds.
For fixed $0<M<N$, local $L^2$ convergence gives
\[
 \int_{M<|y|<N}|u|^2
   =\lim_{k\to\infty}\int_{M<|y|<N}|u_k|^2
   \le\limsup_{k\to\infty}\int_{|y|>M}|u_k|^2 .
\]
Letting $N\to\infty$ and using monotone convergence yields
\[
 \int_{|y|>M}|u|^2
   \le\limsup_{k\to\infty}\int_{|y|>M}|u_k|^2,
\]
and the same argument applies to $v$.  Given $\varepsilon>0$, choose
$M$ so large that the sum of the two limsup tails in the hypothesis is
less than $\varepsilon$.  For all sufficiently large $k$, the tails of
$u_k$ and $v_k$ beyond this $M$ are less than $2\varepsilon$ in total,
while the preceding inequalities give the corresponding bound for $u$
and $v$.  Local convergence on $[-M,M]$ and
\[
 \|u_k-u\|_{L^2(|y|>M)}^2
 \le2\|u_k\|_{L^2(|y|>M)}^2
     +2\|u\|_{L^2(|y|>M)}^2
\]
(and similarly for $v_k-v$) now prove global strong $L^2$ convergence.
Finally, after identifying the circle functions with their zero
extensions, the bilinear convergence follows from
\[
 \left|\int u_kv_k-\int uv\right|
 \le\|u_k-u\|_2\|v_k\|_2+\|u\|_2\|v_k-v\|_2,
\]
and the identical estimate with conjugation proves the sesquilinear
convergence.
\end{proof}

\begin{lemma}[Finite-dimensional root-space alternative]\label{lem:root-space-alt}
Let $A$ be a closed operator with compact resolvent and let $\Gamma\subset\rho(A)$ be a positively oriented contour enclosing a finite spectral set $\Sigma_\Gamma\subset\spec(A)$.  Let
\[
    P_\Gamma=\frac{1}{2\pi i}\oint_\Gamma (\zeta-A)^{-1}\,d\zeta
\]
be the corresponding Riesz projection.  If $\operatorname{rank}P_\Gamma>1$, then at least one of the following alternatives holds for the finite-dimensional root space $\Ran P_\Gamma$:
\begin{enumerate}[label=\textup{(\alph*)}]
\item $A$ has at least two distinct eigenvalues $\mu_1,\mu_2$ inside $\Gamma$;
\item some enclosed eigenvalue $\mu\in\Sigma_\Gamma$ has geometric multiplicity at least two;
\item some enclosed eigenvalue $\mu\in\Sigma_\Gamma$ has a Jordan chain of length at least two, i.e. there are $u,w\in\Ran P_\Gamma$ with $u\neq0$, $(A-\mu)u=0$, and $(A-\mu)w=u$.
\end{enumerate}
The alternatives are exhaustive but need not be mutually exclusive.
\end{lemma}

\begin{proof}
The restriction $A|_{\Ran P_\Gamma}$ is a finite-dimensional operator
whose spectrum is precisely $\Sigma_\Gamma$, counted with algebraic
multiplicity, and the assertion is its Jordan normal form.  In the third
alternative a nonzero scalar multiple of the leading chain vector has
been absorbed into $u$.  The parameter there is therefore an actual
enclosed eigenvalue, not necessarily the center of the contour.
\end{proof}

\begin{proposition}[Localized resolvent bounds and eigenvalue convergence]\label{prop:log-scaling}
Fix $K\in\N_0$ and let $\nu_n$, $u_n$, $g_K$ be as in Proposition~\ref{prop:ladder}.
\begin{enumerate}[label=\textup{(\roman*)}]
\item For every compact $\mathcal K\subset\C$ with $\mathcal K\cap\spec(\cL_1)=\emptyset$ there is $s(\mathcal K)$ such that $\mathcal K\subset\rho(\cT_s)$ and $\sup_{\zeta\in\mathcal K}\|(\cT_s-\zeta)^{-1}\|\le C(\mathcal K)$ for all $s\ge s(\mathcal K)$;
\item for each $0\le n\le K$ and $s$ large, the Riesz projection $\Pi_{\Gamma_n}(\cT_s)$ over $\Gamma_n=\{|\zeta-\nu_n|=g_K/3\}$ has rank one; the unique eigenvalue $\mu_s^{(n)}$ of $\cT_s$ in $D(\nu_n,g_K/3)$ is algebraically simple and $\mu_s^{(n)}\to\nu_n$;
\item let $\omega_s^{(n)}$ be a normalized eigenfunction of $\cT_s$ associated with $\mu_s^{(n)}$.  For $s$ large, after choosing the phase so that $\langle \omega_s^{(n)},\hat u_n\rangle$ is real and positive, the eigenfunctions $\omega_s^{(n)}$ converge to $\hat u_n$ in $L^2(\R)$, and
\begin{equation}\label{eq:nonisotropy-transfer}
    \delta(\omega_s^{(n)})\ \longrightarrow\ \delta(u_n)>0 .
\end{equation}
\end{enumerate}
\end{proposition}

\begin{proof}
We verify the hypotheses of Lemma~\ref{lem:singular-sequence-criterion}: every singular sequence --- unit $u_s$ with $\|(\cT_s-\zeta_s)u_s\|\to0$, $\zeta_s\to\zeta_*$ in a fixed compact set --- is tight and converges along a subsequence to a unit maximal-domain eigenfunction of $\cL_1$ at $\zeta_*$.  On compact sets disjoint from $\spec(\cL_1)$ this is impossible, and Lemma~\ref{lem:singular-sequence-criterion} gives the uniform resolvent bounds of \textup{(i)}.

\emph{Uniform a priori structure.}  For each fixed $s$, $V_s\in L^\infty(\T_s)$, so $\cT_s$ is closed on $H^2(\T_s)$ with compact resolvent; in particular its spectrum is discrete and the Riesz projections below are well defined.  Let $u\in H^2(\T_s)$ with $\|u\|=1$, $\zeta$ in a fixed compact set, and $f=(\cT_s-\zeta)u$ with $\|f\|\le C_*$.  Pairing with $u$ and taking real parts (the potential is purely imaginary) gives the global gradient bound $\|u'\|_{L^2(\T_s)}^2\le\|f\|+|\Re\zeta|\le C$.  For a fixed cutoff $\chi$, the IMS identity $\Re\langle(\cT_s-\zeta)u,\chi^2u\rangle=\|(\chi u)'\|^2-\|\chi'u\|^2-\Re\zeta\|\chi u\|^2$ gives uniform local $H^1$ bounds --- explicitly $\sup_{I\subset\T_s,\,|I|\le2}\|u\|_{H^1(I)}\le C$, the constant independent of the period --- hence local $L^\infty$ bounds in one dimension, by the Sobolev embedding on unit intervals, and so \emph{no concentration at the singular point}: $\int_{|y|<\rho}|u|^2\le C\rho$.  For the exterior, fix $M\ge2$ and take $s$ so large that $4M\le\pi R_s^{-1}$: the radial cutoff $\psi_M$ --- a smooth function of the distance $|y|$ to the well, $\psi_M=1$ on $\{|y|\ge2M\}$, $\psi_M=0$ on $\{|y|\le M\}$, $|\psi_M'|\le2/M$ --- is then well defined on $\T_s$, and on $\operatorname{supp}\psi_M$ either the chart comparison \eqref{eq:pointwise-comparison} or the off-chart bound \eqref{eq:offchart-coercive} applies.  Taking imaginary parts of $\langle(\cT_s-\zeta)u,\psi_M^2u\rangle$ and using $|\Im\langle-u'',\psi_M^2u\rangle|\le2\|\psi_M'\|_\infty\|u'\|\|u\|\le C/M$,
\[
    \int V_s\,\psi_M^2|u|^2\ \le\ |\Im\zeta|+\|f\|+\frac CM\ \le\ C .
\]
By \eqref{eq:pointwise-comparison} on the chart part of $\operatorname{supp}\psi_M$ and, after enlarging the threshold to $s\ge s(M)$ so that the off-chart coercive bound \eqref{eq:offchart-coercive} dominates $\log M$ as well, $V_s\ge\log M$ on all of $\operatorname{supp}\psi_M$, so
\begin{equation}\label{eq:exterior-small}
    \int_{|y|\ge2M}|u|^2\ \le\ \frac{C}{\log M},
    \qquad\text{uniformly in }s\ge s(M);
\end{equation}
the finitely many smaller $s$ are irrelevant in the compactness arguments below: given any sequence $s_k\to\infty$, for each fixed $M$ all but finitely many $k$ satisfy $s_k\ge s(M)$, the tail bound applies to those, and $M\to\infty$ is taken only afterwards; this diagonal order is the only way \eqref{eq:exterior-small} is used.
Together these give \emph{tightness}: any such family is precompact in $L^2(\R)$ (fix $M\in\N$ and extract, by Rellich and the uniform local $H^1$ bounds, a subsequence converging in $L^2(-M,M)$; diagonalize over $M$; the sequential tail bound supplied by \eqref{eq:exterior-small} and the uniform smallness of neighborhoods of $0$ then upgrade local to strong $L^2(\R)$ convergence), and norms pass to the limit, so any $L^2$-limit $u_\infty$ has $\|u_\infty\|=1$.  Moreover, for $\varphi\in C_c^\infty(\R)$, split the integral $\int(V_s-\log|y|)u\bar\varphi$ at $|y|=e^{-\ell_s/2}$, where the expansion \eqref{eq:local-coeff-expansion} takes over: on the inner region, \eqref{eq:pointwise-comparison} gives $|V_s-\log|y||\le2|\log|y||$, so that contribution is at most $C\,\|u\|_{L^\infty(\{|y|\le1\})}\,\|\varphi\|_\infty\int_{|y|\le e^{-\ell_s/2}}\bigl|\log|y|\bigr|\,dy\to0$, the first factor uniformly bounded by the local $L^\infty$ bounds; on the outer part of $\operatorname{supp}\varphi$, \eqref{eq:local-coeff-expansion} bounds the contribution, after Cauchy--Schwarz, by $C_A\,\ell_s^{-1}\,\|u\|_{2}\,\bigl\|(\log|y|)^2\varphi\bigr\|_{2}\to0$, the fourth power of the logarithm being locally integrable.  Hence $\int(V_s-\log|y|)u\bar\varphi\to0$; so if, after zero extension to $L^2(\R)$, the right-hand sides $f_s$ are uniformly bounded and $f_s\rightharpoonup f_\infty$ weakly along a subsequence, while $\zeta_s\to\zeta_*$, then --- the derivative term passing because both derivatives fall on the test function, $\int u\,\bar\varphi''\to\int u_\infty\bar\varphi''$ by local $L^2$ convergence, and the resulting identity holding against every $\varphi\in C_c^\infty(\R)$, including test functions straddling $y=0$, so that no interface condition arises at the origin and the limit lies in the maximal full-line domain --- the limit solves
\[
    -u_\infty''+i\log|y|\,u_\infty=\zeta_*u_\infty+f_\infty
    \quad\text{in }\mathcal D'(\R) .
\]
In the singular-sequence case used below, $f_s=(\cT_s-\zeta_s)u_s\to0$ strongly in $L^2$ after zero extension, hence $f_\infty=0$.  Then $u_\infty$ is an $L^2$ distributional solution of the eigenvalue equation; since $\log|y|\in L^2_{\mathrm{loc}}$ and $u_\infty\in L^2$, the product is locally integrable, so $u_\infty''=(i\log|y|-\zeta_*)u_\infty\in L^1_{\mathrm{loc}}$ and $u_\infty,u_\infty'\in AC_{\mathrm{loc}}(\R)$; moreover $-u_\infty''+i\log|y|\,u_\infty=\zeta_*u_\infty\in L^2(\R)$, which is precisely membership in the maximal domain of $\cL_1$ described in Lemma~\ref{lem:parity}; by Lemma~\ref{lem:sector-LG}, applied to $u_\infty(y)$ and to the reflection $u_\infty(-y)$ on $y>0$ --- both are $L^2$ solutions of the same logarithmic half-line equation --- it is recessive at $\pm\infty$: an eigenfunction of $\cL_1$ at $\zeta_*$.

\emph{(i).}  Since $\spec(\cL_1)$ is closed and discrete, any compact $\mathcal K$ disjoint from it satisfies $\dist(\mathcal K,\spec\cL_1)>0$.  If (i) failed for such a compact $\mathcal K$ there would be $s_k\to\infty$ and $\zeta_k\in\mathcal K$ along which either $\zeta_k\in\spec(\cT_{s_k})$ or $\|(\cT_{s_k}-\zeta_k)^{-1}\|\ge k$; since $\cT_s$ has compact resolvent, $\cT_s-\zeta$ is Fredholm of index $0$, so the first case yields a unit eigenfunction $w_k$ and the second a unit near-maximizer $w_k=(\cT_{s_k}-\zeta_k)^{-1}u_k/\|(\cT_{s_k}-\zeta_k)^{-1}u_k\|$; either way $\|w_k\|=1$ and $(\cT_{s_k}-\zeta_k)w_k\to0$.  By the above, a subsequential limit is a unit eigenfunction of $\cL_1$ with eigenvalue in $\mathcal K$, contradicting $\dist(\mathcal K,\spec\cL_1)>0$.

\emph{(ii), rank $\ge1$.}  Fix $0\le n\le K$; note $\dist(\Gamma_n,\spec\cL_1)\ge g_K/3$.  Let $\chi_s(y):=\chi\bigl(e^{-\ell_s/4}|y|\bigr)$ with a fixed $\chi\in C_c^\infty([0,\infty))$, $\chi=1$ on $[0,1]$, $\operatorname{supp}\chi\subset[0,2]$ --- so $\chi_s=1$ on $\{|y|\le e^{\ell_s/4}\}$, $\operatorname{supp}\chi_s\subset\{|y|\le2e^{\ell_s/4}\}$, $\|\chi_s'\|_\infty=O(e^{-\ell_s/4})$, and $\|\chi_s''\|_\infty=O(e^{-\ell_s/2})$ --- and put $w_s=\chi_su_n$: this lies in $H^2(\T_s)$, since by Lemma~\ref{lem:parity} $u_n,u_n'\in AC_{\mathrm{loc}}(\R)$ and, from the eigenvalue equation, $u_n''=(i\log|y|-\nu_n)u_n\in L^2_{\mathrm{loc}}(\R)$ ($\log|y|\in L^2_{\mathrm{loc}}$ and $u_n$ locally bounded), so $u_n\in H^2_{\mathrm{loc}}(\R)$, while $\chi_s$ is smooth and constant near $0$.  For $s$ large, $2e^{\ell_s/4}<\min(\pi,x_0)R_s^{-1}$ --- indeed $e^{\ell_s/4}R_s=e\,e^{-3\ell_s/4}\to0$ --- so $\operatorname{supp}w_s$ lies well inside both the fundamental interval and the cusp chart; extended by zero, $w_s$ is a genuine element of $H^2(\T_s)$; its support is compactly contained in the chosen fundamental interval, so the endpoint traces of $w_s$ and $w_s'$ vanish and the periodic $H^2$ compatibility is automatic.  Every potential comparison used on $\operatorname{supp}w_s$ is the chart estimate.  By the superexponential decay and logarithmic moment bounds of Lemma~\ref{lem:model-rung-decay} and the potential comparison --- split at $|y|=e^{-\ell_s/2}$: on the inner region, \eqref{eq:pointwise-comparison} gives $|V_s-\log|y||\le2|\log|y||$, and $\|u_n\|_\infty<\infty$ with $\int_0^{e^{-\ell_s/2}}(\log y)^2dy\to0$ makes that contribution vanish; on $e^{-\ell_s/2}\le|y|\le2e^{\ell_s/4}$, the expansion \eqref{eq:local-coeff-expansion} applies, and the weighted bound $\|(V_s-\log|y|)w_s\|_{L^2}^2\le C_A^2\,\ell_s^{-2}\int_{e^{-\ell_s/2}\le|y|\le2e^{\ell_s/4}}(1+|\log|y||)^4\,|u_n(y)|^2\,dy=O(\ell_s^{-2})$ absorbs the pointwise growth $O(\ell_s)$ at both edges of the region into the finite moment $\int_\R(1+|\log|y||)^4|u_n|^2\,dy<\infty$, $u_n$ being bounded near $0$ and superexponentially small at infinity ---,
\[
    \|(\cT_s-\nu_n)w_s\|\le\bigl\|\chi_s(V_s-\log|y|)u_n\bigr\|+\|\chi_s''u_n\|+2\|\chi_s'u_n'\|\ \longrightarrow\ 0 ,
\]
while $w_s\to u_n$ strongly in $L^2(\R)$ after zero extension --- $\chi_s\to1$ pointwise and $|(1-\chi_s)u_n|\le|u_n|\in L^2(\R)$, so dominated convergence applies --- in particular $\|w_s\|\to\|u_n\|\neq0$.  The projection part of Lemma~\ref{lem:singular-sequence-criterion}, with $\Gamma=\Gamma_n$, $\lambda=\nu_n$, $w=u_n$, and (i) applied to $\mathcal K=\Gamma_n$, gives $\|\Pi_{\Gamma_n}(\cT_s)w_s-w_s\|\to0$, so $\Pi_{\Gamma_n}(\cT_s)\neq0$ for $s$ large.

\emph{(ii), rank $\le1$ and simplicity.}  Note that $\cT_s$ is complex symmetric on the circle in the closed-operator sense above (Lemma~\ref{lem:complex-symmetry}(i), applied on $\T_s$); in particular, $\mathfrak b(\cT_su,v)=\mathfrak b(u,\cT_sv)$ with $\mathfrak b(u,v)=\int_{\T_s}uv$ and no boundary terms.  All eigenvalues enclosed by $\Gamma_n$ lie in the fixed compact $\overline D(\nu_n,g_K/3)$, so the uniform a priori estimates from the opening paragraph of this proof apply, with $f=0$ and uniform constants, to every eigenfunction appearing in the alternatives below; in particular every normalized such eigenfunction is sequentially tight and has, along a subsequence, a strong $L^2(\R)$ limit of norm one --- a unimodular multiple of $\hat u_n$, the full-line eigenspace at $\nu_n$ being one-dimensional (Proposition~\ref{prop:ladder}(ii)).  By part \textup{(i)}, $\Gamma_n\subset\rho(\cT_s)$ for all large $s$, so the Riesz projection is defined and has finite rank.  By Lemma~\ref{lem:root-space-alt}, a Riesz range of rank at least two forces one of three configurations: two distinct eigenvalues in the disc, one eigenvalue of geometric multiplicity at least two, or a Jordan chain of length at least two (hence a length-two initial segment); each is excluded in turn.  Suppose along $s_k\to\infty$ the disc $D(\nu_n,g_K/3)$ contained two eigenvalues $\mu_k^{(1)}\neq\mu_k^{(2)}$ with unit eigenfunctions $\varphi_k^{(1)},\varphi_k^{(2)}$.  Each family is tight and, after passing to subsequences and adjusting phases, $\varphi_k^{(j)}\to e^{i\theta_j}\hat u_n$ in $L^2(\R)$ (the limits are unit eigenfunctions of $\cL_1$ with eigenvalue in $\overline D(\nu_n,g_K/3)$, hence equal to $\nu_n$, whose full-line eigenspace is one-dimensional by Proposition~\ref{prop:ladder}(ii)).  But $\mathfrak b$-symmetry gives $\mathfrak b(\varphi_k^{(1)},\varphi_k^{(2)})=0$ for distinct eigenvalues, while $\mathfrak b$ passes to the limit by Lemma~\ref{lem:tight-pairings}: after passing to a common subsequence, the a priori structure, applied with $f=0$ and $\zeta=\mu_k^{(j)}$ in the fixed disc, gives the sequential tail bound and local convergence simultaneously for both families.  Hence
\[
    0=\lim_k\mathfrak b(\varphi_k^{(1)},\varphi_k^{(2)})=e^{i(\theta_1+\theta_2)}\,\frac{\int_\R u_n^2}{\|u_n\|^2}\neq0
\]
by \eqref{eq:ladder-overlap}: contradiction.  If instead a single eigenvalue $\mu_k$ had geometric multiplicity two, take orthonormal eigenfunctions $\varphi_k^{(1)}\perp\varphi_k^{(2)}$: both families are tight with $L^2$ limits $e^{i\theta_j}\hat u_n$, and Lemma~\ref{lem:tight-pairings}, now for the sesquilinear product, gives $0=\langle\varphi_k^{(1)},\varphi_k^{(2)}\rangle\to e^{i(\theta_1-\theta_2)}\neq0$: contradiction.  Finally, a Jordan chain $(\cT_{s_k}-\mu_k)w_k=\varphi_k$ (the generalized vector $w_k$ varying with $k$) forces $\mathfrak b(\varphi_k,\varphi_k)=0$ at each fixed $k$ by Lemma~\ref{lem:complex-symmetry}(iii) (exact on the circle); no compactness of the $w_k$ is used and only the eigenfunctions $\varphi_k$ are passed to the limit, while $\mathfrak b(\varphi_k,\varphi_k)\to e^{2i\theta_1}\int u_n^2/\|u_n\|^2\neq0$: contradiction.  Hence exactly one eigenvalue $\mu_s^{(n)}\in D(\nu_n,g_K/3)$, algebraically simple.  To identify its limit directly, let $s_k\to\infty$ be any sequence along which $\mu_{s_k}^{(n)}\to\mu$ and take normalized eigenfunctions.  The a priori structure above, with $f=0$ and $\zeta=\mu_{s_k}^{(n)}$ in the fixed disc, gives a subsequence converging strongly in $L^2(\R)$ to a unit eigenfunction of $\cL_1$ at $\mu$.  But $\overline D(\nu_n,g_K/3)$ meets $\spec(\cL_1)$ only at $\nu_n$, hence $\mu=\nu_n$.  Since every convergent subsequence has this limit, $\mu_s^{(n)}\to\nu_n$.

\emph{(iii).}  The normalized eigenfunctions $\omega_s^{(n)}$ are sequentially tight, and every subsequence has a further subsequence converging in $L^2(\R)$ to a unit eigenfunction of $\cL_1$ at $\nu_n$, i.e.\ to $e^{i\theta}\hat u_n$ for some $\theta$.  The overlap $\langle \omega_s^{(n)},\hat u_n\rangle$ is therefore nonzero for all sufficiently large $s$: otherwise a subsequence with zero overlap would converge to $e^{i\theta}\hat u_n$, whose overlap with $\hat u_n$ has modulus one.  Thus the phase normalization is legitimate for large $s$, and under it every subsequential limit equals $\hat u_n$ itself.  Hence the full family converges, $\omega_s^{(n)}\to\hat u_n$ in $L^2(\R)$, and $\delta(\omega_s^{(n)})=|\mathfrak b(\omega_s^{(n)},\omega_s^{(n)})|\to|\mathfrak b(\hat u_n,\hat u_n)|=\delta(u_n)>0$ by \eqref{eq:ladder-overlap}.
\end{proof}

\subsection{The low ladder of \texorpdfstring{$T_s$}{Ts}}

\begin{proposition}[The low ladder of \texorpdfstring{$T_s$}{Ts}]\label{prop:log-Ts}
Fix $K\in\N_0$.  For $s$ large, $T_s$ has algebraically simple eigenvalues $\lambda_s^{(0)},\dots,\lambda_s^{(K)}$,
\begin{equation}\label{eq:lambda-s-asymp}
    \lambda_s^{(n)}=is\,\ell_s^{-A}+As\,\ell_s^{-A-1}\,\mu_s^{(n)},
    \qquad \mu_s^{(n)}=\nu_n+o(1),\quad 0\le n\le K,
\end{equation}
whose normalized eigenfunctions may be taken as $v_s^{(n)}=U_s^{-1}\omega_s^{(n)}$, with $\omega_s^{(n)}$ the normalized rescaled eigenfunctions of Proposition~\ref{prop:log-scaling}, and satisfy $\delta(v_s^{(n)})\ge\delta_K>0$.  Each is the only spectrum of $T_s$ in the disc $\{|\lambda-\lambda_s^{(n)}|\le\rho_s\}$ with $\rho_s=(g_KA/6)\,s\,\ell_s^{-A-1}$, and on the circle $|\lambda-\lambda_s^{(n)}|=\rho_s$ the tameness bound \eqref{eq:tameness} holds with a constant $C_t=C_t(K)$ independent of $s$.  Moreover, if $B\subset\C$ is a closed rectangle with $\partial B\cap\spec(\cL_1)=\emptyset$ and $K$ is large enough that $\spec(\cL_1)\cap B\subset\{\nu_0,\dots,\nu_K\}$, then for $s\ge s(B)$
\begin{equation}\label{eq:Ts-completeness}
    \spec(T_s)\cap\bigl\{is\,\ell_s^{-A}+As\,\ell_s^{-A-1}\,\zeta:\ \zeta\in B\bigr\}
    =\bigl\{\lambda_s^{(n)}:\ \nu_n\in B\bigr\}.
\end{equation}
\end{proposition}

\begin{proof}
The unitary $U_s$ conjugates $T_s$ to $R_s^{-2}(\cT_s+i\ell_s/A)$, so $\lambda=R_s^{-2}(\zeta+i\ell_s/A)$ maps $\spec(\cT_s)$ onto $\spec(T_s)$, preserving algebraic multiplicities, and $R_s^{-2}=As\ell_s^{-A-1}$ by \eqref{eq:Rs-main}.  Proposition~\ref{prop:log-scaling}(ii) gives, for each $0\le n\le K$, the unique simple eigenvalue $\mu_s^{(n)}\to\nu_n$ of $\cT_s$ in $D(\nu_n,g_K/3)$, which is \eqref{eq:lambda-s-asymp}.  The circle $|\lambda-\lambda_s^{(n)}|=\rho_s$ corresponds to $|\zeta-\mu_s^{(n)}|=g_K/6$, which for large $s$ (when $|\mu_s^{(n)}-\nu_n|\le g_K/12$) lies inside the compact set $\{g_K/12\le|\zeta-\nu_n|\le g_K/4\}$, disjoint from $\spec(\cL_1)$; hence, by Proposition~\ref{prop:log-scaling}(i),
\[
    \|(T_s-\lambda)^{-1}\|=R_s^2\,\|(\cT_s-\zeta)^{-1}\|\le CR_s^2=\frac{Cg_K/6}{\rho_s},
\]
which is \eqref{eq:tameness} with $C_t=Cg_K/6$, while the disc $\{|\zeta-\mu_s^{(n)}|\le g_K/6\}$ lies in $D(\nu_n,g_K/3)$, where $\mu_s^{(n)}$ is the only spectrum.  The unitary dilation $U_s^{-1}$ leaves $\delta$ invariant, so \eqref{eq:nonisotropy-transfer} gives $\delta(v_s^{(n)})=\delta(\omega_s^{(n)})\ge\delta_K:=\tfrac12\min_{0\le n\le K}\delta(u_n)$ for $s$ large.

For \eqref{eq:Ts-completeness}: the set $B\setminus\bigcup_{\nu_n\in B}D(\nu_n,g_K/3)$ is compact and disjoint from the closed discrete set $\spec(\cL_1)$ --- hence at positive distance from it --- so by Proposition~\ref{prop:log-scaling}(i) it is free of $\spec(\cT_s)$ for $s$ large, while each disc $\overline D(\nu_n,g_K/3)$ meets $\spec(\cT_s)$ exactly in $\{\mu_s^{(n)}\}$; and $\mu_s^{(n)}\to\nu_n$, which lies in the interior of $B$ when $\nu_n\in B$ (since $\partial B\cap\spec(\cL_1)=\emptyset$) and outside $\overline B$ otherwise.  Mapping by $\lambda=R_s^{-2}(\zeta+i\ell_s/A)$ gives \eqref{eq:Ts-completeness}.
\end{proof}

\subsection{Proof of Theorem \ref{thm:E}}

\begin{proof}[Proof of Theorem~\ref{thm:E}\textup{(a)--(b)}]

\emph{Part (a).}  Fix a closed rectangle $B$ with
$\partial B\cap\spec(\cL_1)=\emptyset$, and choose $K$ so that every
model rung in $B$ belongs to $\{\nu_0,\dots,\nu_K\}$.  For each such rung,
Proposition~\ref{prop:log-Ts} supplies an algebraically simple eigenvalue
$\lambda_s^{(n)}$ of $T_s$, isolated at
\[
    \rho_s=\frac{g_KA}{6}s\ell_s^{-A-1},
\]
with a uniform tameness constant and
$|\lambda_s^{(n)}|\le\Lambda_s:=2s\ell_s^{-A}$ for large $s$.  We have
$(\Lambda_s+\rho_s)/(2s)\le2\ell_s^{-A}\le1$ and, uniformly for the
finitely many rungs under consideration,
\[
    \eps_s=C_t(\|a\|_\infty+1)\frac{\Lambda_s+\rho_s}{2s\rho_s}\le C\,\frac{\ell_s}{s}\to0,
    \quad
    4C_t(\Lambda_s+\rho_s)\eps_s\le C\,\ell_s^{\,1-A}=o(\rho_s),
\]
since $\ell_s^{\,1-A}/(s\ell_s^{-A-1})=\ell_s^2/s\to0$.
Thus Proposition~\ref{prop:pencil-transfer} gives an algebraically simple
block eigenvalue $z_j^{(n)}$ for every $\nu_n\in B$.  Using
\eqref{eq:lambda-s-asymp} in \eqref{eq:resonance-inclusion}, and
$\ell_s^{1-A}/s=o(\ell_s^{-A-1})$, gives
\eqref{eq:ladder-asymptotic}; its real and imaginary parts are
\eqref{eq:ladder-real-imag}.

It remains to exclude additional roots.  Put
\[
    \mathfrak B_s(B)
      =\{is\ell_s^{-A}+As\ell_s^{-A-1}\beta:\beta\in B\}.
\]
By \eqref{eq:Ts-completeness}, this window contains exactly the
$T_s$-eigenvalues $\lambda_s^{(n)}$ with $\nu_n\in B$.  On its boundary,
the scaling identity \eqref{eq:scaled-Ts} and
Proposition~\ref{prop:log-scaling}(i) give
\[
    \|(T_s-\lambda)^{-1}\|
       =R_s^2\|(\cT_s-\beta)^{-1}\|\le C_BR_s^2,
    \qquad \beta\in\partial B.
\]
Moreover $|\lambda|\le 2s\ell_s^{-A}$ for large $s$, and hence
\begin{align*}
 &\sup_{\lambda\in\partial\mathfrak B_s(B)}
       \|(T_s-\lambda)^{-1}\|
       \frac{|\lambda|}{2s}
       \left(\|a\|_\infty+\frac{|\lambda|}{2s}\right)\\
 &\qquad\le C_B(\|a\|_\infty+1)\,R_s^2\,\ell_s^{-A}
    =\frac{C_B(\|a\|_\infty+1)}{A}\,\frac{\ell_s}{s}
      \longrightarrow 0 .
\end{align*}
Proposition~\ref{prop:cluster-transfer}, applied to
$\operatorname{int}\mathfrak B_s(B)$, therefore identifies the total
pencil-root count with the $T_s$ eigenvalue count.  The simple roots
already constructed exhaust that count, and the affine change
$\lambda=-2is(z-is)$ identifies the resulting window with
$\mathcal W_j(B)$.  This proves the finite-window completeness in
\textup{(a)}.

\medskip
\emph{Part (b).}  Choose a small closed rectangle containing $\nu_0$ and
no other model rung, and apply \textup{(a)}.  This gives
\eqref{eq:log-resonance-asymptotic}; taking real and imaginary parts with
$\nu_*=\pi/4+iE_0^N$ gives \eqref{eq:log-real-imag} and
\eqref{eq:log-imag-shift}.  In particular
$|z_j|\asymp s_j$ and $-\Re z_j\lesssim(\log s_j)^{-A}$.

The mass computation at the start of the section and
Theorem~\ref{thm:B} give the global generator-resolvent upper bound.
Since $a$ is nontrivial, Lemma~\ref{lem:axis-cleanness} gives
$i\R\subset\rho(\cA)$, in particular $0\in\rho(\cA)$.  For each distinct
large transverse frequency, a scalar pencil eigenfunction for the
principal rung embeds, by Lemma~\ref{lem:blocks}, as a genuine eigenvector
of the full generator.  These frequencies are dyadically syndetic by
\S\ref{subsec:notation}, so
Lemma~\ref{lem:envelope}(a)--(c) combines the resolvent upper bound with
this eigenvalue family and yields \eqref{eq:sharp-envelope}.  This proves
\textup{(b)}.  Part \textup{(c)} is proved after the stable-cusp transfer
below.
\end{proof}

\subsection{Stable cusp perturbations}

The exact cusp identifies the model cleanly, but the ladder is not an
artifact of exact symmetry.  This subsection verifies that the three
properties used above---local convergence to $\log|y|$, coercive tails,
and compactly supported model quasimodes---survive arbitrary measurable
relative perturbations of lower logarithmic order.  Once those properties
are established, the same Riesz and pencil transfers give the identical
limiting ladder and leading constants.

We use the replacement-damping convention introduced after Lemma~\ref{lem:blocks}: for a real bounded $d$, $\cA_d$, $\cA_{d,j}$, and $P_d$ denote the generator, its transverse blocks, and the associated block pencil with $a$ replaced by $d$.  Here $d=\tilde a$ is always a fixed nonnegative damping with positive integral.

For a stable cusp perturbation $\tilde a$ with a chosen chart radius $x_1$ as in Lemma~\ref{lem:stable-cusp-estimates}, define its exterior lower-bound datum by
\begin{equation}\label{eq:stable-exterior-datum}
    m_{\tilde a}(\rho):=\operatorname*{ess\,inf}_{\{|x|\ge\rho\}}\tilde a,
    \qquad 0<\rho\le x_1 .
\end{equation}
The hypothesis below is precisely $m_{\tilde a}(\rho)>0$ for every such $\rho$.  When constants or thresholds are said to depend on exterior lower bounds, they depend on the relevant values of this function $m_{\tilde a}$ for the fixed exterior radii used in the proof; this datum contains no upper bound for $\tilde a$.

\begin{lemma}[Stable cusp estimates]\label{lem:stable-cusp-estimates}
With $A>0$ and $x_0\in(0,1)$ fixed as in Theorem~\ref{thm:E}, let $\sigma>0$ and let $0\le\tilde a\in L^\infty(\T)$ admit a representative such that, for some $x_1\in(0,x_0]$, some $C_W>0$, and some measurable $w$,
\[
    \tilde a(x)=(\log(e/|x|))^{-A}(1+w(x)),\qquad
    |w(x)|\le C_W(\log(e/|x|))^{-1-\sigma}
\]
for a.e. $0<|x|<x_1$, and assume $\operatorname*{ess\,inf}_{\{\rho\le |x|\}}\tilde a>0$ for every $\rho\in(0,x_1]$.  Let $a_0(x)=(\log(e/|x|))^{-A}$ denote the local cusp model on $0<|x|<x_1$, and define its \emph{local} rescaled potential by
\[
 V_{s,\mathrm{loc}}^{a_0}(y):=sR_s^2a_0(R_sy)-\frac{\ell_s}{A},
 \qquad 0<R_s|y|<x_1.
\]
Every occurrence of $V_{s,\mathrm{loc}}^{a_0}$ below is restricted to
this chart; no global extension of $a_0$ is intended or needed, and its
value at $y=0$ may be chosen arbitrarily.  After possibly shrinking the cusp chart to $|x|<x_*$, with $x_*>0$ depending only on $x_1,C_W,\sigma,A$, the rescaled potentials $V_s^{\tilde a}=sR_s^2\tilde a(R_sy)-\ell_s/A$ satisfy, in the essential-a.e. sense for this representative, the precise replacements for the cusp estimates used in Proposition~\ref{prop:log-scaling}.  Namely, for $s$ large, the following are $L^\infty_{\mathrm{ess}}$ or essential-infimum estimates on the indicated sets:
\begin{align}
    |V_s^{\tilde a}(y)|&\le C\bigl(1+|\log|y||\bigr),
        &&0<|y|\le1, \label{eq:stable-inner-bound}\\
    |V_s^{\tilde a}(y)-\log|y||&\le
        C\left(\frac{(\log|y|)^2}{\ell_s}+\ell_s^{-\sigma}\right),
        &&e^{-\ell_s/2}\le |y|\le e^{\ell_s/2}, \label{eq:stable-local-expansion}\\
    V_s^{\tilde a}(y)&\ge \log M-C,
        &&|y|\ge M,
\end{align}
for every fixed $M\ge2$ and all sufficiently large $s=s(M)$; here $|y|$ is the distance on the rescaled circle $\T_s$ to $0$, represented on $(-\pi/R_s,\pi/R_s]$, and the estimates are essential-a.e. on the corresponding set.  The constant $C$ in these three estimates may be chosen as $C=C(A,x_1,C_W,\sigma)$, independent of $M$, $s$, and the exterior lower-bound datum; only the threshold may depend on the exterior datum $m_{\tilde a}$ in \eqref{eq:stable-exterior-datum} at the relevant fixed radii.  In particular, after increasing this same constant $C$ if necessary, there is $M_0\ge2$ depending only on $A,x_1,C_W,\sigma$ such that for every $M\ge M_0$ and all sufficiently large $s=s(M)$,
\begin{equation}\label{eq:stable-exterior-positive}
    V_s^{\tilde a}(y)\ge \frac12\log M,\qquad |y|\ge M.
\end{equation}
On every fixed compact subset of $\R\setminus\{0\}$, $V_s^{\tilde a}\to\log|y|$ in $L^\infty_{\mathrm{ess}}$, and
\[
        \Theta_{\tilde a}(r)\asymp(\log(e/r))^{-A}
        \qquad (r\downarrow0),
\]
the constants in this $\asymp$ depending also on the exterior datum $m_{\tilde a}$ (through windows positioned outside the well).  These estimates are exactly the replacement inputs used in the stability proposition below.
\end{lemma}

\begin{proof}
On the chart write $m=\log|y|$; then $a_0(R_sy)=(\ell_s-m)^{-A}$ and
\[
 |sR_s^2(\tilde a-a_0)(R_sy)|\le C sR_s^2(\ell_s-m)^{-A-1-\sigma}
       =C\ell_s^{A+1}(\ell_s-m)^{-A-1-\sigma}.
\]
If $0<|y|\le1$, then $m\le0$ and the last display is $O(1)$, while the model potential $V_{s,\mathrm{loc}}^{a_0}(y)$ satisfies $|V_{s,\mathrm{loc}}^{a_0}(y)|\le|\log|y||$ by \eqref{eq:pointwise-comparison}; this gives \eqref{eq:stable-inner-bound}.  If $e^{-\ell_s/2}\le |y|\le e^{\ell_s/2}$, then $|m|\le\ell_s/2$, so the perturbation is $O(\ell_s^{-\sigma})$, and \eqref{eq:stable-local-expansion} follows from \eqref{eq:local-coeff-expansion}.  In particular $V_s^{\tilde a}\to\log|y|$ in $L^\infty_{\mathrm{ess}}$ on fixed compact subsets of $\R\setminus\{0\}$.

Shrink the chart so that $|w|\le1/2$ a.e.\ on $0<|x|<x_*$.  Then $\frac12a_0\le\tilde a\le\frac32a_0$ a.e.\ near the well, and the stated $\Theta_{\tilde a}$ asymptotic follows by the same two-case window argument used for the exact cusp at the start of Section~\ref{sec:log-resonance}: for $0<r\le x_*/2$ and any window $(x-r,x+r)$, at least half of the window lies in $\{|t|\ge r/2\}$, where the original-scale dummy variable $t$ satisfies $\tilde a(t)\ge\min\bigl(\tfrac12(\log(2e/r))^{-A},\ \operatorname*{ess\,inf}_{\{x_*\le|x|\}}\tilde a\bigr)\ge c\,(\log(e/r))^{-A}$ a.e.\ --- the first alternative where the window meets the shrunken chart, the second where it does not, the exterior essential lower bound being positive by hypothesis --- so $\Theta_{\tilde a}(r)\ge\tfrac c2(\log(e/r))^{-A}$; while at the centered window $(2r)^{-1}\int_{-r}^r\tilde a\le\tfrac32\,(2r)^{-1}\int_{-r}^ra_0\le C(\log(e/r))^{-A}$.  In the part of the chart where $\ell_s-m\le \kappa\ell_s$, choose $\kappa>0$ so small that $\frac12\kappa^{-A}-1>0$; then
\[
V_s^{\tilde a}\ge sR_s^2\bigl(\tfrac12(\kappa\ell_s)^{-A}-\ell_s^{-A}\bigr)\to+\infty.
\]
On the remaining chart region, $V_s^{\tilde a}\ge V_{s,\mathrm{loc}}^{a_0}-o(1)\ge\log|y|-1$ for $|y|\ge2$ and large $s$.  Off the shrunken chart, if $m=m_{\tilde a}(x_*)>0$, then the negative shift is dominated exactly as in \eqref{eq:offchart-coercive}: $V_s^{\tilde a}\ge m sR_s^2-\ell_s/A\ge (m/2)sR_s^2$ for all large $s$, hence the essential infimum tends to $+\infty$ on each fixed exterior region.  Combining these alternatives yields the exterior coercivity $V_s^{\tilde a}\ge\log M-C$ in the essential-infimum sense on $|y|\ge M$ for fixed $M$ and large $s$.  The constant $C$ here is produced by the chart alternatives alone (in fact $C=1$ suffices for the middle region once $s$ is large), and thus depends only on $A,x_1,C_W,\sigma$; the exterior essential lower bound of $\tilde a$ enters only through the threshold $s(M)$, since off the chart and in the deep zone the potential exceeds $\log M$ outright for $s$ large.  This is what makes $M_0$ below independent of the exterior data.  In the exterior-tail step of Proposition~\ref{prop:log-scaling} one chooses $M$ only after all local compactness estimates have been obtained; hence we may take $M\ge M_0:=\max(2,e^{2C})$, and then $V_s^{\tilde a}\ge\frac12\log M$ on the exterior.  The same computation gives the tail bound with the denominator $\log M-C\ge\frac12\log M$, hence still $O((\log M)^{-1})$.  These estimates are precisely the replacement inputs used in the transfer proposition below: \eqref{eq:stable-inner-bound} gives the near-zero integrability of the multiplication term, used to pass the potential through the exponentially small core --- the no-concentration input itself comes from the real-part/IMS local $H^1$ bounds, which do not involve the potential; \eqref{eq:stable-local-expansion} gives local distributional convergence away from the exponentially small core and, in mode-weighted $L^2$ form, the compact-window quasimode residuals; and \eqref{eq:stable-exterior-positive} gives the exterior-tail bound.  The spectral consequences themselves are proved in Proposition~\ref{prop:stable-cusp-transfer}.
\end{proof}

\begin{proposition}[Stable-cusp transfer of the logarithmic spectral analysis]\label{prop:stable-cusp-transfer}
Under the hypotheses of Lemma~\ref{lem:stable-cusp-estimates}, Propositions~\ref{prop:log-scaling} and~\ref{prop:log-Ts} remain valid with $\tilde a$ in place of the exact cusp $a$, possibly after replacing the cusp chart by the shrunken radius $x_*$ of Lemma~\ref{lem:stable-cusp-estimates}, while keeping the same scale $R_s,\ell_s$ and the same limiting operator $\cL_1$.  Explicitly, for every fixed $K$: (i) every compact $\mathcal K\Subset\rho(\cL_1)$ lies in $\rho(\cT^{\tilde a}_s)$ for $s$ large, with $\sup_{\zeta\in\mathcal K}\|(\cT^{\tilde a}_s-\zeta)^{-1}\|$ uniformly bounded; (ii) each contour $\Gamma_n$, $0\le n\le K$, eventually encloses exactly one eigenvalue $\tilde\mu_s^{(n)}$, algebraically simple, with $\tilde\mu_s^{(n)}\to\nu_n$, normalized eigenfunctions converging in $L^2(\R)$ to $\hat u_n$ after the usual phase choice, and non-isotropy converging to $\delta(u_n)>0$; (iii) the rescaled conclusions of Proposition~\ref{prop:log-Ts} --- the expansion \eqref{eq:lambda-s-asymp}, the isolation radius $\rho_s$, the tameness constant, and the finite-window completeness \eqref{eq:Ts-completeness} --- hold verbatim for $T_s^{\tilde a}=-\partial_x^2+is\tilde a$.  The local limiting coefficients and model constants are unchanged.  Constants and thresholds in the stable limiting estimates may depend on $(A,x_1,C_W,\sigma)$ and the exterior datum $m_{\tilde a}$; the pencil-transfer, finite-window, and generator/envelope thresholds may also depend on $\|\tilde a\|_{L^\infty(\T)}$.  Consequently, writing $\ell_j:=\ell_{s_j}$, for every fixed closed rectangle $B$ with $\partial B\cap\spec(\cL_1)=\emptyset$, the block generator $\cA_{\tilde a,j}$ has exactly one algebraically simple eigenvalue for each rung $\nu_n\in B$ in the window \eqref{eq:ladder-window}, no other spectrum there, and
\[
    z_j^{\tilde a,(n)}=is_j-\frac12\ell_j^{-A}+\frac{iA\nu_n}{2}\ell_j^{-A-1}+o(\ell_j^{-A-1}) .
\]
The regularized envelope in Theorem~\ref{thm:E} also holds for $\tilde a$.  In particular, for every fixed rung $n$ and every sufficiently large transverse frequency $s_j$, the eigenvalue in the preceding display is denoted consistently by $z_j^{\tilde a,(n)}$ and satisfies the same finite-window completeness statement as in Theorem~\ref{thm:E}\textup{(a)}.
\end{proposition}

\begin{proof}
We indicate explicitly how each ingredient in the exact-cusp proof is replaced, and how the conclusions are recovered.  Fix $K$ and set
\[
    \mathcal K_K:=\bigcup_{0\le n\le K}\overline D(\nu_n,g_K/3).
\]
The proof is organized as four transfer subclaims: first the singular-sequence compactness and resolvent exclusion for $\cT_s^{\tilde a}$; second the nonzero Riesz projections, rank-one/simplicity exclusions, eigenfunction convergence, and non-isotropy; third the rescaling to $T_s^{\tilde a}$ with isolation, tameness, and finite-window completeness; and fourth the pencil transfer, rectangle count, axis cleanness, and envelope.
The estimates in Subclaim~1 are a priori compactness and tightness estimates
for $\zeta$ in an arbitrary fixed compact subset of $\C$; only the
resolvent-exclusion conclusion requires compact subsets of $\rho(\cL_1)$.
With this distinction in force, Subclaim~1 verifies the singular-sequence
hypotheses of Lemma~\ref{lem:singular-sequence-criterion} for
$\cT^{\tilde a}_s$, giving Proposition~\ref{prop:log-scaling}(i).

Subclaim~2 combines the rank-$\ge1$ construction --- with the two-stage
choice, $R=R(\varepsilon,K)$ before $s(\varepsilon,K)$, recorded below ---
with the rank-$\le1$ and simplicity exclusions to give
Proposition~\ref{prop:log-scaling}(ii)--(iii).  Every \emph{eigenfunction}
arising in the alternatives --- including the leading eigenfunction of the
Jordan-chain alternative --- has its eigenvalue inside the fixed compact
$\mathcal K_K$, hence satisfies the Subclaim~1 estimates with $f=0$ and
$\zeta\in\mathcal K_K$: global $H^1$ bounds, no concentration, the exterior
tail, and local $L^\infty$ bounds.  The sequential form of the tight-pairing
lemma therefore applies to every pairing used.  In the Jordan-chain
alternative only the eigenfunction is passed to a limit; the generalized
vector, which solves the inhomogeneous equation
$(\cT^{\tilde a}_s-\mu)w=u$ rather than the $f=0$ equation, enters solely
through complex symmetry to give $\mathfrak b(u,u)=0$.  The exclusions are
the same finite-dimensional alternatives, now with the estimates in
Subclaim~1 explicitly supplying the required tightness.  They use only
complex symmetry on $\T_s$ (Lemma~\ref{lem:complex-symmetry}), tight pairings
(Lemma~\ref{lem:tight-pairings}) with the tail estimate from Subclaim~1, the
root-space alternative (Lemma~\ref{lem:root-space-alt}), and the ladder
overlap \eqref{eq:ladder-overlap}; none of these inputs involves the damping
beyond the Subclaim~1 estimates.

Subclaim~3 is the algebraic rescaling to $T_s^{\tilde a}$ from
Proposition~\ref{prop:log-Ts}, and yields its eigenvalue expansion, isolation
radii, and tameness unchanged.  Subclaim~4 combines the pencil bound and
rectangle Rouch\'e count with boundedness of $\tilde a$, the boundary
resolvent from Subclaim~1, and thresholds allowed to depend on
$\|\tilde a\|_{L^\infty}$.

\smallskip
\noindent\emph{Subclaim 1: singular-sequence compactness and resolvent
exclusion.}  First, the real-part identities giving global $H^1$ bounds and
local IMS bounds are unchanged, because the potential is purely imaginary;
in particular the local $H^1$ (hence local $L^\infty$) bounds give the same
no-concentration estimate at the logarithmic singularity.  The near-zero
estimate \eqref{eq:stable-inner-bound} is used to pass the potential term
through the exponentially small core, and the local expansion
\eqref{eq:stable-local-expansion} gives, on every compact subset of
$\R\setminus\{0\}$, convergence of the multiplication term to $\log|y|$ in
$L^\infty_{\rm ess}$.  For test functions $\varphi$ whose support meets $0$,
split $\int(V_s^{\tilde a}-\log|y|)u\bar\varphi$ at
$|y|=e^{-\ell_s/2}$ exactly as in the proof of
Proposition~\ref{prop:log-scaling}: on the inner region,
\eqref{eq:stable-inner-bound} and the local $L^\infty$ bounds control the
integrand by $C(1+|\log|y||)$, whose integral over
$\{|y|\le e^{-\ell_s/2}\}$ vanishes, while on the outer part
\eqref{eq:stable-local-expansion} bounds the contribution, after
Cauchy--Schwarz, by
$C\ell_s^{-1}\,\|u\|_2\,\bigl\|(\log|y|)^2\varphi\bigr\|_2
+C\ell_s^{-\sigma}\|u\|_2\|\varphi\|_2\to0$, the fourth power of the
logarithm being locally integrable.  Hence every normalized singular sequence
has the same local distributional limit.  The exterior step is also
unchanged after the only necessary modification: choose $M\ge M_0$ in
\eqref{eq:stable-exterior-positive}, so the denominator in the tail estimate
is at least $\frac12\log M$ and the tail remains
$O((\log M)^{-1})$.  Thus the singular-sequence criterion, uniform contour
resolvent exclusion, and compact-window resolvent bounds of
Proposition~\ref{prop:log-scaling}(i) hold for $\tilde a$.

\emph{Subclaim 2: nonzero projections, rank one, and non-isotropy.}  Second, the rank-$\ge1$ construction uses fixed compactly supported cutoffs of the model eigenfunctions.  Let $w_R=\chi_Ru_n$ (as before, $w_R\in H^2$ because $u_n''=(i\log|y|-\nu_n)u_n\in L^2_{\rm loc}$, $\log|y|\in L^2_{\rm loc}$, and the cutoff is smooth and compactly supported; for large $s$ the support lies in the fundamental interval, so the periodic endpoint traces of the zero extension vanish), where $\chi_R(y):=\chi(|y|/R)$ for a fixed $\chi\in C_c^\infty([0,\infty))$ with $\chi=1$ on $[0,1]$ and $\operatorname{supp}\chi\subset[0,2]$ --- so that $\chi_R=1$ on $\{|y|\le R\}$, $\operatorname{supp}\chi_R\subset\{|y|\le2R\}$, $\|\chi_R'\|_\infty=O(R^{-1})$, $\|\chi_R''\|_\infty=O(R^{-2})$, and in particular $\|\chi_R'\|_\infty+\|\chi_R''\|_\infty\le C$ uniformly in $R\ge1$.  The commutator residual $\varepsilon(R):=\|\chi_R''u_n\|_2+2\|\chi_R'u_n'\|_2$ obeys $\varepsilon(R)\le Ce^{-cR(\log R)^{1/2}}$ by Lemma~\ref{lem:model-rung-decay}; it remains to estimate the multiplication error.  Split the support into $|y|<e^{-\ell_s/2}$ and its complement.  On the inner piece, \eqref{eq:stable-inner-bound}, the corresponding cusp bound, and local boundedness of $u_n$ give
\[
    \int_{|y|<e^{-\ell_s/2}}
       \bigl(|V_s^{\tilde a}|+|V_{s,\mathrm{loc}}^{a_0}|+|\log|y||+1\bigr)^2 |w_R|^2\,dy=o(1),
\]
for each fixed $R$, since $\int_0^{e^{-\ell_s/2}}(1+|\log y|)^2dy\to0$.  On the complementary part $e^{-\ell_s/2}\le |y|\le2R$, the pointwise error is not uniformly small near the moving inner edge; we use \eqref{eq:stable-local-expansion} only in the mode-weighted $L^2$ form
\[
    \begin{aligned}
    \bigl\|(V_s^{\tilde a}-\log|y|)w_R\bigr\|_{L^2(e^{-\ell_s/2}\le |y|\le2R)}^2
    &\le C\ell_s^{-2}\int_{\R}(1+|\log|y||)^4 |w_R(y)|^2\,dy \\
    &\quad +C\ell_s^{-2\sigma}\|w_R\|_2^2=o_s(1).
    \end{aligned}
\]
The fourth logarithmic moment is finite because $u_n$ is bounded near $0$ and superexponentially small at infinity.  Hence $\|(\cT_s^{\tilde a}-\nu_n)w_R\|\le o_s(1)+\varepsilon(R)$, with the limits taken in this order: given $\varepsilon>0$, first choose $R=R(\varepsilon,K)$ so large that $\varepsilon(R)\le\varepsilon/2$ and $\|w_R\|\ge\tfrac12\|u_n\|$ simultaneously for the finitely many fixed model eigenfunctions $u_n$, $0\le n\le K$; then choose $s(\varepsilon,K)$ so that the displayed multiplication errors --- whose constants are independent of $R$, the weighted moments being bounded by those of $u_n$ itself --- are at most $\varepsilon/2$ for all $s\ge s(\varepsilon,K)$ and all $n\le K$.  With $C_\Gamma$ the uniform bound on $\sup_{\zeta\in\Gamma_n}\|(\cT^{\tilde a}_s-\zeta)^{-1}\|$ for $s$ large --- supplied for $\tilde a$ by Subclaim~1 --- the projection identity in Lemma~\ref{lem:singular-sequence-criterion} gives $\|\Pi_{\Gamma_n}(\cT^{\tilde a}_s)w_R-w_R\|\le C\,C_\Gamma\bigl(o_s(1)+\varepsilon(R)\bigr)\le C\,C_\Gamma\,\varepsilon$; taking $\varepsilon<\min_{0\le n\le K}\|u_n\|/(4C\,C_\Gamma)$ in the two-stage choice --- one threshold for all rungs, the subsequent $s(\varepsilon,K)$ chosen after this finite minimum --- makes this smaller than $\tfrac12\|w_R\|$, so the Riesz projection is nonzero at every rung for all large $s$.  No norm-smallness at a fixed $R$-independent rate in $s$ alone is claimed.  The rank-$\le1$ alternatives are excluded exactly as before, with the following convention: the compactness estimates are applied only to the eigenfunction families appearing in the two-eigenvalue and geometric-multiplicity alternatives, and to the eigenfunction $u_s$ in the Jordan-chain alternative.  In that Jordan case the generalized vector is not passed to a compactness limit; it is used only through complex symmetry to infer $\mathfrak b(u_s,u_s)=0$.  Lemma~\ref{lem:tight-pairings} then applies to the relevant eigenfunctions, and the limiting mode has nonzero overlap \eqref{eq:ladder-overlap}.  Therefore the Riesz ranges are rank one, the eigenvalues converge rung by rung, and non-isotropy persists.

\emph{Subclaim 3: rescaling, isolation, tameness, and finite-window completeness.}  Third, the algebraic rescaling is explicit.  With the same unitary $U_s$ as in \eqref{eq:scaled-Ts},
\[
    \cT_s^{\tilde a}=R_s^2U_sT_s^{\tilde a}U_s^{-1}-\frac{i\ell_s}{A},
    \qquad
    T_s^{\tilde a}=R_s^{-2}U_s^{-1}\Bigl(\cT_s^{\tilde a}+\frac{i\ell_s}{A}\Bigr)U_s.
\]
Thus the affine map
\[
    \lambda=R_s^{-2}\Bigl(\zeta+\frac{i\ell_s}{A}\Bigr)
       =is\ell_s^{-A}+As\ell_s^{-A-1}\zeta,
    \qquad R_s^{-2}=As\ell_s^{-A-1},
\]
carries the spectrum of $\cT_s^{\tilde a}$ onto the spectrum of $T_s^{\tilde a}$, preserving algebraic multiplicity, and
\[
    \|(T_s^{\tilde a}-\lambda)^{-1}\|
       =R_s^2\|(\cT_s^{\tilde a}-\zeta)^{-1}\|.
\]
Consequently
\[
    \tilde\lambda_s^{(n)}:=R_s^{-2}\Bigl(\tilde\mu_s^{(n)}+\frac{i\ell_s}{A}\Bigr)
    =is\ell_s^{-A}+As\ell_s^{-A-1}\tilde\mu_s^{(n)}
\]
is algebraically simple and has the same expansion as in \eqref{eq:lambda-s-asymp}.  The circle
\[
    |\lambda-\tilde\lambda_s^{(n)}|=\rho_s^{(K)},
    \qquad \rho_s^{(K)}:=\frac{Ag_K}{6}s\ell_s^{-A-1}=R_s^{-2}\frac{g_K}{6},
\]
corresponds exactly to $|\zeta-\tilde\mu_s^{(n)}|=g_K/6$.  For large $s$ this circle lies in a fixed compact annulus disjoint from $\spec(\cL_1)$, so Subclaim~1 gives
\[
    \|(T_s^{\tilde a}-\lambda)^{-1}\|
       \le C R_s^2
       =\frac{C g_K/6}{\rho_s^{(K)}}.
\]
This is the required tameness bound, and the same affine unitary map preserves non-isotropy of the normalized eigenfunctions.  Finally fix a closed rectangle $B$ with $\partial B\cap\spec(\cL_1)=\emptyset$.  Put $\delta_B:=\operatorname{dist}(\partial B,\spec\cL_1)>0$ and let
\[
    B^+:=\{\zeta\in\C:\operatorname{dist}(\zeta,B)\le\delta_B/2\}.
\]
Since $\spec(\cL_1)$ is closed and discrete and $B^+$ is bounded, only finitely many model rungs lie in $B^+$.  Choose $K$ so that every rung in $B^+$ is among $\nu_0,\dots,\nu_K$ and the finite contours under discussion include those rungs.  Set
\[
    \mathcal K_B:=B^+\setminus\bigcup_{\nu_n\in B^+}D(\nu_n,g_K/3).
\]
This compact set is disjoint from $\spec(\cL_1)$.  By Subclaim~1 it is eventually contained in $\rho(\cT_s^{\tilde a})$; by Subclaim~2 each disc $D(\nu_n,g_K/3)$ with $\nu_n\in B^+$ contains exactly one algebraically simple eigenvalue $\tilde\mu_s^{(n)}\to\nu_n$, and no other eigenvalue.  Any $\cT_s^{\tilde a}$ eigenvalue lying in $B$ lies automatically in $B^+$; hence it is either in the resolvent set $\mathcal K_B$ just excluded or in one of these finitely many discs.  If the corresponding model rung lies in $B^+\setminus B$, convergence keeps the perturbed eigenvalue outside $B$ for large $s$, while the rungs in $B$ are precisely the listed ones.  Hence
\[
    \spec(T_s^{\tilde a})\cap
    \{is\ell_s^{-A}+As\ell_s^{-A-1}\beta:\beta\in B\}
    =\{\tilde\lambda_s^{(n)}:\nu_n\in B\}
\]
with algebraic multiplicity one for each listed eigenvalue.

\emph{Subclaim 4: pencil transfer, rectangle count, axis cleanness, and envelope.}  Finally, the quadratic pencil perturbation has the same size on every fixed contour: if $\lambda$ lies in the low $T_s$ window, then $w(\lambda)=i\lambda/(2s)=O(\ell_s^{-A})$, and
\[
    \|w(\lambda)\tilde a+w(\lambda)^2\|_{L^2\to L^2}
    \le C(\|\tilde a\|_{L^\infty}+1)\ell_s^{-A}.
\]
Thus the smallness threshold in the pencil step is allowed to depend on $\|\tilde a\|_{L^\infty}$.  For each fixed rung, Proposition~\ref{prop:pencil-transfer} with $d=\tilde a$ and the isolation and tameness from Subclaim~3 produces the asserted algebraically simple block eigenvalue and the displayed asymptotic; the non-isotropy from Subclaim~2 is recorded for the accompanying condition-number and phase-rigidity statements, not as a direct hypothesis of Proposition~\ref{prop:pencil-transfer}.

For the finite-window count, let
\[
    \mathfrak B_s(B)=\{is\ell_s^{-A}+As\ell_s^{-A-1}\beta:\beta\in B\}.
\]
For $\lambda=is\ell_s^{-A}+As\ell_s^{-A-1}\beta$ with $\beta\in\partial B$ --- a compact subset of $\rho(\cL_1)$ --- Subclaim~1 and the scaling identity above give
\[
    \|(T_s^{\tilde a}-\lambda)^{-1}\|
       =R_s^2\|(\cT_s^{\tilde a}-\beta)^{-1}\|
       \le C_B R_s^2.
\]
Together with the previous perturbation estimate this yields
\[
\begin{aligned}
    \sup_{\lambda\in\partial\mathfrak B_s(B)}
    &\|(T_s^{\tilde a}-\lambda)^{-1}\|
      \frac{|\lambda|}{2s}
      \left(\|\tilde a\|_{L^\infty}+\frac{|\lambda|}{2s}\right) \\
    &\le C_B(\|\tilde a\|_{L^\infty}+1)R_s^2\ell_s^{-A}
      =C_B(\|\tilde a\|_{L^\infty}+1)\frac{\ell_s}{As}\to0.
\end{aligned}
\]
Proposition~\ref{prop:cluster-transfer}, applied to
$\operatorname{int}\mathfrak B_s(B)$, therefore identifies the total
pencil-root count with the $T_s^{\tilde a}$ eigenvalue count from
Subclaim~3, namely $\#\{n:\nu_n\in B\}$.  The simple roots already
constructed for those rungs exhaust this count, so no additional
block-generator eigenvalues enter the window.

For the regularized envelope, the same low-frequency hypotheses required by Lemma~\ref{lem:envelope} are verified here.  Since $\int_\T\tilde a>0$, Lemma~\ref{lem:axis-cleanness} gives $i\R\subset\rho(\cA_{\tilde a})$, in particular $0\in\rho(\cA_{\tilde a})$.  The upper envelope follows from $\Theta_{\tilde a}(r)\asymp(\log(e/r))^{-A}$ and Corollary~\ref{cor:log}, and the lower envelope comes from the genuine block eigenvalues just constructed.
\end{proof}

\begin{proof}[Proof of Theorem~\ref{thm:E}\textup{(c)}]
Replace $x_1$ by $\min(x_1,x_0)$ if necessary.  The hypotheses are then
exactly those of Lemma~\ref{lem:stable-cusp-estimates}.  That lemma gives
the same local logarithmic limit, coercive tails, compact-window
quasimodes, and mass asymptotics as for the exact cusp.
Proposition~\ref{prop:stable-cusp-transfer} consequently gives the full
finite-window ladder, its algebraic count, and the two leading
coefficients with the unchanged scales $R_s,\ell_s$.  Its final
axis-cleanness and envelope argument gives $i\R\subset\rho(\cA_{\tilde
a})$ and the sharp regularized decay.  These are precisely the conclusions
of part~\textup{(c)}.
\end{proof}

\begin{remark}[Fixed measurable perturbations]\label{rem:rough-members}
Theorem~\ref{thm:E}\textup{(c)} allows arbitrary measurable modulation
patterns, provided their relative amplitude is lower order at the well.
Fix $\sigma>0$ and $x_1\in(0,1)$.  For every measurable $F\subset\T$ and
$0<\eps_0\le1$, define, for $0<|x|<x_1$,
\[
 \tilde a(x)=(\log(e/|x|))^{-A}
 \Bigl(1+\eps_0(\log(e/|x|))^{-1-\sigma}
       (2\one_F(x)-1)\Bigr)
\]
and set $\tilde a(0)=0$.  Any bounded nonnegative extension to $\T$
satisfying
\[
 \operatorname*{ess\,inf}_{\{|x|\ge x_1\}}\tilde a>0
\]
is admissible in Theorem~\ref{thm:E}\textup{(c)}.  Indeed, on each punctured
annulus $\rho\le |x|<x_1$ the displayed factor is bounded below by a
positive constant, while the extension supplies the required positivity
outside the chart.  Thus neither continuity nor geometric regularity of the
perturbation is required.
\end{remark}

\section{Isometric holonomy and the boundary of flat separation}
\label{sec:mapping-torus}

The global product topology can be relaxed without changing the
one-dimensional mechanism.  Let $\varphi\in\operatorname{Isom}(Y)$ and
form the locally product mapping torus
\[
 X_\varphi=([0,2\pi]\times Y)/\bigl((2\pi,y)\sim(0,\varphi(y))\bigr),
\]
with the metric induced by $dx^2+g_Y$ and damping pulled back from the
base circle.  This class includes genuinely nonproduct examples, while
retaining flat one-dimensional separation with a holonomy phase.
More generally, for $L>0$, $k\in\{1,2\}$, and
$\theta\in\R/2\pi\Z$, set
\[
 H^k_\theta(0,L):=\{v\in H^k(0,L):
 v^{(j)}(L)=e^{i\theta}v^{(j)}(0),\ 0\le j<k\}.
\]
We abbreviate $H^k_\theta=H^k_\theta(0,2\pi)$ and use
$H^k_\theta(I)$ for the translated endpoint conditions on any bounded
interval $I$.
We use the pullback convention
\[
 U_\varphi=\varphi^*,\qquad U_\varphi f:=f\circ\varphi.
\]
Since $\varphi$ is an isometry, $U_\varphi$ is unitary and commutes with
$-\Delta_Y$.

\begin{thmx}[Isometric mapping-torus extension]\label{thm:F}
Let $Y$ be a compact connected Riemannian manifold without boundary with
$\dim Y\ge1$, let
$\varphi\in\operatorname{Isom}(Y)$, and let $X_\varphi$ carry the locally
product metric defined above.  The damping is pulled back from a function
on the base circle.
\begin{enumerate}[label=\textup{(\alph*)},leftmargin=*]
\item Let $0\le a\in L^\infty(\T)$.  Theorems~\ref{thm:A}--\ref{thm:C},
including Theorem~\ref{thm:B}\textup{(iii)$\Rightarrow$(i)} and
Corollary~\ref{cor:log},
hold on $X_\varphi$ with the same lower interval mass.  More precisely,
condition \textup{(iii)} of Theorem~\ref{thm:B} is replaced by the
following statement: there exist $C,s_0>0$, a dyadically syndetic set
$\cS\subset\Sigma_Y\cap[s_0,\infty)$, and, for every $s\in\cS$, a choice
\[
 e^{i\theta_s}\in
 \spec\!\left(U_\varphi|_{E_{s^2}}\right),
 \qquad E_{s^2}=\ker(-\Delta_Y-s^2),
\]
such that $Q^{\theta_s}_{s,0}$ is invertible and
\[
 Q^{\theta_s}_{s,0}=-\partial_x^2+isa(x)
 \quad\text{on}\quad H^2_{\theta_s},
 \qquad
 \bigl\|(Q^{\theta_s}_{s,0})^{-1}\bigr\|
 \le Cs^{-1}L(s).
\]
The direct estimates are uniform in the phase, while one such choice per
dyadic octave suffices for the inverse implication.

\item Suppose first that $a$ satisfies the logarithmic-cusp hypotheses of
Theorem~\ref{thm:E}.  In every sufficiently high joint
transverse/Floquet block, the finite-window statement and rung-by-rung
asymptotics of Theorem~\ref{thm:E}\textup{(a)} hold with the same scales,
limiting ladder, algebraic count, and two leading coefficients.  Every
listed eigenvalue is algebraically simple within that scalar joint block.
The principal rungs and the global resolvent estimate give on $X_\varphi$
the two-sided regularized semigroup envelope of
Theorem~\ref{thm:E}\textup{(b)}.

If $\tilde a$ satisfies the stable measurable-cusp hypotheses of
Theorem~\ref{thm:E}\textup{(c)}, all preceding conclusions hold with
$\tilde a$ in place of $a$, again with the same scales, model ladder, and
two leading coefficients.  No finite-order assumption on $\varphi$ is
required.  Algebraic multiplicities in the full generator are obtained
by summing the scalar joint-block multiplicities over all transverse
eigenvectors and holonomy phases producing the same spectral point.
\end{enumerate}
In particular, both main results apply to every isometric mapping torus,
including the flat Klein bottle obtained from $Y=\T$ and the reflection
$\varphi(y)=-y$.
\end{thmx}

The mapping-torus theorem removes global product topology without changing
the scalar mechanism.  A smooth variation of the fiber metric is
different: it can change the stabilization class while leaving both the
damping and its lower interval mass fixed.

\begin{proposition}[Failure of $\Theta_a$-classification under smooth warping]
\label{prop:warped-sharpness}
On $\T_x\times\T_y$, let
\[
 g_\eps=dx^2+(2+\eps\sin x)^2dy^2,
 \qquad |\eps|<1,
\]
and let $a\in C(\T)$ satisfy $a\ge0$ and $a^{-1}(0)=\{0\}$.  Denote by
$\cA_{\eps,a}$ the damped-wave generator for $(\T^2,g_\eps)$ with
damping $a(x)$.  Then $e^{t\cA_{0,a}}$ is not exponentially stable,
whereas $e^{t\cA_{\eps,a}}$ is exponentially stable for every
$0<|\eps|<1$.  The base-circle quantity $\Theta_a$ is independent of
$\eps$.  Consequently, classification by $\Theta_a$ alone cannot hold
uniformly over the class of smooth warped metrics, even for metrics
arbitrarily close in $C^\infty$ to the flat product metric $g_0$.
\end{proposition}

\begin{corollary}[The logarithmic cusp under warping]
\label{cor:warped-log-cusp}
Fix $A>0$ and $0<x_0<1$, and let $a_A\in C(\T)$ be nonnegative,
positive on $\T\setminus\{0\}$, even on $(-x_0,x_0)$, and such that
\[
 a_A(0)=0,
 \qquad
 a_A(x)=(\log(e/|x|))^{-A},\qquad 0<|x|<x_0.
\]
Then
\begin{equation}\label{eq:warped-log-mass}
 \Theta_{a_A}(r)\sim(\log(e/r))^{-A},
 \qquad r\downarrow0,
\end{equation}
for every metric $g_\eps$ above.  For $\eps=0$, the finite-window ladder
and two-term eigenvalue asymptotics of
Theorem~\ref{thm:E}\textup{(a)--(b)} hold, and
\[
 c e^{-Ct^{1/(A+1)}}
 \le \|e^{t\cA_{0,a_A}}\cA_{0,a_A}^{-1}\|
 \le C e^{-ct^{1/(A+1)}},
 \qquad t\ge1.
\]
For every $0<|\eps|<1$, $e^{t\cA_{\eps,a_A}}$ is exponentially stable.
Thus an arbitrarily small smooth warping changes the sharp logarithmic-cusp
decay from stretched exponential to exponential without changing either
the damping or its lower interval mass.
\end{corollary}

The mapping-torus extension identifies exactly which part of the product
geometry is global bookkeeping and which part drives the estimates.  The
holonomy changes the scalar boundary condition, but it does not change the
local one-dimensional equation.  We first make that reduction explicit.

Because $U_\varphi$ is unitary and commutes with $-\Delta_Y$, every
finite-dimensional eigenspace $E_\lambda$ has an orthonormal basis of joint
eigenvectors
\begin{equation}\label{eq:mapping-joint-basis}
 -\Delta_Y\psi_{\lambda,q}=\lambda\psi_{\lambda,q},
 \qquad
 U_\varphi\psi_{\lambda,q}=e^{i\theta_{\lambda,q}}
      \psi_{\lambda,q}.
\end{equation}
A function on $X_\varphi$ is a function on $[0,2\pi]\times Y$ satisfying
$u(2\pi,y)=u(0,\varphi(y))$, with the analogous trace identity for
$\partial_xu$.  In the Floquet notation introduced before
Theorem~\ref{thm:F}, the coefficient of $\psi_{\lambda,q}$ belongs to
$H^2_{\theta_{\lambda,q}}$.
The stationary operator is the orthogonal sum of the exact Floquet
blocks
\begin{equation}\label{eq:mapping-floquet-block}
 Q^\theta_{s,\mu}=-\partial_x^2-\mu+isa(x),
 \qquad \mathcal D(Q^\theta_{s,\mu})=H^2_\theta,
 \qquad \mu=s^2-\lambda.
\end{equation}
The generator blocks and their pencils have the same boundary condition.
Elliptic regularity on the compact manifold $X_\varphi$ assembles the
block inverses exactly as in Lemma~\ref{lem:blocks}; in particular the
stationary norm is the supremum of the Floquet-block norms and algebraic
multiplicities add over the joint blocks.  The Lax--Milgram,
elliptic-regularity, and compact-embedding proof of
Lemma~\ref{lem:blocks}\textup{(i)} applies verbatim on the compact
manifold $X_\varphi$; hence its damped-wave generator has compact
resolvent.

The logarithmic spectral argument requires a bilinear symmetry.  For a
general Floquet phase it pairs the opposite sectors rather than acting
within one sector.

\begin{lemma}[Paired Floquet symmetry and multiplicity]
\label{lem:paired-floquet}
Let $V\in L^\infty(0,L)$ be real and let
\[
 A_V^\theta=-\partial_x^2+iV,
 \qquad \mathcal D(A_V^\theta)=H^2_\theta(0,L).
\]
For one vector in the $\theta$ sector and one in the $-\theta$ sector,
in either order, set $\mathfrak b_\theta(u,v)=\int_0^Luv$.  Then:
\begin{enumerate}[label=\textup{(\roman*)},leftmargin=*]
\item The two sectors are bilinear transposes:
\begin{equation}\label{eq:paired-floquet-green}
 \mathfrak b_\theta(A_V^\theta u,v)=\mathfrak b_\theta(u,A_V^{-\theta}v),
 \qquad u\in H^2_\theta,\quad v\in H^2_{-\theta}.
\end{equation}
Moreover
\begin{equation}\label{eq:paired-floquet-adjoint}
 A_V^{-\theta}=C(A_V^\theta)^*C,
\end{equation}
where $C$ is pointwise conjugation.

\item The operators $A_V^\theta$ and $A_V^{-\theta}$ have the same
eigenvalues, with the same geometric and algebraic multiplicities.  If
$M_V(\zeta)$ is the monodromy matrix for $-u''+iVu=\zeta u$, normalized
at $0$, then $\det M_V(\zeta)=1$ and
\begin{equation}\label{eq:floquet-determinants}
 \begin{aligned}
 D_\theta(\zeta)
   &:=\det(M_V(\zeta)-e^{i\theta}I)
     =e^{i\theta}\bigl(2\cos\theta-\operatorname{tr}M_V(\zeta)\bigr),\\
 D_{-\theta}(\zeta)&=e^{-2i\theta}D_\theta(\zeta).
 \end{aligned}
\end{equation}
The order of a zero of $D_\theta$ is the algebraic multiplicity of the
corresponding eigenvalue of $A_V^\theta$.

\item Let $\lambda$ be algebraically simple, and choose
$0\ne u_\theta\in\ker(A_V^\theta-\lambda)$ and
$0\ne u_{-\theta}\in\ker(A_V^{-\theta}-\lambda)$.  Then
\begin{equation}\label{eq:cross-overlap-nonzero}
 \mathfrak b_\theta(u_\theta,u_{-\theta})\ne0.
\end{equation}
If $\Pi_\theta(\lambda)$ and $\Pi_{-\theta}(\lambda)$ are the two Riesz
projections, then
\begin{equation}\label{eq:cross-projection-formula}
 \begin{aligned}
 \Pi_\theta(\lambda)f
   &=\frac{\mathfrak b_\theta(f,u_{-\theta})}
           {\mathfrak b_\theta(u_\theta,u_{-\theta})}\,u_\theta,\\
 \Pi_{-\theta}(\lambda)g
   &=\frac{\mathfrak b_\theta(g,u_\theta)}
           {\mathfrak b_\theta(u_{-\theta},u_\theta)}\,u_{-\theta}.
 \end{aligned}
\end{equation}
Consequently, with
\begin{equation}\label{eq:cross-nonisotropy}
 \delta_\times(u_\theta,u_{-\theta})
 :=\frac{|\mathfrak b_\theta(u_\theta,u_{-\theta})|}
         {\|u_\theta\|_2\|u_{-\theta}\|_2},
\end{equation}
one has
\begin{equation}\label{eq:cross-projection-norm}
 \|\Pi_\theta(\lambda)\|
 =\|\Pi_{-\theta}(\lambda)\|
 =\delta_\times(u_\theta,u_{-\theta})^{-1}.
\end{equation}

\item If $A_V^\theta u=\lambda u$ and
$A_V^{-\theta}v=\mu v$ with $\lambda\ne\mu$, then
$\mathfrak b_\theta(u,v)=0$.  If $(A_V^\theta-\lambda)w=u$ and
$A_V^{-\theta}v=\lambda v$, then
\begin{equation}\label{eq:cross-jordan-obstruction}
 \mathfrak b_\theta(u,v)=0.
\end{equation}
Thus a nonzero cross-overlap excludes a length-two Jordan chain in either
paired sector.
\end{enumerate}
\end{lemma}

\begin{proof}
For \textup{(i)}, integration by parts twice leaves the boundary form
$[u'v-uv']_0^L$.  At $L$ the two factors acquire the phases
$e^{i\theta}$ and $e^{-i\theta}$, whose product is one, so the boundary
form vanishes.  The Hilbert-space adjoint of $A_V^\theta$ is
$-\partial_x^2-iV$ on $H^2_\theta$; conjugation maps this domain onto
$H^2_{-\theta}$ and gives \eqref{eq:paired-floquet-adjoint}.

For \textup{(ii)}, the first-order system associated with the scalar ODE
has trace zero, hence Liouville's formula gives
$\det M_V(\zeta)=1$.  Formula~\eqref{eq:floquet-determinants} follows from
$\det(M-rI)=1-r\operatorname{tr}M+r^2$ with $r=e^{i\theta}$.
To identify the zero order, consider on the unrestricted space
$H^2(0,L)$ the family
\[
 \mathcal F_\theta(\zeta)u
 :=\bigl(-u''+iVu-\zeta u,\,
 u(L)-e^{i\theta}u(0),\,
 u'(L)-e^{i\theta}u'(0)\bigr).
\]
Solving the inhomogeneous ODE with prescribed Cauchy data analytically
conjugates this family to
$I_{L^2}\oplus(M_V(\zeta)-e^{i\theta}I)$.  Splitting $H^2(0,L)$ into
$H^2_\theta$ and a fixed two-dimensional right inverse of the boundary
trace analytically conjugates it to
$(A_V^\theta-\zeta)\oplus I_{\C^2}$.  Analytic invertible left and right
factors preserve root functions and their orders.  Hence the zero order
of $D_\theta$ equals the operator algebraic multiplicity, and the second
identity in \eqref{eq:floquet-determinants} proves equality of the paired
multiplicities.  Their geometric multiplicities agree as well: if
$\theta\notin\{0,\pi\}$, each eigenspace is one-dimensional, since a
two-dimensional eigenspace would force
$M_V(\zeta)=e^{i\theta}I$ and hence $e^{2i\theta}=1$; for
$\theta\in\{0,\pi\}$ the two Floquet domains coincide.  Equivalently,
the Riesz projections satisfy
\begin{equation}\label{eq:paired-riesz-adjoint}
 \Pi_{-\theta}(\lambda)=C\Pi_\theta(\lambda)^*C.
\end{equation}

For \textup{(iii)}, $\overline{u_{-\theta}}$ spans the adjoint
eigenspace of $(A_V^\theta)^*$ at $\bar\lambda$.  The rank-one Riesz
formula gives \eqref{eq:cross-projection-formula}; its normalizing
denominator cannot vanish.  Taking norms gives
\eqref{eq:cross-projection-norm}.  Finally,
\eqref{eq:paired-floquet-green} gives
$(\lambda-\mu)\mathfrak b_\theta(u,v)=0$ and, in the Jordan case,
\[
 \mathfrak b_\theta(u,v)
 =\mathfrak b_\theta((A_V^\theta-\lambda)w,v)
 =\mathfrak b_\theta(w,(A_V^{-\theta}-\lambda)v)=0.
\]
This proves \textup{(iii)--(iv)}.
\end{proof}

\begin{remark}[Holomorphic pencils]\label{rem:paired-floquet-pencil}
The monodromy argument does not use the special form $iV$ or the reality
of $V$.  If $q(z,\cdot)\in L^\infty(0,L)$ depends holomorphically on $z$
and
\[
 \mathcal P_\theta(z)=-\partial_x^2+q(z,x)
 :H^2_\theta(0,L)\longrightarrow L^2(0,L),
\]
then the characteristic determinants of $\mathcal P_\theta$ and
$\mathcal P_{-\theta}$ again differ by the nonzero factor
$e^{-2i\theta}$.  Their characteristic values therefore have the same
Gohberg--Sigal multiplicities.  This applies in particular to the exact
damped-wave quadratic pencil in every transverse block.
\end{remark}

\begin{proof}[Proof of Theorem~\ref{thm:F}\textup{(a)}]
The case $a\equiv0$ in Theorem~\ref{thm:A} is immediate, and
Theorems~\ref{thm:B}--\ref{thm:C} impose nontriviality whenever it is
needed.  We may therefore assume $\int_\T a>0$ in the axis-cleanness
argument below.
Put $\alpha=\theta/(2\pi)$, chosen in $[-1/2,1/2]$.  The unitary gauge
$v(x)=e^{i\alpha x}w(x)$ identifies $H^2_\theta$ with the periodic domain
and conjugates the free part to
\[
 -\partial_x^2\longmapsto-(\partial_x+i\alpha)^2.
\]
Consequently its Fourier frequencies are $n+\alpha$, $n\in\Z$.  The
propagating window used in Lemma~\ref{lem:observability} contains a
uniformly bounded number of these shifted frequencies, independently of
$\alpha$, and the complementary free spectral gap is uniform as well.
Nazarov's estimate in Lemma~\ref{lem:nazarov-input} is
independent of the locations and separation of the frequencies, so the
entire propagating-cluster argument is uniform in the holonomy phase.

The other local input is equally phase-independent.  The quasi-periodic
extension of $v\in H^1_\theta$ has periodic modulus, and the one-dimensional
diamagnetic inequality gives $|v|\in H^1(\T)$ with
$|(|v|)'|\le |v'|$ a.e.  Lemma~\ref{lem:wp}, applied to $|v|$, therefore
has the same constant in every phase.  Sesquilinear integration by parts
for $v\in H^2_\theta$ has no endpoint term because
the phase of $v$ cancels that of $\bar v$.  Hence the upper engine,
including all nonresonant blocks, holds uniformly for
$Q^\theta_{s,\mu}$.

The Fredholm-surjectivity step is also uniform.  The adjoint of
$Q^\theta_{s,\mu}$ has potential $-isa$ on the same Floquet domain, and
pointwise conjugation gives
\[
 C(Q^\theta_{s,\mu})^*C=Q^{-\theta}_{s,\mu}.
\]
Since the preceding estimate holds for every phase, injectivity in the
opposite phase supplies the required adjoint injectivity.

For the inverse direction, let $v\in H^2(\T)$ be the absorbing function
from Proposition~\ref{prop:absorbing}.  Its support lies in a proper base
arc, so in the fixed representative $[0,2\pi]$ one may choose
$x_*\in(0,2\pi)$ with an open neighborhood on which $v$ vanishes.
For any phase $\theta$, define in this fixed representative
\[
 v_\theta(x):=
 \begin{cases}
   v(x),&0\le x\le x_*,\\
   e^{i\theta}v(x),&x_*\le x\le2\pi.
 \end{cases}
\]
The phase change occurs where $v$ vanishes identically; hence
$v_\theta\in H^2(0,2\pi)$.  Periodicity of the traces of $v$ gives
\[
 v_\theta(2\pi)=e^{i\theta}v_\theta(0),\qquad
 v_\theta'(2\pi)=e^{i\theta}v_\theta'(0),
\]
so $v_\theta\in H^2_\theta$ in the original fixed Floquet realization.
Moreover, the multiplying phase is locally constant wherever $v$ or its
residual is nonzero, and the damping representative has not been moved;
therefore
\[
 \|v_\theta\|_2=\|v\|_2,
 \qquad
 \|Q^\theta_{s,0}v_\theta\|_2
   =\|Q^0_{s,0}v\|_2.
\]
Thus Proposition~\ref{prop:absorbing} supplies
\eqref{eq:absorbing-scalar-lower} in every phase, in particular in the
chosen phase $\theta_s$ occurring in $E_{s^2}$.  At $s=\sqrt\lambda$,
multiplication by the corresponding joint eigenvector in
\eqref{eq:mapping-joint-basis} inserts this scalar $\mu=0$ quasimode into
the required joint Floquet block.  Hence a bound for one such phase at
one frequency per octave gives the same scalar lower estimate, and the
proof of Theorem~\ref{thm:B}\textup{(iii)$\Rightarrow$(i)} is unchanged.

The direct-sum identity following \eqref{eq:mapping-floquet-block}, the
stationary--generator comparison of Lemma~\ref{lem:comparison}, and axis
cleanness complete the operator-theoretic passage.  At a nonzero point of
the imaginary axis, the imaginary-part identity gives $av=0$ for every
joint block component; the resulting scalar constant-coefficient ODE and
positive-measure unique continuation force that component to vanish.  At
zero, $\cA([u],v)=0$ first makes $v$ constant, while integration of
$\Delta u=av$ over the connected compact mapping torus gives
$v\int a=0$, hence $v=0$; then $u$ is harmonic, so it is constant and
vanishes in the homogeneous energy quotient.  Compact resolvent excludes
the rest of the imaginary axis.

The profile construction in Theorem~\ref{thm:C}\textup{(a)} is entirely
on the base circle, while the upper and lower power estimates in
Theorem~\ref{thm:C}\textup{(b)} use only the uniform block estimates and
the same localized resonant test.  The transverse eigenvalues, and hence
their dyadic nonlacunarity, are unchanged by the holonomy.  This proves all
claims in part~\textup{(a)}.
\end{proof}

\begin{proposition}[Uniform Floquet transfer of the logarithmic ladder]
\label{prop:uniform-floquet-ladder}
Let $d$ be either the exact cusp $a$ of Theorem~\ref{thm:E} or a fixed
stable measurable cusp $\tilde a$ as in
Theorem~\ref{thm:E}\textup{(c)}.  On
$I_s=(-\pi/R_s,\pi/R_s)$, let
\[
 \mathcal T^{d,\theta}_s=-\partial_\eta^2+iV_s^d,
 \qquad \mathcal D(\mathcal T^{d,\theta}_s)=H^2_\theta(I_s),
 \qquad \theta\in\R/2\pi\Z,
\]
where $V_s^d(\eta)=sR_s^2d(R_s\eta)-\ell_s/A$.  Then the compact-window
conclusions of Propositions~\ref{prop:log-scaling} and
\ref{prop:log-Ts}, and of Proposition~\ref{prop:stable-cusp-transfer} in
the stable case, hold uniformly in $\theta$.  More precisely, for every
fixed $K$:
\begin{enumerate}[label=\textup{(\roman*)},leftmargin=*]
\item if $\mathcal K\Subset\rho(\cL_1)$, then
\begin{equation}\label{eq:uniform-floquet-resolvent}
 \sup_{\theta\in\R/2\pi\Z}\sup_{\zeta\in\mathcal K}
 \| (\mathcal T_s^{d,\theta}-\zeta)^{-1}\|\le C_{\mathcal K}
\end{equation}
for all sufficiently large $s$;
\item for $0\le n\le K$, the contour $\Gamma_n$ eventually encloses
exactly one eigenvalue $\mu_{s,\theta}^{(n)}$ of
$\mathcal T_s^{d,\theta}$, algebraically simple, and
\begin{equation}\label{eq:uniform-floquet-rung-limit}
 \sup_\theta|\mu_{s,\theta}^{(n)}-\nu_n|\longrightarrow0.
\end{equation}
The paired sectors have the identical eigenvalue,
$\mu_{s,-\theta}^{(n)}=\mu_{s,\theta}^{(n)}$.  If normalized
eigenfunctions $\omega_{s,\theta}^{(n)}$ and
$\omega_{s,-\theta}^{(n)}$ are phased so that each has positive real
$L^2$ pairing with $\hat u_n$, then, after zero extension to $\R$,
\begin{equation}\label{eq:uniform-cross-nonisotropy}
 \sup_\theta\|\omega_{s,\theta}^{(n)}-\hat u_n\|_2
 +\sup_\theta\|\omega_{s,-\theta}^{(n)}-\hat u_n\|_2
 \longrightarrow0,
\end{equation}
and
\begin{equation}\label{eq:uniform-cross-overlap-limit}
 \inf_\theta
 \delta_\times\bigl(\omega_{s,\theta}^{(n)},
                    \omega_{s,-\theta}^{(n)}\bigr)
 \longrightarrow \delta(u_n)>0.
\end{equation}
In particular the paired Riesz projections are uniformly conditioned;
\item after undoing the dilation, the rung expansion, isolation radius,
tameness, and finite-window completeness of
Proposition~\ref{prop:log-Ts} hold uniformly in $\theta$.  The simple-mode
and finite-window pencil transfers hold on $H^2_\theta(0,2\pi)$ with the
same uniformity.  Thus every fixed model rectangle transfers, with its
full algebraic count and with no additional roots, to every Floquet
generator block.
\end{enumerate}
\end{proposition}

\begin{proof}
Every a priori estimate in Proposition~\ref{prop:log-scaling} is uniform
in $\theta$.  Sesquilinear integration by parts has no boundary
contribution on $H^2_\theta$, so the real-part, IMS, local $H^1$,
no-concentration, and exterior-tail estimates are unchanged.  If
\eqref{eq:uniform-floquet-resolvent} failed, sequences
$s_k\to\infty$, $\theta_k$, $\zeta_k\in\mathcal K$, and normalized
singular vectors in the $\theta_k$ sectors would, by the same compactness
and tightness argument, produce a nonzero eigenfunction of $\cL_1$ at a
point of $\mathcal K$.  This proves \textup{(i)}.

Compactly supported cutoffs of the model eigenfunctions vanish near the
endpoints of $I_s$ and hence belong to $H^2_\theta(I_s)$ for every
$\theta$.  The projection argument in
Lemma~\ref{lem:singular-sequence-criterion}, with the uniform contour
bound from \textup{(i)}, makes every model Riesz projection nonzero in
every phase, with a phase-independent threshold.

We prove that its rank is one.  Otherwise there are
$s_k\to\infty$ and phases $\theta_k$ for which one of the three
alternatives of Lemma~\ref{lem:root-space-alt} occurs.  In the
two-eigenvalue alternative, take a normalized $\theta_k$-sector
eigenfunction at the first eigenvalue and, using
Lemma~\ref{lem:paired-floquet}\textup{(ii)}, a normalized
$-\theta_k$-sector eigenfunction at the second.  Their cross-pairing is
zero by Lemma~\ref{lem:paired-floquet}\textup{(iv)}.  Compactness and
tightness make both converge, after subsequences and multiplication by
unimodular constants, to the same model mode $\hat u_n$, contradicting
$\int_\R u_n^2\ne0$.  Geometric multiplicity at least two is excluded as
in Proposition~\ref{prop:log-scaling}: two orthonormal eigenfunctions in
one sector converge strongly to scalar multiples of the same unit model
mode.  Finally, if
\[
 (\mathcal T_s^{d,\theta}-\mu)w=u,
\]
choose a normalized eigenfunction $v$ at the same eigenvalue in the
paired $-\theta$ sector.  Equation~\eqref{eq:cross-jordan-obstruction}
gives $\int uv=0$, whereas tight convergence of $u$ and $v$ gives the
nonzero limit $\int_\R u_n^2/\|u_n\|_2^2$.  No compactness estimate for
$w$ is used.  This proves algebraic simplicity uniformly in phase.

The same sequential contradiction proves
\eqref{eq:uniform-floquet-rung-limit} and the uniform eigenfunction
convergence in \eqref{eq:uniform-cross-nonisotropy}.  The tight-pairing
lemma then gives, uniformly in $\theta$,
\[
 \int_{I_s}\omega_{s,\theta}^{(n)}
                  \omega_{s,-\theta}^{(n)}
 \longrightarrow
 \int_\R\hat u_n^2,
\]
which is \eqref{eq:uniform-cross-overlap-limit}; the projection norm
statement follows from \eqref{eq:cross-projection-norm}.

The dilation preserves the Floquet phase.  Hence the affine rescaling,
isolation radii, and finite-window argument of
Proposition~\ref{prop:log-Ts} are unchanged, and their constants are
uniform by \textup{(i)}.  Propositions~\ref{prop:pencil-transfer} and
\ref{prop:cluster-transfer} also apply on $H^2_\theta$: their
integrations by parts are sesquilinear, $-\partial_x^2+1$ is invertible
on every Floquet domain, and their Gohberg--Sigal homotopies use a fixed
domain for each block.  All quantitative input bounds are uniform in
$\theta$, so the resulting thresholds are uniform.  The companion-chain
calculation of Lemma~\ref{lem:companion} is algebraic and remains valid on
$H^2_\theta\times H^1_\theta$; hence the Gohberg--Sigal multiplicity of
each pencil root equals its algebraic multiplicity in the corresponding
Floquet generator block.  This proves
\textup{(iii)} for the exact cusp.

For a stable measurable cusp,
Lemma~\ref{lem:stable-cusp-estimates} changes only the local potential
comparisons and exterior positivity datum; it is independent of the
boundary phase.  The preceding compactness, paired-sector rank
exclusions, rescaling, and pencil transfer therefore apply verbatim, with
the parameter dependence stated in
Proposition~\ref{prop:stable-cusp-transfer}.
\end{proof}

\begin{proof}[Proof of Theorem~\ref{thm:F}\textup{(b)}]
Pointwise conjugation on $L^2(Y)$ commutes with both $-\Delta_Y$ and the
real pullback operator $U_\varphi$.  Hence every holonomy phase in
$E_\lambda$ occurs with its conjugate: if
$U_\varphi\psi=e^{i\theta}\psi$, then
$U_\varphi\bar\psi=e^{-i\theta}\bar\psi$.  This supplies the paired
Floquet sectors used in Lemma~\ref{lem:paired-floquet}; no restriction on
the order of $\varphi$ is involved.

Choose the cut in the base circle away from the cusp.  In each joint
transverse/holonomy block, dilation about the cusp gives the operator
$\mathcal T_s^{a,\theta}$ of
Proposition~\ref{prop:uniform-floquet-ladder}.  That proposition gives,
uniformly in the holonomy phase, one algebraically simple eigenvalue per
model rung, compact-window resolvent bounds, isolation and tameness, and
the complete finite-window count.  Its final clause transfers these
rungs through the exact quadratic pencil on the same Floquet domain.
Undoing the scaling gives precisely the window
\eqref{eq:ladder-window} and the asymptotics
\eqref{eq:ladder-asymptotic}--\eqref{eq:ladder-real-imag}.  This proves the
blockwise assertion for the exact cusp.

The global upper resolvent bound is part~\textup{(a)}.  At every distinct
positive transverse frequency, choose any one joint holonomy eigenvector;
the principal rung in that scalar block has the asymptotics
\eqref{eq:log-resonance-asymptotic}--\eqref{eq:log-imag-shift}.  The set
of distinct transverse frequencies is the same as on the product and is
dyadically syndetic by \S\ref{subsec:notation}, so
Lemma~\ref{lem:envelope} gives the matching lower
regularized semigroup envelope.  This proves the global conclusion of
Theorem~\ref{thm:E}\textup{(b)} on $X_\varphi$.

For a stable measurable cusp, the stable part of
Proposition~\ref{prop:uniform-floquet-ladder} supplies the same blockwise
count, asymptotics, and principal rungs with unchanged model constants.
Part~\textup{(a)} supplies the corresponding global upper bound, so the
same envelope argument applies.

Finally, the joint decomposition~\eqref{eq:mapping-joint-basis} is an
orthogonal reducing decomposition.  The direct-sum Riesz-projection
argument of Lemma~\ref{lem:blocks} therefore adds algebraic
multiplicities over precisely those joint blocks that produce the same
spectral point.  In particular, conjugate nonreal holonomy phases produce
paired scalar blocks with identical rung locations, as follows also from
Remark~\ref{rem:paired-floquet-pencil}; each scalar contribution is
simple, while the full multiplicity is their sum.
\end{proof}

\begin{proof}[Proof of Proposition~\ref{prop:warped-sharpness}]
Write $f_\eps(x)=2+\eps\sin x$.  When $\eps=0$, the closed curve
$\{x=0\}$ is a unit-speed geodesic after a constant reparametrization,
and the damping vanishes identically on it.  The Gaussian-beam necessity
of geometric control therefore rules out exponential stability
\cite{Ralston1969,Lebeau1996}.

Suppose now that $0<|\eps|<1$.  A unit-speed geodesic that never meets
$\{a>0\}$ would have $x(t)\equiv0$.  Its $x$-equation is
\[
 \ddot x-f_\eps(x)f_\eps'(x)\dot y^{\,2}=0.
\]
At $x=0$ this becomes $-2\eps\dot y^{\,2}=0$, which is impossible:
unit speed and $\dot x=0$ imply $\dot y\ne0$.  Hence every forward
unit-speed geodesic meets the open set $\{a>0\}$.

This pointwise conclusion is uniform for the fixed nonzero $\eps$ under
consideration.  For each initial covector
$\rho\in S^*_{g_\eps}\T^2$, let $\gamma_{\rho,\eps}$ denote the
$g_\eps$-unit-speed geodesic issued from $\rho$, and choose $T_\rho>0$
such that
\[
 \int_0^{T_\rho}a(\gamma_{\rho,\eps}(t))\,dt>0.
\]
Continuous dependence of the $g_\eps$-geodesic flow gives a neighborhood
$U_\rho$ on which the same integral has a positive lower bound.  A finite
subcover of the compact unit cosphere $S^*_{g_\eps}\T^2$, followed by
taking the largest time and the smallest lower bound, yields
$T_\eps,c_\eps>0$ such that
\[
 \int_0^{T_\eps}a(\gamma_{\rho,\eps}(t))\,dt\ge c_\eps
 \qquad(\rho\in S^*_{g_\eps}\T^2).
\]
The geometric control theorem gives exponential stability
\cite{RT,BLR}.  No uniformity as $\eps\to0$ is asserted: the
geometric-control constants, and hence the exponential-stability constants
supplied by this argument, degenerate in that limit, consistently with
failure of exponential stability at $\eps=0$.  Finally, $a$ and the base coordinate have not changed,
so \eqref{eq:Theta} gives the same $\Theta_a$ for every $\eps$.  Since
$g_\eps\to g_0$ in $C^\infty$ as $\eps\to0$, the final assertion follows.
\end{proof}

\begin{proof}[Proof of Corollary~\ref{cor:warped-log-cusp}]
For all sufficiently small $r$, the interval of length $2r$ on which the
average of $a_A$ is least is centered at zero.  Indeed, if
$2r<x_0/4$, every such interval not contained in
$(-x_0/2,x_0/2)$ lies in $\{|x|\ge x_0/4\}$, where $a_A$ has a positive
lower bound.  The centered average tends to zero, so a minimizing
interval must eventually lie in $(-x_0/2,x_0/2)$; there the claim follows
because $a_A$ is even and increasing with $|x|$.  Therefore
\[
 \Theta_{a_A}(r)
 =\frac1r\int_0^r(\log(e/x))^{-A}\,dx
 =\int_0^1\bigl(\log(e/r)+\log(1/t)\bigr)^{-A}\,dt.
\]
After division by $(\log(e/r))^{-A}$, dominated convergence gives
\eqref{eq:warped-log-mass}.  For $\eps=0$ the metric is the exact product
of the base circle with the circle $(\T,4dy^2)$, and $a_A$ satisfies the
hypotheses of Theorem~\ref{thm:E}.  Its conclusions, including the stated
two-sided decay envelope, follow.  For $0<|\eps|<1$, exponential stability
is Proposition~\ref{prop:warped-sharpness}.
\end{proof}

Analytically, the $k$th Fourier block of the warped Laplacian is unitarily
equivalent to
\[
 -\partial_x^2+\frac{k^2}{f_\eps(x)^2}+V_\eps(x),
 \qquad
 V_\eps=\frac{f_\eps''}{2f_\eps}
          -\frac{(f_\eps')^2}{4f_\eps^2}.
\]
For $k\asymp s$ the $x$-dependent term $k^2f_\eps^{-2}$ has principal
size $s^2$; it is not a constant detuning parameter.  The proposition
therefore identifies a change in the trapping geometry, and proves only
that $\Theta_a$ by itself is no longer a complete invariant once flat
separation is lost.

\section{Conclusion}

Theorems~\ref{thm:B} and~\ref{thm:E} answer the two questions posed in
the introduction.  For arbitrary bounded measurable transverse damping,
the worst local average $\Theta_a$ is the coefficient-level invariant at
the exponential endpoint and throughout the critical slowly varying class.
The inverse direction is sparse: one resonant scalar channel at one
transverse frequency per octave recovers the full multiscale mass
condition.  Theorems~\ref{thm:A} and~\ref{thm:C} supply, respectively, the
endpoint and the abundance and power-scale counterexamples.

For the canonical logarithmic cusp, the resolvent scale is carried by true
spectrum rather than only by pseudospectrum.  Theorem~\ref{thm:D} gives the
exact model identity, and Theorem~\ref{thm:E} transfers every fixed finite
window with its complete blockwise algebraic count and two-term
coefficients.  The principal eigenvalues then meet the global resolvent
bound to give the sharp two-sided regularized decay envelope.  The same
leading description survives measurable lower-order relative
perturbations.

Theorem~\ref{thm:F} separates topology from geometry: isometric holonomy
preserves both main results, while an arbitrarily small smooth warping may
replace the logarithmic regime by exponential stabilization without
changing either the damping or $\Theta_a$.  The quantitative content is
the mass--resolvent--spectrum--decay chain displayed in
\eqref{eq:spine}.
\section*{Data availability}

No datasets were generated or analyzed in connection with this article.

\end{document}